\documentclass[11pt, letterpaper, francais]{smfart}
\usepackage{amsfonts, amssymb, amsmath, latexsym, enumerate, epsfig, color, mathrsfs, fancyhdr, supertabular, stmaryrd, smfthm, yhmath, mathtools, pifont, chngcntr, pinlabel, ifthen}
\usepackage[all]{xy}
\usepackage{tikz-cd}
\usepackage[french,refpage]{nomencl}

\PassOptionsToPackage{normalem}{ulem}
\usepackage{ulem}

\providecolor{added}{rgb}{0,0,1}
\providecolor{deleted}{rgb}{1,0,0}
\newcommand{\added}[1]{{\color{added}{}#1}}
\newcommand{\deleted}[1]{\ifmmode\text{\color{deleted}\sout{\ensuremath{#1}}}\else{\color{deleted}\sout{#1}}\fi}

\usepackage[french,english]{babel}

\usepackage[T1]{fontenc}
\usepackage[mac]{inputenc}

\usepackage{hyperref}

\theoremstyle{plain} 
\newtheorem{theointro}{Th\'eor\`eme}

\newtheorem{lemmintro}[theointro]{Lemme}

\theoremstyle{definition} 
\newenvironment{nota}{\begin{enonce}[definition]{Notation}}{\end{enonce}}

\DeclareFontFamily{U}{mathc}{}
\DeclareFontShape{U}{mathc}{m}{it}%
{<->s*[1.03] mathc10}{}
\DeclareMathAlphabet{\mathcal}{U}{mathc}{m}{it}

\newcommand{\N}{\ensuremath{\mathbf{N}}}
\newcommand{\Z}{\ensuremath{\mathbf{Z}}}
\newcommand{\Zud}{\ensuremath{\mathbf{Z}[\frac12]}}
\newcommand{\Q}{\ensuremath{\mathbf{Q}}}
\newcommand{\Qbar}{\ensuremath{\overline{\mathbf{Q}}}}
\newcommand{\R}{\ensuremath{\mathbf{R}}}
\newcommand{\C}{\ensuremath{\mathbf{C}}}

\newcommand{\A}{\ensuremath{\mathbf{A}}}

\renewcommand{\P}{\ensuremath{\mathbf{P}}}

\newcommand{\G}{\ensuremath{\mathbf{G}}}

\newcommand{\E}[2]{\ensuremath{\mathbf{A}^{#1,\mathrm{an}}_{#2}}}
\newcommand{\EP}[2]{\ensuremath{\mathbf{P}^{#1,\mathrm{an}}_{#2}}}

\newcommand{\wti}[1]{\ensuremath{\widetilde{#1}}}

\newcommand{\too}{\longrightarrow}
\newcommand{\simtoo}{\overset{\sim}{\longrightarrow}}
\newcommand{\mapstoo}{\longmapsto}

\DeclareMathOperator{\Spec}{Spec}

\DeclareMathOperator{\Gal}{Gal}

\newcommand{\Rat}{\textrm{Rat}}

\DeclareMathOperator*{\colim}{colim}

\newcommand*\diff{\mathop{}\!\mathrm{d}}

\DeclareMathOperator{\PGL}{PGL}
\DeclareMathOperator{\SL}{SL}

\newcommand{\pr}{\mathrm{pr}}
 
\DeclareMathOperator{\smax}{smax}

\newcommand{\Seg}{\textrm{Seg}}

\newcommand{\cA}{\mathcal{A}}

\newcommand{\cC}{\mathcal{C}}
\newcommand{\cD}{\mathcal{D}}
\newcommand{\cE}{\mathcal{E}}

\newcommand{\cH}{\mathcal{H}}

\newcommand{\cM}{\mathcal{M}}
\newcommand{\cN}{\mathcal{N}}
\newcommand{\cO}{\mathcal{O}}
\newcommand{\cP}{\mathcal{P}}

\newcommand{\cU}{\mathcal{U}}

\newcommand{\cX}{\mathcal{X}}
\newcommand{\cY}{\mathcal{Y}}
\newcommand{\cZ}{\mathcal{Z}}

\newcommand{\cn}[2]{\ensuremath{\llbracket{#1},{#2}\rrbracket}}

\DeclarePairedDelimiter{\intoo}{]}{[}
\DeclarePairedDelimiter{\intof}{]}{]}
\DeclarePairedDelimiter{\intfo}{[}{[}
\DeclarePairedDelimiter{\intff}{[}{]}

\newcommand{\la}{\ensuremath{\langle}}
\newcommand{\ra}{\ensuremath{\rangle}}

\newcommand{\eps}{\ensuremath{\varepsilon}}

\newcommand{\disp}{\displaystyle}

\DeclarePairedDelimiter\abs{\lvert}{\rvert}
\newcommand\labs[1]{\log(\abs{#1})}
\DeclarePairedDelimiter\bigabs{\big\lvert}{\big\rvert}
\newcommand\biglabs[1]{\log\big(\bigabs{#1}\big)}
\DeclarePairedDelimiter\Bigabs{\Big\lvert}{\Big\rvert}

\DeclarePairedDelimiter\norm{\lVert}{\rVert}

\newcommand\wc{{\mkern 2mu\cdot\mkern 2mu}}
\newcommand\va{\abs{\wc}}
\newcommand\nm{\norm{\wc}}
\newcommand\cf{\textit{cf}.}
\newcommand\an{\mathrm{an}}

\newcommand{\fonction}[5]{\begin{array}{cccc}
#1 \colon &#2&\too&#3\\
&#4&\mapstoo&#5
\end{array}}

\countdef \commentaires = 1
\newcommand{\JP}[1]{\ifnum\commentaires = 1{\color{magenta}{#1}}\fi}

\def \spectriv(#1,#2,#3,#4,#5,#6){
\foreach \x in {0,...,#6}
\draw[line width=.001pt] ({#1},{#2}) -- ({#1+#5*cos((\x)/(#6)*(360/#4)+#3)},{#2+#5*sin((\x)/(#6)*(360/#4)+#3)}) ;}

\def \spectrivcentre(#1,#2,#3,#4,#5,#6){
\foreach \x in {-#6,...,#6}
\draw[line width=.001pt] ({#1},{#2}) -- ({#1+#5*cos((\x)/(#6)*(360/#4)+#3)},{#2+#5*sin((\x)/(#6)*(360/#4)+#3)}) ;}

\def \spectrivcentreenplus(#1,#2,#3,#4,#5,#6,#7,#8,#9){
\foreach \y in {-#5,...,#5}
\spectrivcentre(#1+#6*cos((\y)/(#5)*(360/#4)+#3),#2+#6*sin((\y)/(#5)*(360/#4)+#3),(\y)/(#5)*(360/#4)+#3,#7,#8,#9) 
;}
\begin{document}

\title[Dynamique analytique sur~$\Z$. II]{Dynamique analytique sur~$\Z$. II~: \'Ecart uniforme entre Latt\`es et conjecture de Bogomolov-Fu-Tschinkel.}
\alttitle{Analytic dynamics over~$\Z$. II: Uniform gap between Latt\`es and conjecture of Bogomolov-Fu-Tschinkel.}

\author{J\'er\^ome Poineau}
\address{Normandie Univ., UNICAEN, CNRS, Laboratoire de math\'ematiques Nicolas Oresme, 14000 Caen, France}
\email{\href{mailto:jerome.poineau@unicaen.fr}{jerome.poineau@unicaen.fr}}
\urladdr{\url{https://poineau.users.lmno.cnrs.fr/}}

\date{\today}

\makeatletter
\@namedef{subjclassname}{Classification math\'ematique par sujets \textup{(2020)}}
\makeatother

\altkeywords{Berkovich spaces over~$\Z$, Latt\`es morphisms, elliptic curves, torsion points, mutual energy, Arakelov theory}

\begin{abstract}
Nous montrons que l'\'energie mutuelle (ou le produit d'intersection au sens de la th\'eorie d'Arakelov) de deux syst\`emes dynamiques associ\'es \`a des morphismes de Latt\`es sur~$\mathbf{\bar Q}$ est uniform\'ement minor\'ee et en d\'eduisons une preuve d'une conjecture de Bogomolov-Fu-Tschinkel~: le nombre d'images communes de points de torsion de deux courbes elliptiques sur~$\mathbf{C}$ non isomorphes par un morphisme standard vers la droite projective est uniform\'ement born\'e. 

La d\'emonstration repose de fa\c con essentielle sur la th\'eorie des espaces de Berkovich sur~$\mathbf{Z}$ et sur un argument original permettant d'obtenir une estimation globale \`a partir d'une estimation centrale (au-dessus d'un corps trivialement valu\'e). 
\end{abstract}

\begin{altabstract}
We prove that the mutual energy (or the intersection product in the sense of Arakelov theory) of two dynamical systems associated to Latt\`es morphisms over~$\mathbf{\bar Q}$ is uniformly bounded below and deduce a proof of a conjecture of Bogomolov-Fu-Tschinkel: the number of common images of torsion points of two non-isomorphic elliptic curves over~$\mathbf{C}$ by a standard morphism to the projective line is uniformly bounded.

The proof crucially relies on the theory of Berkovich spaces over~$\mathbf{Z}$ and on an original argument allowing to obtain a global estimate from a central estimate (over a trivially valued field).
\end{altabstract}

\maketitle

\tableofcontents

\section{Introduction}

Nous poursuivons ici l'\'etude des syst\`emes dynamiques dans le cadre des espaces de Berkovich sur~$\Z$ initi\'ee dans~\cite{DynamiqueI} et en tirons quelques applications arithm\'etiques. Nous proposons en particulier une preuve du r\'esultat suivant, qui r\'esout par l'affirmative une conjecture due \`a F.~Bogomolov, H.~Fu et Yu.~Tschinkel (\cf~\cite[conjectures~2 et~12]{BFT}).

\begin{theointro}[\textmd{\emph{infra} corollaire~\ref{cor:BFTfinal}}]\label{th:BFT}
Il existe $M \in \R_{>0}$ telle que, pour toutes courbes elliptiques~$E_{a}$ et~$E_{b}$ sur~$\C$ et tous rev\^etements doubles $\pi_{a} \colon E_{a}\to \P^1_{\C}$ et $\pi_{b} \colon E_{b}\to \P^1_{\C}$ tels que $\pi_{a}(E_{a}[2]) \ne \pi_{b}(E_{b}[2])$, on ait
\[ \sharp \big( \pi_{a}(E_{a}[\infty]) \cap \pi_{b}(E_{b}[\infty])\big) \le M.  \]
\end{theointro}

Notre strat\'egie s'inspire directement de celle mise en \oe uvre dans~\cite{DKY}, o\`u L.~DeMarco, H.~Krieger et H.~Ye d\'emontrent le th\'eor\`eme~\ref{th:BFT} pour des courbes elliptiques~$E_{a}$ et~$E_{b}$, ou plut\^ot des paires $(E_{a},\pi_{a})$ et $(E_{b},\pi_{b})$, appartenant \`a la famille de Legendre. Il s'agit d'une approche de nature dynamique. 

Rappelons qu'\`a toute paire~$(E,\pi)$ comme ci-dessus, on peut associer un endomorphisme~$L$ de~$\P^1_{\C}$, dit de Latt\`es, caract\'eris\'e par le diagramme commutatif suivant~:
\[\begin{tikzcd}
E \ar[r, "{[}2{]}"] \ar[d, "\pi"]& E \ar[d, "\pi"]\\
\P^1_{\C} \ar[r, "L"] & \P^1_{\C}
\end{tikzcd}.\]
L'ensemble $\pi(E[\infty])$ co\"incide alors avec l'ensemble des points pr\'ep\'eriodiques de~$L$. Il s'agit donc de comparer les syst\`emes dynamiques associ\'es \`a deux morphismes de Latt\`es distincts sur~$\P^1$. 

Si l'on remplace le corps des nombres complexes~$\C$ par le corps des nombres alg\'ebriques~$\Qbar$, on dispose d'outils arithm\'etiques qui permettent d'aborder le probl\`eme. En particulier, la th\'eorie d'Arakelov permet d'associer \`a deux syst\`emes dynamiques~$\varphi_{a}$ et~$\varphi_{b}$ sur~$\P^1$ un nombre r\'eel positif $\la \varphi_{a},\varphi_{b} \ra$\footnote{Il s'agit pr\'ecis\'ement du produit d'intersection des premi\`eres classes de Chern des fibr\'es m\'etris\'es associ\'es \`a~$\varphi_{a}$ et~$\varphi_{b}$.}, qui se comporte comme une distance (ou plut\^ot son carr\'e). Dans le cas des morphismes de Latt\`es associ\'es \`a la multiplication par~2, cette distance est uniform\'ement minor\'ee. 

\begin{theointro}[\textmd{\emph{infra} th\'eor\`emes~\ref{th:hEm0} et~\ref{th:minorationenergie}}]\label{th:ecartLattes}
Il existe $m_{0} \in \R_{>0}$ telle que, pour tous morphismes de Latt\`es $L_{a}\ne L_{b}$ sur~$\Qbar$ comme ci-dessus, on ait 
\[ \la L_{a},L_{b} \ra \ge m_{0}.\]
\end{theointro}

Ce r\'esultat propose une explication conceptuelle \`a la conjecture de Bogomolov-Fu-Tschinkel. C'est un ingr\'edient crucial dans notre preuve du th\'eor\`eme~\ref{th:BFT}. 

La d\'emonstration du th\'eor\`eme~\ref{th:ecartLattes} fait intervenir une minoration globale par une hauteur, que nous d\'eduisons d'une minoration \og centrale \fg, au-dessus d'un corps trivialement valu\'e.

\subsection{Ingr\'edients}\label{sec:ingredients}

Le th\'eor\`eme~\ref{th:ecartLattes} porte sur des paires de morphismes de Latt\`es associ\'es \`a la multiplication par~2. Pour le d\'emontrer, nous nous pla\c cons sur un espace de modules convenable~$\cY$. D\'ecrivons-le ici en quelques mots. 

Les applications de Latt\`es~$L_{a}$ et~$L_{b}$ (et m\^eme les couples $(E_{a},\pi_{a})$ et $(E_{b},\pi_{b})$, \`a isomorphisme pr\`es), sont d\'etermin\'es par la donn\'ee des ensembles images de la 2-torsion $\pi_{a}(E_{a}[2])$ et $\pi_{b}(E_{b}[2])$. Chacun d'eux poss\`ede exactement quatre points dans~$\P^1(\Qbar)$, disons $a_{1},a_{2},a_{3},a_{4}$ pour $\pi_{a}(E_{a}[2])$ et $b_{1},b_{2},b_{3},b_{4}$ pour $\pi_{b}(E_{b}[2])$. L'espace de modules correspondant est donc, \textit{a priori}, de dimension~8.

\`A l'aide d'homographies, on peut se ramener \`a un espace de dimension $5 = 8-3$. Le proc\'ed\'e classique consisterait \`a envoyer trois des points pr\'ec\'edents sur~$0,1,\infty$. Afin de ne pas briser la sym\'etrie, nous pr\'ef\'erons ne fixer que deux des points ($a_{4}$ en $\infty$ et $b_{4}$ en 0) et garder le degr\'e de libert\'e suppl\'ementaire pour travailler \`a homoth\'etie pr\`es. En d'autres termes, l'espace de module~$\cY$ que nous consid\'erons est un sous-espace de~$\P^5_{\Zud}$ avec coordonn\'ees homog\`enes $[a_{1}\mathbin: a_{2}\mathbin: a_{3} \mathbin: b_{1} \mathbin: b_{2}\mathbin: b_{3}]$. (Nous excluons la caract\'eristique~2 pour des raisons li\'ees \`a la d\'efinition des morphismes de Latt\`es consid\'er\'es.)

\bigbreak

Un outil fondamental utilis\'e dans ce texte est l'\'energie mutuelle $\la \varphi_{a},\varphi_{b}\ra$. Nous l'introduisons ici rapidement, en suivant la pr\'esentation qu'en donnent Ch.~Favre et J.~Rivera--Letelier dans~\cite{FRLEquidistribution}. 

Pla\c cons-nous tout d'abord dans le cadre complexe. Pour deux mesures~$\mu_{a}$ et~$\mu_{b}$ suffisamment r\'eguli\`eres sur~$\P^1(\C)$, on pose 
\[ \la \mu_{a},\mu_{b} \ra := - \int_{\C^2 \setminus \textrm{Diag}} \log(\abs{x-y}) \diff (\mu_{a}-\mu_{b})(x) \otimes \diff (\mu_{a}-\mu_{b})(y).\]

On peut d\'efinir de m\^eme une \'energie mutuelle sur tout corps ultram\'etrique complet~$k$ en rempla\c cant~$\P^1(\C)$ par la droite projective~$\EP{1}{k}$ au sens de Berkovich et en g\'en\'eralisant convenablement la valeur absolue $\abs{x-y}$.

Du c\^ot\'e global, \'etant donn\'e un corps de nombres~$K$, avec un ensemble de places not\'e~$M_{K}$, et deux familles de mesures suffisamment r\'eguli\`eres $\mu_{a}=(\mu_{a,v})_{v\in M_{K}}$ et $\mu_{b}=(\mu_{b,v})_{v\in M_{K}}$, on pose 
\[ \la \mu_{a},\mu_{b} \ra = \sum_{v\in M_{K}} N_{v} \, \la \mu_{a,v},\mu_{b,v} \ra,\]
o\`u les $N_{v}$ sont des constantes de normalisation.

Dans le cadre qui nous int\'eresse, des familles de mesures naturelles sont fournies par la th\'eorie des syst\`emes dynamiques. Soient $L_{a}$ et~$L_{b}$ des morphismes de Latt\`es sur~$\P^1_{K}$. Pour $v\in M_{K}$, notons $\mu_{a,v}$ (resp. $\mu_{b,v}$) la mesure d'\'equilibre associ\'ee \`a~$L_{a}$ (resp. $L_{b}$) sur $\EP{1}{K_{v}}$ ou $\P^1(\C)$, en fonction du type, fini ou infini, de la place consid\'er\'ee. Rappelons que ces mesures sont essentiellement caract\'eris\'ees par les propri\'et\'es 
\[{L_{a}}^* \mu_{a,v} = 4 \,\mu_{a,v} \textrm{ et } {L_{b}}^* \mu_{b,v} = 4 \,\mu_{b,v},\]
les morphismes $L_{a}$ et~$L_{b}$ \'etant ici de degr\'e~4. On pose alors 
\[ \la L_{a},L_{b}\ra := \la (\mu_{a,v})_{v\in M_{K}},(\mu_{b,v})_{v\in M_{K}} \ra.\]

Mentionnons que l'\'energie mutuelle permet de retrouver certaines hauteurs classiques. En effet, soit $F$ un ensemble fini de~$\bar K$ invariant par~$\Gal(\bar K/K)$. Pour $v\in M_{K}$, notons $[F]_{v}$ la mesure sur~$\EP{1}{K_{v}}$ ou $\P^1(\C)$ provenant de la mesure de probabilit\'e \'equidistribu\'ee sur les images des points de~$F$ dans~$\P^1(\bar K_{v})$ ou~$\P^1(\C)$. On a alors
\[ \la L_{a}, [F] \ra := \la (\mu_{a,v})_{v\in M_{K}},([F]_{v})_{v\in M_{K}} \ra = h_{L_{a}}(F),\]
o\`u le membre de droite d\'esigne la hauteur dynamique de~$F$ associ\'ee \`a~$L_{a}$. En particulier, on peut retrouver les points pr\'ep\'eriodiques de~$L_{a}$ comme le lieu d'annulation de $\la L_{a}, [\wc]\ra$.

\subsection{Strat\'egie de la preuve du th\'eor\`eme~\ref{th:ecartLattes}}

Le th\'eor\`eme~\ref{th:ecartLattes} constitue, selon nous, le r\'esultat le plus important de ce texte. La strat\'egie que nous mettons en \oe uvre pour le d\'emontrer pr\'esente un int\'er\^et propre et nous l'exposons ici en d\'etail. Le principe g\'en\'eral s'en exprime simplement~: les invariants arithm\'etiques globaux poss\`edent des analogues, que nous qualifierons de \emph{centraux}, d\'efinis en valuation triviale et o\`u se refl\`ete leur comportement.

Nous appliquons ici ce principe \`a l'\'energie globale $\la L_{a},L_{b} \ra$. Le point de d\'epart est la  d\'ecomposition en somme de facteurs locaux~:
\[ \la L_{a},L_{b} \ra = \sum_{v\in M_{K}} N_{v} \, \la L_{a,v},L_{b,v} \ra,\]
o\`u $K$ est un corps de nombres sur lequel~$L_{a}$ et~$L_{b}$ sont d\'efinis. 

La strat\'egie se d\'ecompose en 3 \'etapes. La premi\`ere et la derni\`ere sont g\'en\'erales et pourraient s'appliquer dans bien d'autres situations. La deuxi\`eme, en revanche, est sp\'ecifique au probl\`eme consid\'er\'e.

\bigbreak

\noindent $\blacktriangleright$ \textbf{\'Etape~1 : du discret au continu.}

La premi\`ere \'etape consiste \`a compl\'eter la famille discr\`ete $(\la L_{a},L_{b} \ra_{v})_{v\in M_{K}}$ en une famille continue. Pour ce faire, nous aurons recours \`a la th\'eorie des espaces de Berkovich sur les anneaux de Banach et, plus pr\'ecis\'ement ici, sur l'anneau des entiers~$\cO_{K}$ du corps de nombres~$K$. Elle a \'et\'e esquiss\'ee par V.~Berkovich dans~\cite{rouge}, d\'evelopp\'ee plus en d\'etails par l'auteur dans~\cite{A1Z,EtudeLocale} en ce qui concerne les aspects alg\'ebriques, puis par l'auteur en collaboration avec Th.~Lemanissier dans~\cite{CTC} en ce qui concerne les aspects topologiques et cohomologiques. Des d\'etails figurent dans \cite[section~2]{DynamiqueI}.

Rappelons que le spectre analytique~$\cM(\cO_{K})$ est d\'efini comme l'ensemble des semi-normes multiplicatives sur~$\cO_{K}$. Il contient donc les valeurs absolues normalis\'ees~$\va_{v}$ pour $v\in M_{K}$, ainsi que beaucoup d'autres points, par exemple les puissances de ces valeurs absolues ou encore la valeur absolue triviale~$\va_{0}$ (envoyant tout \'el\'ement non nul sur~1). Du point de vue topologique, l'espace~$\cM(\cO_{K})$ poss\`ede de bonnes propri\'et\'es~: s\'eparation, compacit\'e locale et globale, connexit\'e par arcs locale et globale. La topologie de~$\cM(\cO_{K})$ est de nature ad\'elique et, bien qu'\'etant ferm\'e, le point central~$\va_{0}$ se comporte comme une sorte de point g\'en\'erique, au sens o\`u ses voisinages sont tr\`es gros (\cf~figure~\ref{fig:MOK}).

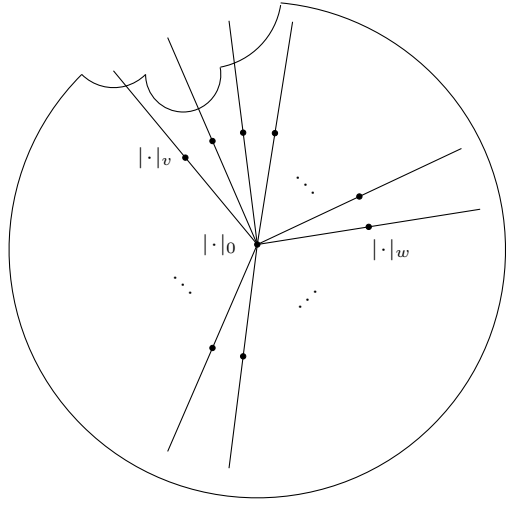
\begin{figure}[!h]
\centering
\begin{tikzpicture}
\draw (0,0) -- ({3*cos(0.72*pi r)} ,{3*sin(0.72*pi r)});
\draw (0,0) -- ({3*cos(0.63*pi r)} ,{3*sin(0.63*pi r)});
\draw (0,0) -- ({3*cos(0.54*pi r)} ,{3*sin(0.54*pi r)});
\draw (0,0) -- ({3*cos(0.45*pi r)} ,{3*sin(0.45*pi r)});
\draw (0,0) -- ({3*cos(0.14*pi r)} ,{3*sin(0.14*pi r)});
\draw (0,0) -- ({3*cos(0.05*pi r)} ,{3*sin(0.05*pi r)});
\draw (0,0) -- ({3*cos(-0.63*pi r)} ,{3*sin(-0.63*pi r)});
\draw (0,0) -- ({3*cos(-0.54*pi r)} ,{3*sin(-0.54*pi r)});

\draw[fill=black] (0,0) circle (1pt);
\draw[fill=black] ({1.5*cos(0.72*pi r)} ,{1.5*sin(0.72*pi r)}) circle (1pt);
\draw[fill=black] ({1.5*cos(0.63*pi r)} ,{1.5*sin(0.63*pi r)}) circle (1pt);
\draw[fill=black] ({1.5*cos(0.54*pi r)} ,{1.5*sin(0.54*pi r)}) circle (1pt);
\draw[fill=black] ({1.5*cos(0.45*pi r)} ,{1.5*sin(0.45*pi r)}) circle (1pt);
\draw[fill=black] ({1.5*cos(0.14*pi r)} ,{1.5*sin(0.14*pi r)}) circle (1pt);
\draw[fill=black] ({1.5*cos(0.05*pi r)} ,{1.5*sin(0.05*pi r)}) circle (1pt);
\draw[fill=black] ({1.5*cos(-0.54*pi r)} ,{1.5*sin(-0.54*pi r)}) circle (1pt);
\draw[fill=black] ({1.5*cos(-0.63*pi r)} ,{1.5*sin(-0.63*pi r)}) circle (1pt);

\draw ({3.3*cos(0.75*pi r)},{3.2*sin(0.75*pi r)}) arc ({0.75*pi r}:{2.47*pi r}:3.3);
\draw ({3.3*cos(0.75*pi r)},{3.2*sin(0.75*pi r)}) arc ({-0.75*pi r}:{-0.25*pi r}:0.6);
\draw ({3.3*cos(0.75*pi r)+sqrt(2)*0.6},{3.2*sin(0.75*pi r)}) arc ({1*pi r}:{2.06*pi r}:.5);
\draw ({3.3*cos(2.47*pi r)},{3.2*sin(2.47*pi r)}) arc ({-0.05*pi r}:{-0.44*pi r}:1);

\begin{scriptsize}
\draw (-.5,0) node {$\va_{0}$};
\draw({1.5*cos(0.72*pi r)-.4} ,{1.5*sin(0.72*pi r)}) node {$\va_{v}$};
\draw ({.9*cos(0.45*pi r)+.4} ,{.9*sin(0.45*pi r)}) node {$\cdot$};
\draw ({.9*cos(0.45*pi r)+.5} ,{.9*sin(0.45*pi r)-.1}) node {$\cdot$};
\draw ({.9*cos(0.45*pi r)+.6} ,{.9*sin(0.45*pi r)-.2}) node {$\cdot$};
\draw ({1.5*cos(0.05*pi r)+.3} ,{1.5*sin(0.05*pi r)-.3}) node {$\va_{w}$};
\draw ({1*cos(-0.22*pi r)},{1*sin(-0.22*pi r)}) node {$\cdot$};
\draw ({1*cos(-0.22*pi r)-.1},{1*sin(-0.22*pi r)-.1}) node {$\cdot$};
\draw ({1*cos(-0.22*pi r)-.2} ,{1*sin(-0.22*pi r)-.2}) node {$\cdot$};
\draw ({1.1*cos(-0.8*pi r)} ,{1.1*sin(-0.8*pi r)}) node {$\cdot$};
\draw ({1.1*cos(-0.8*pi r)-.1} ,{1.1*sin(-0.8*pi r)+.1}) node {$\cdot$};
\draw ({1.1*cos(-0.8*pi r)-.2} ,{1.1*sin(-0.8*pi r)+.2}) node {$\cdot$};
\end{scriptsize}
\end{tikzpicture}
\caption{Un voisinage de~$\va_{0}$ dans $\cM(\cO_{K})$.}\label{fig:MOK}
\end{figure}

La d\'efinition d'\'energie locale propos\'ee par Ch.~Favre et J.~Rivera--Letelier est remarquablement g\'en\'erale et fait sens sur tout corps valu\'e complet. Pour tout point $x$ de $\cM(\cO_{K})$ (\`a l'exception de ceux de caract\'eristique~2 pour des raisons sp\'ecifiques aux morphismes de Latt\`es consid\'er\'es), on peut donc d\'efinir la quantit\'e $\la L_{a,x},L_{b,x} \ra$ comme l'\'energie mutuelle locale des endomorphismes induits par~$L_{a}$ et~$L_{b}$ sur la droite projective de Berkovich sur le corps~$\cH(x)$ associ\'e \`a~$x$. (Si la semi-norme~$\va_{x}$ associ\'ee \`a~$x$ est une valeur absolue, $\cH(x)$ n'est autre que le compl\'ete de~$K$ par rapport \`a~$\va_{x}$.) 

D'apr\`es \cite[corollaire~D]{DynamiqueI}, la fonction 
\[ x\in \cM(\cO_{K}) \setminus V(2) \mapstoo \la L_{a,x},L_{b,x} \ra \in \R_{\ge 0}\]
est continue. On peut donc se faire une id\'ee du comportement de la famille $(\la L_{a,v},L_{b,v} \ra)_{v\in M_{K}}$ en \'etudiant le comportement de la quantit\'e centrale correspondante~$\la L_{a,0},L_{b,0} \ra$ au-dessus de~$\va_{0}$. 

Dans la suite de ce texte, nous \'etudions la famille $(\la L_{a,v},L_{b,v} \ra)_{v\in M_{K}}$ en faisant varier non seulement~$v$, comme ci-dessus, mais \'egalement le couple $(L_{a},L_{b})$. \`A cet effet, nous travaillerons sur l'espace de modules~$\cY$ (d\'efini sur $\Z[\frac12]$) 
mentionn\'e au d\'ebut de la section~\ref{sec:ingredients}. Comme on s'y attend, il poss\`ede un analytifi\'e~$Y$ 
(sur $\cM(\Z) \setminus V(2)$). 
La propri\'et\'e de continuit\'e de l'\'energie \'enonc\'ee plus haut reste valable dans ce cadre. 

\begin{theointro}[\textmd{\emph{infra} th\'eor\`eme~\ref{th:Econtinue}}]\label{th:continuintro}
La fonction
\[ (a,b,x) \in Y \mapstoo  \la L_{a,x},L_{b,x} \ra \in \R_{\ge 0}\]
est continue.
\end{theointro}

\bigbreak

\noindent $\blacktriangleright$ \textbf{\'Etape~2 : minoration centrale.}

Dans la deuxi\`eme \'etape, nous minorons l'\'energie mutuelle $\la L_{a,0},L_{b,0} \ra$ de deux applications de Latt\`es~$L_{a,0}$ et~$L_{b,0}$ centrales, c'est-\`a-dire correspondant \`a un point de~$Y$ au-dessus de~$\va_{0}$. Elles sont d\'efinies sur un certain corps valu\'e~$k$, qui est une extension de $(\Q,\va_{0})$. 

Dans ce cadre, plusieurs simplifications se produisent, et la th\'eorie peut \^etre rendue explicite. On sait par exemple d\'ecrire les mesures d'\'equilibre associ\'ees \`a~$L_{a,0}$ et~$L_{b,0}$: elles sont proportionnelles aux mesures de Lebesgue sur deux segments~$I_{a,0}$ et~$I_{b,0}$ contenus dans~$\EP{1}{k}$. La th\'eorie de Favre-Rivera--Letelier se pr\^ete \'egalement tr\`es bien \`a des calculs concrets. Nous obtenons finalement les expressions g\'en\'erales suivantes (\cf~figure~\ref{fig:IaIb} pour les notations). 

\begin{theointro}[\textmd{\emph{infra} th\'eor\`emes~\ref{thm:Eintersectionvide} et~\ref{thm:Ecasgeneral}}]\label{th:La0Lb0}
Si $I_{a,0}\cap I_{b,0} = \emptyset$, on a 
\begin{equation*}\label{eq:I1I2disjoints}
 \la L_{a,0},L_{b,0} \ra = \frac16\, \ell_{a} + \frac16\, \ell_{b} +\frac12 \, d_{ab}  - \frac12 \frac{\ell'_{a} \ell''_{a}}{\ell_{a}} -\frac12 \frac{\ell'_{b} \ell''_{b}}{\ell_{b}}.
 \end{equation*}
Si $I_{a,0}\cap I_{b,0} \ne \emptyset$, on a 
\begin{equation*}\label{eq:I1I2intersection}
 \la L_{a,0},L_{b,0} \ra = \frac16 \ell_{a} + \frac16\ell_{b}-\ell_{ab} +\frac16 \frac{\ell_{ab}^3}{\ell_{a}\ell_{b}} -\frac12 \frac{\ell'_{a}\ell''_{a}}{\ell_{a}} - \frac12 \frac{\ell'_{b}\ell''_{b}}{\ell_{b}} - \frac12 \frac{(\ell'_{a}\ell'_{b}+\ell''_{a}\ell''_{b})\ell_{ab}}{\ell_{a}\ell_{b}}.
\end{equation*}
\end{theointro}

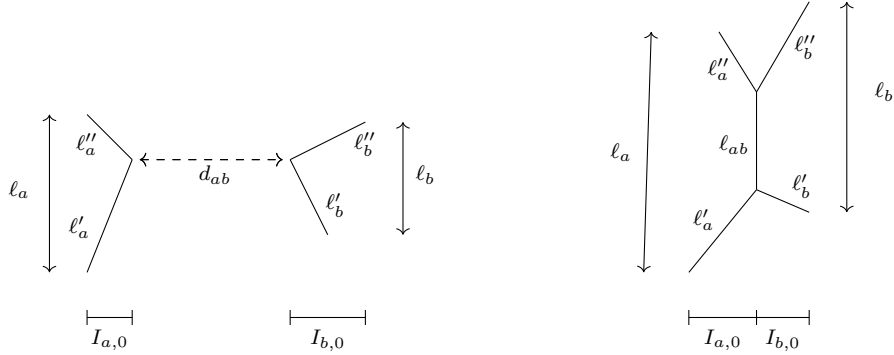
\begin{figure}[!h]
\centering
\begin{tikzpicture}
\draw  (-4.5-8,2.2)-- (-3.9-8,1.6);
\draw  (-3.9-8,1.6)-- (-4.5-8,0.1);
\draw [<->] (-5-8,2.2) -- (-5-8,0.1);
\draw (-1.8-8,1.6)-- (-0.8-8,2.1);
\draw  (-1.8-8,1.6)-- (-1.3-8,0.6);
\draw [<->] (-0.3-8,2.1) -- (-0.3-8,0.6);
\draw [<->,dashed] (-3.8-8,1.6) -- (-1.9-8,1.6);
\draw [<->,dashed] (-3.8-8,1.6) -- (-1.9-8,1.6);
\draw (-4.5-8,-0.5) -- (-3.9-8,-0.5);
\draw (-4.5-8,-0.4) -- (-4.5-8,-0.6);
\draw (-3.9-8,-0.4) -- (-3.9-8,-0.6);
\draw (-1.8-8,-0.5) -- (-0.8-8,-0.5);
\draw (-1.8-8,-0.4) -- (-1.8-8,-0.6);
\draw (-0.8-8,-0.4) -- (-0.8-8,-0.6);

\begin{scriptsize}
\draw (-5.4-8,1.2) node {$\ell_a$};
\draw (-8,1.4) node {$\ell_b$};
\draw (-4.5-8,1.8) node {$\ell''_a$};
\draw (-4.6-8,0.7) node {$\ell'_a$};
\draw (-0.8-8,1.8) node {$\ell''_b$};
\draw (-1.2-8,1) node {$\ell'_b$};
\draw (-2.8-8,1.4) node {$d_{ab}$};
\draw (-4.2-8,-0.8) node {$I_{a,0}$};
\draw (-1.3-8,-0.8) node {$I_{b,0}$};
\end{scriptsize}

\draw  (-3.6,2.5)-- (-3.6,1.2);
\draw  (-3.6,1.2)-- (-4.5,0.1);
\draw(-3.6,2.5)-- (-4.1,3.3);
\draw [<->] (-5,3.3) -- (-5.1,0.1);
\draw  (-3.6,2.5)-- (-2.9,3.7);
\draw  (-3.6,1.2)-- (-2.9,0.9);
\draw [<->] (-2.4,3.7) -- (-2.4,0.9);
\draw (-4.5,-0.5) -- (-3.6,-0.5);
\draw (-4.5,-0.4) -- (-4.5,-0.6);
\draw (-3.6,-0.4) -- (-3.6,-0.6);
\draw (-3.6,-0.5) -- (-2.9,-0.5);
\draw (-2.9,-0.4) -- (-2.9,-0.6);

\begin{scriptsize}
\draw (-5.4,1.7) node {$\ell_a$};
\draw (-1.9,2.5) node {$\ell_b$};
\draw(-3.9,1.8) node {$\ell_{ab}$};
\draw (-4.3,0.8) node {$\ell'_a$};
\draw (-4.1,2.8) node {$\ell''_a$};
\draw(-2.95,3.1) node {$\ell''_b$};
\draw (-3,1.25) node {$\ell'_b$};
\draw (-4.05,-0.8) node {$I_{a,0}$};
\draw (-3.25,-0.8) node {$I_{b,0}$};
\end{scriptsize}

\end{tikzpicture}
\caption{Configurations de~$I_{a,0}$ et $I_{b,0}$.}\label{fig:IaIb}
\end{figure}

Nous souhaitons disposer d'une minoration de $\la L_{a,0},L_{b,0} \ra$ sur toute la partie~$Y_{0}$ de l'espace de modules~$Y$ situ\'ee au-dessus de~$\va_{0}$. \`A cause de la nature ultram\'etrique du corps~$(\Q,\va_{0})$, l'\'energie $\la L_{a},L_{b} \ra_{0}$ peut \^etre nulle m\^eme lorsque les morphismes~$L_{a,0}$ et~$L_{b,0}$ sont distincts. Il est donc vain de chercher \`a la minorer par une constante strictement positive. Nous parvenons, en revanche, \`a la minorer par une expression d\'ependant de param\`etres $a_{i},b_{j}$ $(1\le i,j\le 3)$ d\'ecrivant~$Y$, de la forme
\[  \la L_{a,0},L_{b,0} \ra \ge F(\labs{a_{i}},\labs{b_{j}},\ 1\le i,j\le 3),\]
sous certaines conditions peu contraignantes.

\bigbreak

\noindent $\blacktriangleright$ \textbf{\'Etape~3: propagation du central au global.}

Dans la troisi\`eme \'etape, nous minorons l'\'energie mutuelle globale $\la L_{a},L_{b} \ra$ de deux applications de Latt\`es $L_{a}$ et $L_{b}$ associ\'ees \`a la multiplication par~2.

Nous commen\c cons par d\'emontrer un r\'esultat technique permettant de passer d'une minoration centrale, c'est-\`a-dire au-dessus de~$\va_{0}$, \`a une minoration similaire en presque toute place, en nous appuyant sur le caract\`ere ad\'elique de la topologie de~$\cM(\Z)$, et une minoration plus faible en les places restantes. 

Un ingr\'edient essentiel est l'application \emph{flot} qui, pour $\eps \in \intof{0,1}$ envoie un point~$x$ associ\'e \`a une semi-norme~$\va_{x}$ sur le point~$x^\eps$ associ\'e \`a la semi-norme~$\va_{x}^\eps$. On dit qu'un espace de Berkovich~$V$ est flottant s'il est stable par le flot. On dit qu'une fonction $g\colon V \to \R$ est $\log$-flottante si, pour tous $x\in V$ et $\eps \in \intof{0,1}$, on a $g(x^\eps) = \eps g(x)$. Des exemples typiques sont les fonctions $\labs{a_{i}}$, $\labs{b_{j}}$, ou encore $x\mapsto \la L_{a,x},L_{b,x} \ra$.

\begin{lemmintro}[\textmd{\emph{infra} lemme~\ref{lem:extensionmajoration}}]\label{lem:extensionintro}
Soient $V$ un espace de Berkovich sur~$\Z$ flottant et $\cE,f_{1},f_{2} \in\cC(V,\R)$ des fonctions continues $\log$-flottantes. Notons $\pr \colon V \to \cM(\Z)$ la projection canonique et supposons qu'elle est presque surjective. Supposons qu'il existe $s_{0} \in  \R_{>0}$ 
tel que 
\begin{enumerate}[i)]
\item l'application $\pr_{\vert \{f_{1}=s_{0}\}}$ est propre ;
\item on a $\cE \ge f_{2}$ sur $\{f_{1}=s_{0}\} \cap \pr^{-1}(\va_{0})$.
\end{enumerate}
Alors, pour tout $\alpha\in \R_{>0}$,  il existe $t_{\alpha}\in  \R_{>0}$ et un ensemble fini~$P_{\alpha}$ de nombres premiers tels que
\[\forall p \notin P_{\alpha} \cup \{\infty\},\ \cE \ge\, -\frac{\alpha}{s_{0}}\, f_{1}+f_{2} \textrm{ sur } \{f_{1}>0\} \cap \pr^{-1}(\va_{p})\]
et 
\[\forall p \in P_{\alpha} \cup \{\infty\},\ \cE \ge\, -\frac{\alpha}{s_{0}}\, f_{1}+f_{2} \textrm{ sur } \{f_{1}>t_{\alpha}\} \cap \pr^{-1}(\va_{p}).\]
\end{lemmintro}

En sommant sur les places, on obtient une minoration globale par une fonction de type hauteur (\cf~th\'eor\`eme~\ref{thm:minoration} et corollaire~\ref{cor:minorationF1}). Dans notre cas, en partant de la minoration centrale obtenue \`a l'\'etape~2, on parvient au r\'esultat suivant.

\begin{theointro}[\textmd{\emph{infra} th\'eor\`eme~\ref{th:hEhab}}]\label{th:minorationLatteshauteur}
Il existe $C\in \R_{>0}$ et $D\in \R$ telles que, pour tous morphismes de Latt\`es $L_{a}\ne L_{b}$ sur~$\Qbar$ comme pr\'ec\'edemment, on ait 
\[ \la L_{a},L_{b} \ra \ge C \, h([a\mathbin: b]) + D,\]
o\`u $h([a\colon b])$ d\'esigne la hauteur logarithmique du point de l'espace projectif de coordonn\'ees homog\`enes $[a_{1}\mathbin: a_{2}\mathbin: a_{3}\mathbin: b_{1}\mathbin: b_{2}\mathbin: b_{3}]$.\end{theointro}

Un raisonnement par l'absurde entra\^ine alors la minoration absolue du th\'eor\`eme~\ref{th:ecartLattes}. En effet, supposons l'existence d'une suite de points dont l'\'energie globale associ\'ee tend vers~0. Pour des indices assez grands, l'\'energie archim\'edienne de ces points est strictement positive, mais petite, donc ces points se rapprochent du bord de l'espace de modules. Cette propri\'et\'e peut se traduire en termes de la valeur absolue archim\'edienne des coordonn\'ees, par exemple par le fait qu'elles soit tr\`es petite (ou d'autres fa\c cons qui se traitent de m\^eme). La formule du produit impose alors aux valeurs absolues ultram\'etriques d'\^etre tr\`es grandes, donc \`a la hauteur d'\^etre tr\`es grande et l'on aboutit \`a une contradiction.

\medbreak

Afin d'insister sur le caract\`ere tr\`es g\'en\'eral et ais\'ement reproductible de la strat\'egie pr\'esent\'ee ici, et notamment des \'etapes~1 et~3, nous indiquons \'egalement comment obtenir un analogue du th\'eor\`eme~\ref{th:minorationLatteshauteur} pour des syst\`emes dynamiques de la forme $P_{c} \colon z \mapsto z^2+c$ avec $c\in \Qbar$. Ce r\'esultat, avec des constantes explicites, a \'et\'e \'etablie par L.~DeMarco, H.~Krieger et H.~Ye dans~\cite{DKY2}. Nous nous appuyons sur leurs calculs pour obtenir la minoration centrale n\'ecessaire \`a notre \'etape~2. Le r\'esultat s'obtient alors en quelques lignes, gr\^ace au lemme~\ref{lem:extensionintro}.

\begin{theointro}[\textmd{\emph{infra} remarque~\ref{rem:PcPd}}]\label{th:minorationpolquad}
Il existe $C'\in \R_{>0}$ et $D'\in \R$ telles que, pour tous $c \ne d \in \Qbar$, on ait 
\[ \la P_{c},P_{d} \ra \ge C' \, h(c,d) + D',\]
o\`u $h(c,d)$ d\'esigne la hauteur logarithmique usuelle sur $\Qbar^2$.
\end{theointro}

Une minoration du type de celle du th\'eor\`eme~\ref{th:ecartLattes} s'en d\'eduit sans peine par un argument similaire \`a celui esquiss\'e plus haut dans le cas des morphismes de Latt\`es. Nous renvoyons le lecteur int\'eress\'e \`a~\cite{DKY2} pour des cons\'equences en termes de majoration uniforme de nombre de points pr\'ep\'eriodiques communs, dans la lign\'ee du th\'eor\`eme~\ref{th:BFT}. 

\subsection{Application \`a la conjecture de Bogomolov-Fu-Tschinkel}

Nous avons d\'ej\`a indiqu\'e plus haut que l'\'energie mutuelle se comportait essentiellement comme le carr\'e d'une distance. En appliquant ce principe \`a partir de deux morphismes de Latt\`es $L_{a}$ et $L_{b}$ sur un corps de nombres~$K$ et d'un ensemble fini $F$ de points de~$\bar K$ stable par $\Gal(\bar K/K)$, on obtient la suite d'in\'egalit\'es
\begin{align*}
\delta_{0}^{1/2} & \le \la L_{a},L_{b}\ra^{1/2}\\
& \le \la L_{a},[F]\ra^{1/2} + \la L_{b},[F]\ra^{1/2} + \textrm{ terme d'erreur}\\
& \le h_{L_{a}}(F)^{1/2} + h_{L_{b}}(F)^{1/2} + \textrm{ terme d'erreur}.
\end{align*}
Dans le cadre de la conjecture de Bogomolov-Fu-Tschinkel, o\`u~$F$ est compos\'e de points pr\'ep\'eriodiques \`a la fois pour~$L_{a}$ et~$L_{b}$, les hauteurs dynamiques sont nulles et seul subsiste le terme d'erreur. C'est donc lui qu'il faut estimer.

On peut y parvenir en utilisant, de nouveau, une strat\'egie de propagation du central au global g\'en\'eralisant celle mise en \oe uvre dans \cite[section~7]{DKY}, strat\'egie que nous mettions en \oe uvre dans une premi\`ere version de ce texte. Dans le temps pr\'ec\'edant la publication, Thomas Gauthier s'est int\'eress\'e \`a ces estim\'ees et en a propos\'e une version plus g\'en\'erale, valable pour des familles arbitraires de syst\`emes dynamiques sur~$\P^1$, en s'appuyant sur des techniques de potentiels H\"older (\cf~\cite[Theorem~B]{GauthierHoelderEstimates}). Dans notre contexte, elle s'\'enonce ainsi.

\begin{theointro}[\textmd{\emph{infra} th\'eor\`eme~\ref{th:CsurS}}]\label{th:LaLbhab}
Il existe $C_{1} \in \R_{>0}$ tel que, pour tout $\delta \in \intoo{0,1}$, on ait
\[\la L_{a},L_{b}\ra \le h_{L_{a}}(F) + h_{L_{b}}(F) + C_{1}\, \Big(\delta - \frac{\log(\delta)}{\sharp F}\Big)\,  (h([a\mathbin:b])+1).\]
\end{theointro}

Pour un choix de~$\delta$ ad\'equat, on peut combiner cette in\'egalit\'e avec celle des th\'eor\`emes~\ref{th:ecartLattes} (pour les points de petite hauteur) et~\ref{th:minorationLatteshauteur} (pour les points de grande hauteur) pour obtenir la majoration absolue de~$\sharp F$ d\'esir\'ee.

\subsection{Autres approches}

Soulignons de nouveau que notre strat\'egie suit fid\`element celle propos\'ee par L.~DeMarco, H.~Krieger et H.~Ye dans~\cite{DKY}, pour traiter le cas de courbes elliptiques~$E_{a}$ et~$E_{b}$ sous forme de Legendre. En particulier, l'id\'ee de passer par la succession de th\'eor\`emes~\ref{th:ecartLattes}, \ref{th:minorationLatteshauteur}, \ref{th:LaLbhab} pour aboutir \`a la conjecture de Bogomolov-Fu-Tschinkel leur est due.

Indiquons plus pr\'ecis\'ement les diff\'erences entre notre travail et~\cite{DKY}. La principale concerne la dimension de l'espace de modules consid\'er\'e. Le fait d'imposer aux courbes~$E_{a}$ et~$E_{b}$ d'\^etre sous forme de Legendre permet de fixer trois des param\`etres pour chaque courbe (en $0,1,\infty$ par convention) et laisse donc seulement deux variables libres (au lieu de 5 dans le cas g\'en\'eral). En outre, dans une grande partie de~\cite{DKY}, c'est le quotient des param\`etres qui intervient, et l'\'etude porte alors sur un espace de dimension~1, cadre dans lequel existent de nombreuses techniques sp\'ecifiques. \`A titre d'exemple, on peut citer les d\'eg\'en\'erescences de mesures, \'etudi\'ees par Ch.~Favre sur un disque hybride dans~\cite{FavreEndomorphisms}, et que nous avons g\'en\'eralis\'es \`a une base arbitraire dans~\cite{DynamiqueI}. Elles sont au c\oe ur de l'\'etape~1.

Des diff\'erences interviennent \'egalement dans l'\'etape~2, o\`u la combinatoire des segments~$I_{a}$ et~$I_{b}$ se complique de fa\c con sensible. En effet, dans~\cite{DKY}, le fait d'imposer trois coordonn\'ees communes force les deux segments \`a \^etre soit align\'es, soit inclus l'un dans l'autre. Par cons\'equent, seul le second cas du th\'eor\`eme~\ref{th:La0Lb0} intervient et les trois derniers termes de la formule disparaissent.

Dans l'\'etape~3, l'approche est plus radicalement diff\'erente. Dans~\cite{DKY}, les \'energies mutuelles sont calcul\'ees explicitement en valuation triviale et \`a chaque place ultram\'etrique, en traitant \`a part le cas de la valuation 2-adique. Du c\^ot\'e archim\'edien, l'\'energie mutuelle est estim\'ee \`a partir de celle en valuation triviale par des techniques de d\'eg\'en\'erescence bas\'ees sur~\cite{FavreEndomorphisms}. Dans ce texte, nous effectuons un unique calcul en valuation triviale et en tirons directement des informations en toute place en nous appuyant sur~\cite{DynamiqueI} (que l'on peut voir comme un analogue global de \cite{FavreEndomorphisms}).

\medbreak

Indiquons finalement d'autres approches, plus indirectes, au th\'eor\`eme~\ref{th:BFT}. Elles nous ont \'et\'e signal\'ees par L.~DeMarco, que nous remercions ici. Les techniques employ\'ees y sont de nature fondamentalement diff\'erente des n\^otres, alg\'ebriques et arithm\'etiques plut\^ot qu'analytiques et dynamiques, m\^eme si les champs d'applications se recoupent parfois, comme ici. .

Dans leur article~\cite{BogomolovTschinkelSmallFields}, F.~Bogomolov et Yu.~Tschinkel d\'emontrent que l'intersection $\pi_{a}(E_{a}[\infty]) \cap \pi_{b}(E_{b}[\infty])$ est finie \textit{via} la conjecture de Manin-Mumford (d\'emontr\'ee par M.~Raynaud dans~\cite{RaynaudManinMumford}) appliqu\'ee \`a la courbe de~$E_{a}\times E_{b}$ obtenue comme image r\'eciproque de la diagonale de~$\P^1\times\P^1$ par $\pi_{a}\times\pi_{b}$. En effectuant cette construction de fa\c con relative sur l'espace de modules~$\cY$, on obtient une courbe relative~$\cC$ dans un espace fibr\'e en produits de courbes elliptiques. Apr\`es quelques v\'erifications d'apparence routini\`ere, la version de la conjecture de Bogomolov relative d\'emontr\'ee par L.~K\"uhne dans~\cite{KuehneRBC} devrait permettre de conclure que l'ensemble des points de torsion contenus dans~$\cC$ n'est pas Zariski-dense, et d'en d\'eduire le th\'eor\`eme~\ref{th:BFT}. De fa\c con alternative, on peut appliquer les r\'esultats r\'ecemment obtenus par Z.~Gao, T.~Ge et L.~K\"uhne dans~\cite{GaoGeKuehne} sur une version uniforme de la conjecture de Mordell-Lang. 

Notons que l'on peut inverser le raisonnement de cette derni\`ere approche pour d\'emontrer, \`a partir du th\'eor\`eme~\ref{th:BFT}, une version uniforme de la conjecture de Manin-Mumford pour certaines familles de courbes. Cette strat\'egie a \'et\'e men\'ee \`a bien dans~\cite[section~9]{DKY} en genre~2, \`a partir de la conjecture de Bogomolov-Fu-Tschinkel pour les paires de courbes elliptiques sous forme de Legendre. En utilisant le cas g\'en\'eral de la conjecture, il est vraisemblable que l'on puisse \'egalement obtenir des r\'esultats en genre~3, 4 ou~5. L'article~\cite{GaoGeKuehne} traitant d\'ej\`a ces questions de fa\c con plus g\'en\'erale, nous n'avons pas cherch\'e \`a poursuivre dans cette direction.

Finalement, mentionnons le texte \cite{DeMarcoMavrakiDynamicsonP1} de L.~DeMarco et M.~Mavraki, post\'erieur \`a celui-ci, qui propose une nouvelle d\'emonstration du th\'eor\`eme~\ref{th:BFT} (mais pas du th\'eor\`eme~\ref{th:ecartLattes}), par des techniques de dynamique complexe, ainsi que des g\'en\'eralisations \`a d'autres classes de morphismes suffisamment g\'en\'eriques.

\subsection{Organisation du texte}

Une introduction rapide \`a la th\'eorie des espaces de Berkovich sur les anneaux de Banach figure dans le texte~\cite{DynamiqueI} qui pr\'ec\`ede celui-ci et nous y renvoyons le lecteur int\'eress\'e. Nous nous contentons ici de rappeler, dans la section~\ref{sec:BerkovichZ}, la description explicite du spectre analytique~$\cM(\Z)$, l'espace de base sur lequel s'effectuent nos constructions. 

La section~\ref{sec:energiemutuelle} commence par quelques rappels sur la th\'eorie de l'\'energie mutuelle de Ch.~Favre et J.~Rivera--Letelier. Elle contient \'egalement des calculs explicites de cette \'energie, menant au th\'eor\`eme~\ref{th:La0Lb0} de l'\'etape~2. 

La section~\ref{sec:Lattes} est consacr\'ee aux morphismes de Latt\`es associ\'ees \`a la multiplication par~2. Nous y construisons l'espace de modules des paires de tels morphismes et d\'emontrons le th\'eor\`eme~\ref{th:continuintro}, sur la  continuit\'e des mesures, qui fait l'objet de l'\'etape~1.

L'\'etape~3 s'effectue en deux temps. Dans la section~\ref{sec:fibrecentralehauteurs}, nous d\'emontrons le lemme~\ref{lem:extensionintro} permettant de passer d'une estimation centrale \`a une estimation en toute place. Nous avons cherch\'e \`a r\'ediger cette partie de fa\c con aussi g\'en\'erale que possible, afin de la rendre applicable dans d'autres contextes. Dans la section~\ref{sec:minoration}, nous utilisons ces r\'esultats dans le cadre qui nous int\'eresse pour d\'emontrer les th\'eor\`emes~\ref{th:minorationLatteshauteur} puis~\ref{th:ecartLattes}. 

La section~\ref{sec:courbeselliptiques} finale est consacr\'ee \`a la d\'emonstration du th\'eor\`eme~\ref{th:BFT}, r\'epondant \`a la conjecture de Bogomolov-Fu-Tschinkel.

\subsection{Remerciements}

L'auteur remercie chaleureusement Velibor Bojkovi\'c pour de nombreux \'echanges fructueux, qui ont notamment permis de d\'etecter une erreur majeure dans une version pr\'eliminaire de ce texte. Il remercie \'egalement Marco Maculan pour sa relecture minutieuse et ses remarques avis\'ees, ainsi que Thomas Gauthier pour lui avoir transmis une version pr\'eliminaire de~\cite{GauthierHoelderEstimates}. Merci \'egalement au rapporteur ou \`a la rapporteuse pour ses commentaires pertinents.

\subsection{Notations}\label{sec:notations}

Nous regroupons ici quelques notations qui seront utilis\'ees tout au long du texte.

\smallbreak

\noindent $\bullet$ On note $\cP$ l'ensemble des nombres premiers. Pour $p\in \cP$, on note~$\va_{p}$ la valeur absolue $p$-adique sur~$\Q$ usuelle, normalis\'ee par la condition $\abs{p}_{p} = \frac1p$. On note~$\va_{\infty}$ la valeur absolue usuelle sur~$\Q$.

Soit $K$ un corps de nombres. On note~$M_{K}$ l'ensemble des places de~$K$. Soit $v\in M_{K}$. On note $v_{\vert\Q}$ l'unique place de~$\Q$ telle que $v \mid v_{\vert\Q}$ et $\va_{v}$ l'unique valeur absolue sur~$K$ associ\'ee \`a la place~$v$ qui \'etend~$\va_{v_{\vert\Q}}$. On pose
\[N_{v} := \frac{[K_{v} \mathbin: \Q_{v_{\vert\Q}}]}{[K \mathbin: \Q]}.\]

Soient $(x_{0},\dotsc,x_{n}) \in \Qbar^{n+1} \setminus \{(0,\dotsc,0)\}$. On d\'efinit la \emph{hauteur projective} de $(x_{0},\dotsc,x_{n})$ par
\[h([x_{0}\mathbin:\dotsb\mathbin:x_{n}]) := \sum_{v\in M_{K}} N_{v}\, \max_{0\le i\le n} (\log(\abs{x_{i}}_{v})),\]
o\`u~$K$ est un corps de nombres contenant $x_{0},\dotsc,x_{n}$. 
Cette quantit\'e est ind\'ependante du choix de~$K$.

Soit $(x_{1},\dotsc,x_{n}) \in \Qbar^n$. On d\'efinit la \emph{hauteur} de $(x_{1},\dotsc,x_{n})$ par
\[h(x_{1},\dotsc,x_{n}) := h([1 \mathbin: x_{1}\mathbin:\dotsb\mathbin:x_{n}]).\]
En posant
\[\fonction{\log^+}{\R}{\R_{\ge 0}}{x}{\max(\log(x),0)},\]
on a donc
\[h(x_{1}\mathbin:\dotsb\mathbin:x_{n}) := \sum_{v\in M_{K}} N_{v}\, \log^+(\max_{1\le i\le n} (\abs{x_{i}}_{v})).\]

Soit $K$ un corps de nombres. Pour toute place~$v$ de~$K$ et tout point $P$ (resp. toute partie~$F$) de~$K$, on note~$P_{v}$ (resp. $F_{v}$) son image dans~$K_{v}$.

\medbreak

\noindent $\bullet$ Pour toute famille finie $(t_{i})_{i\in I}$ de nombres r\'eels, on d\'efinit le \emph{sous-maximum} de $(t_{i})_{i\in I}$ par
\[\smax_{i\in I} (t_{i}) = \min_{j\in I} \big(\max_{i\ne j}(t_{i})\big).\] 
Pour $t_{1}\le \dotsb \le t_{n-1}\le t_{n} \in \R$, on a donc
\[\smax_{1\le i\le n}(t_{i}) = t_{n-1}.\]

\noindent $\bullet$ Soit $k$ un corps valu\'e complet. On note~$\E{1}{k}$ (resp.~$\EP{1}{k}$) la droite affine (resp. projective) analytique au sens de Berkovich. 

Supposons que $k$ est ultram\'etrique. Pour $\alpha\in k$ et $r\in \R_{>0}$, on note~$\eta_{\alpha,r}$ l'unique point du bord de Shilov du disque ferm\'e de centre~$\alpha$ et de rayon~$r$ et $\chi_{\alpha,r}$ la mesure de Dirac support\'ee en~$\eta_{\alpha,r}$. Tout segment~$I$ de~$\EP{1}{k}$ form\'e de points de type~2 ou~3 poss\`ede une longueur (logarithmique) canonique, que nous noterons~$\ell(I)$. Pour $\alpha\in k$ et $r,s\in \R_{>0}$ avec $r\le s$, on a
\[ \ell(\intff{\eta_{\alpha,r},\eta_{\alpha,s}}) = \log\big(\frac s r\big).\]

Supposons que $k = \C$. Alors la droite $\E{1}{k}$ (resp.~$\EP{1}{k}$) est isomorphe \`a~$\C$ (resp.~$\P^1(\C)$). Pour $\alpha\in \C$ et $r\in \R_{>0}$, on note~$\chi_{\alpha,r}$ la mesure de Haar de masse totale~1 sur le cercle $\overline{C}(a,r)$. 

Supposons que $k = \R$. Alors la droite $\E{1}{k}$ (resp.~$\EP{1}{k}$) est isomorphe au quotient de~$\C$ (resp.~$\P^1(\C)$) par la conjugaison complexe. Pour $\alpha\in \R$ et $r\in \R_{>0}$, on note~$\chi_{\alpha,r}$ l'image de la mesure~$\chi_{\alpha,r}$ sur~$\C$. 

Nous renvoyons \`a \cite[sections~2.2 et~6.1]{DynamiqueI} pour des pr\'ecisions sur ces espaces et ces mesures respectivement.

\section{Espaces de Berkovich sur les entiers}\label{sec:BerkovichZ}

Dans~\cite{rouge}, V.~Berkovich d\'efinit une notion d'espace analytique sur un anneau de Banach~$\cA$ (\cf~\'egalement \cite[section~2]{DynamiqueI} pour quelques rappels). Nous nous int\'eressons ici au cas particulier o\`u cet anneau de Banach est l'anneau des entiers relatifs~$\Z$ ou, plus g\'en\'eralement, un anneau d'entiers de corps de nombres, ou encore un localis\'e d'un de ces anneaux.

\subsection{Le cas de~$\Z$}

Consid\'erons l'anneau des entiers relatifs~$\Z$ muni de la valeur absolue usuelle~$\va_{\infty}$. Il s'agit d'un anneau de Banach et l'on peut donc consid\'erer son spectre analytique~$\cM(\Z)$ au sens de V.~Berkovich. Rappelons que ce dernier est d\'efini, ensemblistement, comme l'ensemble des semi-normes multiplicatives sur~$\Z$. Le th\'eor\`eme d'Ostrowski permet de le d\'ecrire explicitement. Il contient~: 
\begin{itemize}
\item la valeur absolue triviale~$\va_{0}$ sur~$\Z$ ;
\item les valeurs absolues archim\'ediennes~$\va_{\infty}^\eps$ avec $\eps \in \intof{0,1}$ ;
\item pour tout nombre premier~$p$, les valeurs absolues $p$-adiques~\added{$\va_{p}^\eps$} \deleted{$\va_{\infty}^\eps$} avec $\eps \in \intoo{0,+\infty}$ ;
\item pour tout nombre premier~$p$, la semi-norme~$\va_{p}^{+\infty}$ d\'efinie comme la composition de l'application quotient $\Z \to \Z/p\Z$ et de la valeur absolue triviale sur~$\Z/p\Z$. 
\end{itemize}
Nous noterons $a_{0}$ le point associ\'e \`a~$\va_{0}$, $a_{\infty}^\eps$ le point associ\'e \`a~$\va_{\infty}^\eps$, pour $\eps \in \intof{0,1}$, et $a_{p}^\eps$ le point associ\'e \`a~$\va_{p}^\eps$, pour $p\in \cP$ et $\eps \in \intof{0,+\infty}$. La figure~\ref{fig:MZ} contient une repr\'esentation d'un plongement (non canonique) de~$\cM(\Z)$ dans~$\R^2$ respectant la topologie.

\begin{figure}[!h]
\centering
\begin{tikzpicture}
\foreach \x [count=\xi] in {-2,-1,...,17}
\draw (0,0) -- ({10*cos(\x*pi/10 r)/\xi},{10*sin(\x*pi/10 r)/\xi}) ;
\foreach \x [count=\xi] in {-2,-1,...,17}
\fill ({10*cos(\x*pi/10 r)/\xi},{10*sin(\x*pi/10 r)/\xi}) circle ({0.07/(sqrt(\xi)}) ;

\draw ({10.5*cos(-pi/5 r)},{10.5*sin(-pi/5 r)}) node{$a_{\infty}$} ;
\fill ({5.5*cos(-pi/5 r)},{5.5*sin(-pi/5 r)}) circle (0.07) ;
\draw ({5.5*cos(-pi/5 r)},{5.5*sin(-pi/5 r)-.1}) node[below]{$a_{\infty}^\eps$} ;

\draw ({11*cos(-pi/10 r)/2+.1},{11*sin(-pi/10 r)/2}) node{$a_2^{+\infty}$} ;
\fill ({2.75*cos(-pi/10 r)},{2.75*sin(-pi/10 r)}) circle ({0.07/(sqrt(2)}) ;
\draw ({2.75*cos(-pi/10 r)},{2.75*sin(-pi/10 r)-.05}) node[below]{$a_2^\eps$} ;

\draw ({12*cos(pi/5 r)/5+.1},{12*sin(pi/5 r)/5+.1}) node{$a_p^{+\infty}$} ;
\end{tikzpicture}
\caption{Le spectre de Berkovich $\cM(\Z)$.}\label{fig:MZ}
\end{figure}
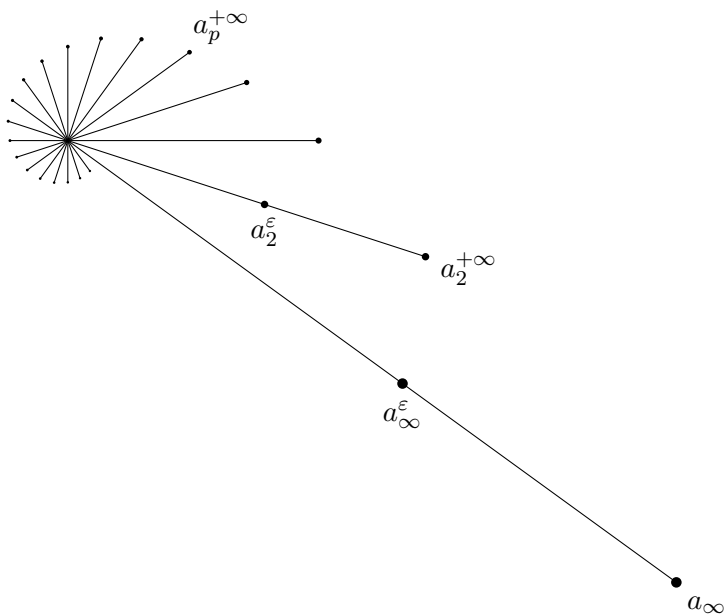

Notons 
\[ \cN(\Q) := \{a_{\infty},\ a_{p},\ p\in \cP\} \subset \cM(\Z)\]
l'ensemble des points de~$\cM(\Z)$ correspondant aux valeurs absolues normalis\'ees.

\medbreak

Il existe une notion d'espace de Berkovich sur~$\Z$. Un tel espace~$X$ poss\`ede un morphisme structural vers le spectre~$\cM(\Z)$, que nous noterons $\pr \colon X \to \cM(\Z)$ dans ce texte. Rappelons qu'\`a tout sch\'ema~$\cX$ localement de type fini sur~$\Z$, on peut associer son \emph{analytifi\'e} $\cX^\an$, qui est un espace de Berkovich sur~$\Z$ (\cf~\cite[section~4.1]{CTC}). 

L'anneau~$\Z$ est un \emph{bon anneau de Banach} au sens de~\cite[d\'efinition~2.3]{DynamiqueI}. Par cons\'equent, les r\'esultats de~\cite{DynamiqueI} s'appliquent. 

L'anneau~$\Z$ est un anneau \emph{flottant}, au sens de~\cite[d\'efinition~2.11]{DynamiqueI}. Cela implique que, pour tout $\eps \in \intof{0,1}$, on dispose d'une application
\[\varphi_{\eps} \colon x \in \cM(\Z) \mapstoo x^\eps \in \cM(\Z),\]
qui consiste \`a envoyer une seminorme sur sa puissance~$\eps$. Cette application est encore d\'efinie pour tout espace affine analytique sur~$\Z$, et m\^eme, par recollement, pour tout analytifi\'e d'un sch\'ema localement de type fini sur~$\Z$.

\subsection{Le cas d'un localis\'e de~$\Z$}\label{sec:Z1N}

Soit $N \in \Z\setminus\{0\}$. Nous souhaiterions disposer d'une th\'eorie semblable \`a la th\'eorie sur~$\Z$ pour l'anneau localis\'e~$\Z[\frac1N]$. 

Posons
\begin{align*} 
\cU_{N}(\Z) &:= \{ x \in \cM(\Z) : N(x) \ne 0\}\\
&= \cM(\Z) \setminus \bigcup_{p\in \cP_{\mid N}}  \{a_{p}^{+\infty}\},
\end{align*}
o\`u $\cP_{\mid N}$ d\'esigne l'ensemble des nombres premiers divisant~$N$. C'est un ouvert flottant de~$\cM(\Z)$, naturellement associ\'e \`a~$\Z[\frac1N]$. 

Pour pouvoir d\'evelopper la th\'eorie comme dans le cas de~$\Z$, nous allons approcher~$\cU_{N}(\Z)$ par une famille de spectres d'anneaux de Banach. Soit $\alpha \in \R_{>0}$. Munissons $\Z[\frac1N]$ de la norme 
\[ \nm_{N,\alpha} := \max(\va_{\infty},\ \va_{p}^\alpha,\ p\in \cP_{\mid N}\}.\] 
Nous obtenons un anneau de Banach dont le spectre s'identifie \`a la partie compacte de~$\cM(\Z)$ obtenue en coupant, pour tout $p\in \cP_{\mid N}$, la branche $p$-adique \`a hauteur de~$\va_{p}^\alpha$~:
\[\cM\Big(\Z\Big[\frac1N\Big],\nm_{N,\alpha}\Big) = \intff{a_{0},a_{\infty}} \cup \bigcup_{p\notin \cP_{\mid N}}  \intff{a_{0},a_{p}^{+\infty}} \cup \bigcup_{p\in \cP_{\mid N}}  \intff{a_{0},a_{p}^\alpha}.\]
Notons $\cU_{N,\alpha}$ l'int\'erieur de~$\cM(\Z[\frac1N],\nm_{N,\alpha})$ dans~$\cM(\Z)$~:
\[\cU_{N,\alpha} = \intff{a_{0},a_{\infty}} \cup \bigcup_{p\notin \cP_{\mid N}}  \intff{a_{0},a_{p}^{+\infty}} \cup \bigcup_{p\in \cP_{\mid N}}  \intfo{a_{0},a_{p}^\alpha}.\]

Pour tous $\alpha,\beta \in \R_{>0}$ avec $\alpha \le \beta$, on a une immersion ouverte $\iota_{\beta,\alpha} \colon \cU_{N,\alpha} \hookrightarrow \cU_{N,\beta}$. On en d\'eduit un isomorphisme naturel
\[ \colim_{\alpha>0}\, \cU_{N,\alpha} \simtoo \cU_{N}.\] 

Nous pouvons maintenant d\'efinir l'analytifi\'e d'un sch\'ema~$X$ localement de type fini sur~$\Z[\frac1N]$ en suivant cette approche. Pour tout $\alpha\in \R_{>0}$, notons $X^\an_{\alpha}$ l'analytifi\'e de~$X$ sur $(\Z[\frac1N],\nm_{N,\alpha})$ et $X^\an_{\alpha-}$ sa restriction au-dessus de l'ouvert~$\cU_{N,\alpha}$.  Pour tous $\alpha,\beta \in \R_{>0}$ avec $\alpha \le \beta$, l'immersion ouverte $\iota_{\beta,\alpha}$ induit une immersion ouverte $X^\an_{\alpha-} \hookrightarrow X^\an_{\beta-}$. On d\'efinit alors l'analytifi\'e de~$X$ par 
\[X^\an :=  \colim_{\alpha>0}\, X^\an_{\alpha-}.\] 
L'espace~$X^\an$ est muni d'un morphisme naturel~$\pr$ vers~$\cU_{N}$. En particulier, c'est un espace $\Z$-analytique. On v\'erifie qu'il est flottant.

\subsection{Le cas d'un anneau d'entiers de corps de nombres}

Soit~$K$ un corps de nombres. Notons~$A$ son anneau d'entiers et munissons-le de la norme
\[\nm_{A} := \max_{\sigma \in M_{K,\infty}} (\va_{\sigma}),\]
o\`u~$M_{K,\infty}$ d\'esigne l'ensemble des places infinies de~$K$. 
C'est un anneau de Banach dont le spectre poss\`ede une description tr\`es similaire \`a celle de~$\cM(\Z)$. Nous renvoyons \`a \cite[section~3.1]{A1Z} pour plus de d\'etails.

Posons 
\[ \cN(K) := \{a_{\sigma},\ \sigma\in M_{K}\} \subset \cM(A).\]

Comme dans le cas de~$\Z$, l'anneau~$A$ est un bon anneau de Banach flottant et l'on peut d\'efinir l'analytifi\'e~$X^\an$ d'un sch\'ema localement de type fini sur~$A$.

Les arguments de la section~\ref{sec:Z1N} s'adaptent \'egalement \`a ce cadre. Soit $d\in A\setminus\{0\}$. Posons
\[\cU_{d}(A) := \{ x \in \cM(A) : d(x) \ne 0\}.\]
On peut analytifier un sch\'ema~$X$ localement de type fini sur~$A[\frac1d]$ de fa\c con \`a obtenir un espace $A$-analytique~$X^\an$ flottant muni d'un morphisme vers~$\cU_{d}(A)$.

\section{\'Energie mutuelle de deux mesures}\label{sec:energiemutuelle}

Dans cette section, nous rappelons la notion d'\'energie mutuelle pour deux mesures de Radon sign\'ees, telle que d\'efinie par Charles Favre et Juan Rivera--Letelier dans~\cite{FRLEquidistribution}. Nous la calculons explicitement dans le cas de deux mesures support\'ees sur un segment dans la droite projective au-dessus d'un corps ultram\'etrique. 

\subsection{D\'efinitions}\label{sec:mesuresdef}

\subsubsection{Cadre local}\label{sec:energielocale}

Commen\c cons par rappeler la d\'efinition d'\'energie mutuelle dans le cas classique de~$\P^1(\C)$,  \cf~\cite[section~2.4]{FRLEquidistribution}. Notons $\textrm{Diag}$ la diagonale de~$\C^2$.

\begin{defi}\label{def:energieC}
Soient $\rho$ et~$\rho'$ deux mesures sur~$\P^1(\C)$. Supposons que la fonction 
\[ (z,w) \in \C^2 \setminus \textrm{Diag} \mapstoo \log(\abs{z-w}) \in \R\] 
est int\'egrable pour $\abs{\rho}\otimes \abs{\rho'}$, o\`u $\abs{\rho}$ et~$\abs{\rho'}$ d\'esignent les mesures traces de~$\rho$ et~$\rho'$. On d\'efinit alors l'\emph{\'energie mutuelle} de~$\rho$ et~$\rho'$ par
\[(\rho,\rho') := - \int_{\C^2 \setminus \textrm{Diag}} \log(\abs{z-w}) \diff \rho(z) \otimes \diff \rho'(w).\]
\end{defi}
Il suit directement de la d\'efinition que l'on a $(\rho',\rho) = (\rho,\rho')$.

\medbreak

Soit $\rho$ une mesure sur~$\P^1(\C)$ dont la trace est \emph{\`a potentiel continu} (c'est-\`a-dire qu'elle peut s'\'ecrire sous la forme $\Delta(u)$, avec $u$ continue, au voisinage de tout point de~$\P^1(\C)$). Soit $\rho'$ une mesure sur~$\P^1(\C)$ satisfaisant l'une des propri\'et\'es suivantes~:
\begin{enumerate}[i)]
\item $\rho'$ est une mesure \`a support fini ne chargeant pas l'infini~;
\item $\abs{\rho'}$ est \`a potentiel continu.
\end{enumerate}
Alors l'hypoth\`ese d'int\'egrabilit\'e de la d\'efinition~\ref{def:energieC} est satisfaite (\cf~\cite[lemme~2.4]{FRLEquidistribution}). En outre, si $g_{\rho}$ est un potentiel de~$\rho$ (c'est-\`a-dire que $\Delta(g_{\rho})=\rho$) et que $\rho'(\P^1(\C))=0$, on a 
\[(\rho,\rho') = - \int_{\C} g_{\rho} \diff \rho'\]
(\cf~\cite[lemme~2.5]{FRLEquidistribution}). Remarquons que cette expression montre que la quantit\'e~$(\rho,\rho')$ est ind\'ependante du choix de la coordonn\'ee sur~$\C$. 

Si, en outre, on a $\rho(\P^1(\C))=0$, alors $(\rho,\rho)\ge 0$ et $(\rho,\rho) = 0$ si, et seulement si, $\rho=0$ (\cf~\cite[proposition~2.6]{FRLEquidistribution}). 

\medbreak

Dans le cadre des espaces de Berkovich sur~$\Z$ dans lequel nous travaillons, nous aurons besoin de disposer de la th\'eorie non seulement sur l'espace~$\P^1(\C)$ classique, mais \'egalement sur $\EP{1}{k}$ pour tout corps valu\'e archim\'edien complet. Ce dernier cas n'est plus g\'en\'eral qu'en apparence et la th\'eorie se ram\`ene imm\'ediatement \`a celle sur le corps~$\C$ muni de la valeur absolue usuelle en utilisant les d\'efinitions et m\'ethodes de~\cite[section~5.2]{DynamiqueI}.

\medbreak

Ch.~Favre et J.~Rivera--Letelier d\'eveloppent \'egalement la th\'eorie pour $\EP{1}{k}$, o\`u~$k$ est un corps valu\'e ultram\'etrique complet (\cite[section~4.4 et~4.5]{FRLEquidistribution}). Par rapport au cas complexe, la seule modification \`a apporter  concerne la fonction $\log(\abs{z-w})$, qu'ils remplacent par un analogue purement ultram\'etrique permettant de l'\'etendre aux points non rationnels de~$\EP{1}{k}$. Tous les \'enonc\'es ci-dessus restent valablent dans le cadre ultram\'etrique. 

Dans un souci de pr\'ecision, indiquons que Ch.~Favre et J.~Rivera--Letelier ne d\'eveloppent la th\'eorie que dans le cas de~$\C_{p}$. Les m\^emes arguments permettent de traiter, sans changement, le cas d'un corps valu\'e ultram\'etrique complet alg\'ebriquement clos. Le cas g\'en\'eral s'y ram\`ene par changement de base, toutes les formules \'etant invariantes par cette op\'eration.

\medbreak

Introduisons encore une notation. Dans tous les cas, lorsque les conditions assurant l'existence sont satisfaites, on pose
\[ \la \rho,\rho' \ra := \frac12 \,(\rho-\rho',\rho-\rho').\]

\subsubsection{Cadre global}\label{sec:energieglobale}

Soit $K$ un corps de nombres. Ch.~Favre et J.~Rivera--Letelier introduisent une notion de mesure ad\'elique sur~$K$, \cf~\cite[d\'efinition~1.1]{FRLEquidistribution}. 

\begin{defi}\label{def:mesureadelique}
Une \emph{mesure ad\'elique} sur~$K$ est une famille de mesures 
\[\rho := (\rho_{v})_{v\in M_{K}}\]
o\`u, pour toute place $v\in M_{K}$, $\rho_{v}$ est une mesure de probabilit\'e sur~$\EP{1}{K_{v}}$, v\'erifiant les propri\'et\'es suivantes~:
\begin{enumerate}[i)]
\item pour toute place $v\in M_{K}$, $\rho_{v}$ poss\`ede un \emph{potentiel continu}~: il existe $g_{v} \colon \EP{1}{K_{v}} \to \R$ continue telle que $\rho_{v} = \chi_{0,1} + \Delta (g_{v})$~;
\item pour presque toute place $v\in M_{K}$, on a $\rho_{v} = \chi_{0,1}$.
\end{enumerate}
\end{defi}

\begin{rema}\label{rem:rhoL}
Dans la d\'efinition originale de Ch.~Favre et J.~Rivera--Letelier, pour toute $v\in M_{K}$, la mesure~$\rho_{v}$ est une mesure sur~$\EP{1}{\C_{v}}$, o\`u~$\C_{v}$ est le compl\'et\'e d'une cl\^oture alg\'ebrique de~$K_{v}$. Cette modification mineure n'a aucune influence sur la th\'eorie d\'evelopp\'ee dans~\cite{FRLEquidistribution}. Elle est importante en ce qu'elle nous permettra d'utiliser les r\'esultats de continuit\'e d\'emontr\'es dans~\cite{DynamiqueI}, les corps~$K_{v}$ apparaissant comme corps r\'esiduels compl\'et\'es dans le spectre de l'anneau des entiers de~$K$.

Rappelons \'egalement que nous pouvons tirer en arri\`ere les mesures par un morphisme fini (\cf~\cite[section~3.2]{DynamiqueI}). Ainsi, \`a partir d'une mesure sur~$\EP{1}{K_{v}}$, pouvons-nous en obtenir une sur~$\EP{1}{\ell}$, pour toute extension finie~$\ell$ de~$K_{v}$, \`a d\'efaut d'en obtenir une sur~$\EP{1}{\C_{v}}$. En particulier, pour toute extension finie~$L$ de~$K$, toute mesure ad\'elique~$\rho$ sur~$K$ induit une mesure ad\'elique~$\rho_{L}$ sur~$L$, par image r\'eciproque. 
\end{rema}

\begin{exem}\label{ex:mFr}
Soit $F$ une partie finie non vide de~$K$. Pour toute $v\in M_{K}$, notons~$F_{v}$ son image dans~$K_{v}$. Soit $r = (r_{v})_{v\in M_{K}}$ une famille de~$\R_{>0}$ dont presque tous les \'el\'ements valent~1. Alors, la famille
\[ m_{F,r} := \big( \frac1{\sharp F_{v}} \, \sum_{P\in F_{v}} \chi_{P,r_{v}}\big)_{v\in M_{K}}\]
est une mesure ad\'elique sur~$K$.

Rappelons, \cf~section~\ref{sec:notations}, que la notation $\chi_{P,r_{v}}$ d\'esigne la mesure de Haar sur le cercle $\overline{C}(P,r_{v})$, ou son pouss\'e en avant, dans le cas archim\'edien, et la mesure de Dirac support\'ee en l'unique point du bord de Shilov de ce cercle, dans le cas ultram\'etrique.
\end{exem}

\begin{exem}\label{ex:rhoR}
Soit $\varphi \in K(T)$ une fraction rationnelle de degr\'e sup\'erieur ou \'egal \`a~2. Pour toute $v\in M_{K}$, notons $\rho_{\varphi,v}$ la mesure d'\'equilibre sur $\EP{1}{K_{v}}$ associ\'ee \`a~$\varphi$. D'apr\`es \cite[th\'eor\`eme~8]{FRLEquidistribution}, la famille $\rho_{\varphi} := (\rho_{\varphi,v})_{v\in M_{K}}$ est une mesure ad\'elique sur~$K$.
\end{exem}

On peut encore d\'efinir l'\'energie mutuelle dans ce contexte. Soient $\rho := (\rho_{v})_{v\in M_{K}}$ et $\rho' := (\rho'_{v})_{v\in M_{K}}$ des mesures ad\'eliques sur~$K$. On pose 
\[ (\rho,\rho') := \sum_{v \in M_{K}} N_{v}\, (\rho_{v},\rho'_{v})\]
et 
\[ \la \rho,\rho'\ra := \frac12 \,(\rho-\rho',\rho-\rho').\]
Il d\'ecoule de la d\'efinition de mesure ad\'elique et des propri\'et\'es expos\'ees \`a la section~\ref{sec:energielocale} que ces quantit\'es font sens. 

\begin{rema}\label{rema:comparaisonenergie}
Soient $\varphi,\psi \in K(T)$ des fractions rationnelles de degr\'e sup\'erieur ou \'egal \`a~2. On peut alors d\'efinir l'\'energie mutuelle $\la \rho_{\varphi},\rho_{\psi}\ra$ comme ci-dessus.

D'autres d\'efinitions existent dans la litt\'erature. Mentionnons tout d'abord celle de C.~Petsche-L.~Szpiro-T.~Tucker (\cf~\cite{PetscheSzpiroTucker}). Elle est \'equivalente \`a la pr\'ec\'edente d'apr\`es~\cite[theorem~9]{Fili} et \cite[theorem~1]{PetscheSzpiroTucker}.

Une autre d\'efinition est due \`a S.-W.~Zhang (\cf~\cite{ZhangSmallPoints}), \`a la suite des travaux de S.~Arakelov et J.-B.~Bost-H.~Gillet-Chr.~Soul\'e. Dans ce cas, l'accouplement appara\^it comme produit d'intersection arithm\'etique des premi\`eres classes de Chern de deux fibr\'es m\'etris\'es sur~$\P^1_{K}$ associ\'es \`a~$\varphi$ et~$\psi$. Le r\'esultat obtenu est encore le m\^eme que pr\'ec\'edemment, comme expliqu\'e dans \cite[section~1.4]{PetscheSzpiroTucker}. 
\end{rema}

\'Enon\c cons maintenant une in\'egalit\'e triangulaire due \`a P.~Fili (\cf~\cite[theorem~1]{Fili}). \JP{Probablement pas utilis\'e au final, mais on peut le laisser quand m\^eme.}

\begin{theo}\label{th:inegalitetriangulaireenergie}
Soient $\rho_{1},\rho_{2},\rho_{3}$ des mesures ad\'eliques sur~$K$. Alors, on a 
\[\la \rho_{1},\rho_{2} \ra^{1/2} \le \la \rho_{1},\rho_{3} \ra^{1/2}  + \la \rho_{3},\rho_{2} \ra^{1/2} .\]
\end{theo}

Introduisons \'egalement des familles de mesures non ad\'eliques pour lesquelles nous pourrons encore calculer des \'en\'ergies mutuelles.

\begin{nota}\label{nota:[F]}
Soit $\bar K$ une cl\^oture alg\'ebrique de~$K$. Soit $F$ une partie finie non vide de~$\P^1(\bar K)$ stable par $\Gal(\bar K/K)$. Soit $v\in M_{K}$ et soit~$\C_{v}$ le compl\'et\'e d'une cl\^oture alg\'ebrique de~$K_{v}$. L'ensemble~$F$ d\'efinit un ensemble~$F_{\C_{v}}$ de points rationnels de~$\EP{1}{\C_{v}}$ (de m\^eme cardinal). Notons $[F]_{\C_{v}}$ la mesure de probabilit\'e sur~$\EP{1}{\C_{v}}$ \'equidistribu\'ee en ces points et~$[F]_{v}$ son image sur~$\EP{1}{K_{v}}$. Posons
\[ [F] := ([F]_{v})_{v\in M_{K}}.\]
\end{nota}

Soit $\rho := (\rho_{v})_{v\in M_{K}}$ une mesure ad\'elique. Soit~$F$ une partie finie non vide de~$\P^1(\bar K)$. On peut encore d\'efinir l'\'energie mutuelle dans ce cadre en posant
\[ h_{\rho}(F) := \frac12 \,\sum_{v \in M_{K}} N_{v}\, (\rho_{v} - [\Gal(\bar K/K) \cdot F]_{v}, \rho_{v} - [\Gal(\bar K/K) \cdot F]_{v}).\]
Il s'agit d'une g\'en\'eralisation de la hauteur classique, que l'on retrouve en choisissant pour~$\rho$ la famille \og constante \fg{} $(\chi_{0,1})_{v\in M_{K}}$. On peut \'egalement retrouver la hauteur canonique associ\'ee \`a une fraction rationnelle~$\varphi$, au sens G.~Call et J.~Silverman (\cf~\cite{CallSilverman}), en choisissant $\rho=\rho_{\varphi}$ (\cf~\cite[th\'eor\`eme~8]{FRLEquidistribution}).

\begin{rema}\label{rem:energieL}
Soit $L$ une extension finie de~$K$. Comme \`a la remarque~\ref{rem:rhoL}, en tirant en arri\`ere les mesures~$[F]_{v}$, on obtient une famille de mesures~$[F]_{L}$. Elle co\"incide avec la famille~$[F]$ calcul\'ee pour le corps~$L$.

En utilisant la compatibilit\'e du laplacien aux images r\'eciproques, on montre que l'\'energie mutuelle est invariante par extension des scalaires, au sens o\`u, pour toutes mesures ad\'eliques~$\rho$ et~$\rho'$, on a
\[ \la \rho,\rho'\ra =  \la \rho_{L},\rho_{L}'\ra\]
et
\[h_{\rho_{L}}(F) = h_{\rho}(F).\]
Par cons\'equent, dans la suite du texte, nous nous autoriserons parfois \`a \'ecrire des quantit\'es $\la \rho,\rho'\ra$ ou $h_{\rho}(F)$ pour des mesures et des ensembles d\'efinis sur~$\Qbar$, sans pr\'eciser de corps de nombres de d\'efinition.
\end{rema}

\subsection{Segment contre segment}\label{sec:segmentsegment}

Soit $k$ un corps valu\'e ultram\'etrique complet. Nous calculons ici des \'energies mutuelles de mesures associ\'ees \`a des segments contenus dans~$\EP{1}{k}$.

\begin{nota}
On note $\Seg_{k}$ l'ensemble des segments de~$\EP{1}{k}$ dont les extr\'emit\'es sont des points de type~2 ou~3. 
\end{nota}

\begin{nota}\label{nota:muI}
Soit $I\in \Seg_{k}$. On d\'efinit une mesure~$\mu_{I}$ sur~$\EP{1}{k}$ de la fa\c con suivante.
\begin{enumerate}[i)]
\item Si $I$ est un singleton~$\{\xi\}$, $\mu_{I}$ est la mesure de Dirac au point~$\xi$.

\item Si $I$ n'est pas un singleton, $\mu_{I}$ est l'image de la mesure de Lebesgue sur~$I$ de masse totale~1.
\end{enumerate}
\end{nota}

\begin{defi}
Soient $I_{1},I_{2} \in \Seg_{k}$. 

On dit que~$I_{1}$ et~$I_{2}$ sont \emph{align\'es} s'ils sont contenus dans un m\^eme segment.

On dit que~$I_{1}$ et~$I_{2}$ sont \emph{embo\^itables} si $I_{1}\cap I_{2}$ est un singleton et $I_{1} \cup I_{2}$ est un segment.
\end{defi}

Le r\'esultat suivant se d\'emontre par un calcul direct.

\begin{lemm}\label{lem:sigmaalphars}
Soient $\alpha\in k$ et $r,s \in \R_{>0}$ avec $r\le s$. Consid\'erons la fonction $\sigma_{\alpha,r,s}\colon \E{1}{k}\to \R$ d\'efinie, pour $z\in \E{1}{k}$, par 
\[\sigma_{\alpha,r,s}(z)  = \begin{cases}
\log(\frac sr)\log(\abs{z-\alpha})  & \textrm{si } \abs{z-\alpha}\ge s~;\\
\frac12\log(\abs{z-\alpha})^2 -\log(r)\log(\abs{z-\alpha}) + \frac12\log(s)^2& \textrm{si } r\le\abs{z-\alpha}\le s~;\\
 \frac12\log(s)^2-\frac12\log(r)^2 & \textrm{si } \abs{z-\alpha}\le r.
\end{cases}
\]
Alors, $\sigma_{\alpha,r,s}$ est sous-harmonique continue sur~$\E{1}{k}$ et on a
\[\Delta(\sigma_{\alpha,r,s}) = \log\big(\frac sr\big) \, \mu_{\intff{\eta_{\alpha,r},\eta_{\alpha,s}}}.\]
\qed
\end{lemm}

Lorsque $k$~est alg\'ebriquement clos, on peut, par un changement de variable, ramener tout segment $I \in \Seg_{k}$ \`a un segment de la forme consid\'er\'ee dans le lemme~\ref{lem:sigmaalphars}, et donc trouver un potentiel de~$\mu_{I}$ satisfaisant des propri\'et\'es analogues.

\begin{nota}
Supposons que $k$ est alg\'ebriquement clos. Soit $I \in \Seg_{k}$. On note $\sigma_{I} \colon \E{1}{k} \to \R$ l'unique fonction sous-harmonique continue telle que
\[ \Delta(\sigma_{I}) = \ell(I)\mu_{I}\]
et qui co\"incide avec $\ell(I) \log(\abs{z})$ au voisinage de l'infini.
\end{nota}

\begin{lemm}
Soient $I_{a},I_{b}\in \Seg_{k}$. Alors, la quantit\'e 
\[\la \mu_{I_{a}},\mu_{I_{b}} \ra = \frac12 (\mu_{I_{a}}-\mu_{I_{b}}, \mu_{I_{a}}-\mu_{I_{b}})\] 
est bien d\'efinie.

En outre, si $k$ est alg\'ebriquement clos, on a
\[\la \mu_{I_{a}},\mu_{I_{b}} \ra =  \frac12\int_{\E{1}{k}} \Big(\frac{1}{\ell(I_{a})}\sigma_{I_{a}}-\frac{1}{\ell(I_{b})}\sigma_{I_{b}}\Big) \diff (\mu_{I_{b}}-\mu_{I_{a}}).\]
\end{lemm}
\begin{proof}
Par changement de base, on se ram\`ene au cas o\`u $k$ est alg\'ebriquement clos. Le r\'esultat d\'ecoule alors du fait que $\mu_{I_{a}}-\mu_{I_{b}} = \Delta(\sigma_{I_{a}}-\sigma_{I_{b}})$ et des consid\'erations expos\'ees \`a la section~\ref{sec:mesuresdef}.
\end{proof}

\begin{nota}
Pour $I_{a},I_{b}\in \Seg_{k}$, on pose 
\[ E(I_{a},I_{b}) := \la \mu_{I_{a}},\mu_{I_{b}} \ra.\]
\end{nota}

Le reste de cette section est consacr\'e \`a la d\'emonstration d'une formule explicite permettant de calculer la quantit\'e $E(I_{a},I_{b})$ en fonction de la position relative des segments et de leur longueur.

\begin{lemm}\label{lem:sommesigma}
Soient $I_{a},I_{b} \in \Seg_{k}$ des segments embo\^itables. Alors, on a 
\[\ell(I_{a} \cup I_{b})\mu_{I_{a}\cup I_{b}} =\ell(I_{a}) \mu_{I_{a}} + \ell(I_{b})\mu_{I_{b}}.\]
En outre, si $k$ est alg\'ebriquement clos, on a 
\[\sigma_{I_{a}\cup I_{b}} = \sigma_{I_{a}}+\sigma_{I_{b}}.\]
\qed
\end{lemm}

\begin{nota}
Pour $I^\ast_{\times} \in \Seg_{k}$, on pose $\ell^\ast_{\times} := \ell(I^\ast_{\times})$.
\end{nota}

\begin{lemm}\label{lem:Eunion} 
Soient $I_{a}$, $I'_{b},I''_{b}$ tels que $I'_{b}$ et $I''_{b}$ soient embo\^itables. Posons $I_{b} := I'_{b}\cup I''_{b}$.
Alors, on a
\[ E(I_{a},I_{b}) = \frac{\ell'_{b}}{\ell_{b}}\, E(I_{a},I'_{b}) + \frac{\ell''_{b}}{\ell_{b}}\, E(I_{a},I''_{b}) - \frac{\ell'_{b} \ell''_{b}}{\ell_{b}^2} \, E(I'_{b},I''_{b}).\]
\end{lemm}
\begin{proof}
On peut supposer que $k$ est alg\'ebriquement clos. D'apr\`es le lemme~\ref{lem:sommesigma},
on a $\ell_{b}\mu_{I_{b}} =\ell'_{b} \mu_{I'_{b}} + \ell''_{b}\mu_{I''_{b}}$ et $\sigma_{I_{b}} = \sigma_{I'_{b}}+\sigma_{I''_{b}}$. 

Nous supprimerons dor\'enavant les~$I$ dans les indices et \'ecrirons~$\sigma_{a}$ au lieu de~$\sigma_{I_{a}}$, $\mu'_{b}$ au lieu de~$\mu_{I'_{b}}$, etc.

Par d\'efinition, on a 
\begin{align*}
2E(I_{a},I_{b}) 
&= \int_{\E{1}{k}} \Big(\frac{1}{\ell_{a}}\sigma_{a}-\frac{1}{\ell_{b}}\sigma_{b}\Big) \diff (\mu_{b}-\mu_{a})\\
&= \frac1{\ell_{a}} \int_{\E{1}{k}} \sigma_{a}\diff \mu_{b} + \frac1{\ell_{b}} \int_{\E{1}{k}} \sigma_{b}\diff \mu_{a}\\
&\quad -\frac1{\ell_{a}}\int_{\E{1}{k}} \sigma_{a}\diff \mu_{a} - \frac1{\ell_{b}}\int_{\E{1}{k}} \sigma_{b}\diff \mu_{b}
\end{align*}
On en d\'eduit que
\begin{align*}
2E(I_{a},I_{b}) 
&= \frac{\ell'_{b}}{\ell_{a}\ell_{b}} \int_{\E{1}{k}} \sigma_{a}\diff \mu'_{b}+  \frac{\ell''_{b}}{\ell_{a}\ell_{b}} \int_{\E{1}{k}} \sigma_{a}\diff \mu''_{b}\\
&\quad + \frac1{\ell_{b}} \int_{\E{1}{k}} \sigma'_{b}\diff \mu_{a} +  \frac1{\ell_{b}} \int_{\E{1}{k}} \sigma''_{b}\diff \mu_{a}\\
&\quad -\frac1{\ell_{a}}\int_{\E{1}{k}} \sigma_{a}\diff \mu_{a} - \frac{\ell'_{b}}{\ell_{b}^2}\int_{\E{1}{k}} \sigma'_{b}\diff \mu'_{b} - \frac{\ell''_{b}}{\ell_{b}^2}\int_{\E{1}{k}} \sigma''_{b}\diff \mu''_{b}\\
&\quad  - \frac{\ell''_{b}}{\ell_{b}^2}\int_{\E{1}{k}} \sigma'_{b}\diff \mu''_{b} - \frac{\ell'_{b}}{\ell_{b}^2}\int_{\E{1}{k}} \sigma''_{b}\diff \mu'_{b}.
\end{align*}
D'autre part, on a 
\begin{align*}
2\frac{\ell'_{b}}{\ell_{b}} E(\mu_{a},\mu'_{b}) 
&= \frac{\ell'_{b}}{\ell_{a}\ell_{b}} \int_{\E{1}{k}} \sigma_{a}\diff \mu'_{b} + \frac1{\ell_{b}} \int_{\E{1}{k}} \sigma'_{b}\diff \mu_{a}\\
&\quad - \frac{\ell'_{b}}{\ell_{a}\ell_{b}} \int_{\E{1}{k}} \sigma_{a}\diff \mu_{a} - \frac1{\ell_{b}}\int_{\E{1}{k}} \sigma'_{b}\diff \mu'_{b},
\end{align*}
\begin{align*}
2\frac{\ell''_{b}}{\ell_{b}} E(\mu_{a},\mu''_{b}) 
&= \frac{\ell''_{b}}{\ell_{a}\ell_{b}} \int_{\E{1}{k}} \sigma_{a}\diff \mu''_{b} + \frac1{\ell_{b}} \int_{\E{1}{k}} \sigma''_{b}\diff \mu_{a}\\
&\quad - \frac{\ell''_{b}}{\ell_{a}\ell_{b}} \int_{\E{1}{k}} \sigma_{a}\diff \mu_{a} - \frac1{\ell_{b}}\int_{\E{1}{k}} \sigma''_{b}\diff \mu''_{b}
\end{align*}
et
\begin{align*}
2\frac{\ell'_{b}\ell''_{b}}{\ell_{b}^2} E(\mu'_{b},\mu''_{b}) 
&= \frac{\ell''_{b}}{\ell_{b}^2}\int_{\E{1}{k}} \sigma'_{b}\diff \mu''_{b} +  \frac{\ell'_{b}}{\ell_{b}^2}\int_{\E{1}{k}} \sigma''_{b}\diff \mu'_{b}\\
&\quad - \frac{\ell''_{b}}{\ell_{b}^2} \int_{\E{1}{k}} \sigma'_{b}\diff \mu'_{b}  - \frac{\ell'_{b}}{\ell_{b}^2}  \int_{\E{1}{k}} \sigma''_{b}\diff \mu''_{b}.
\end{align*}
Le r\'esultat s'en d\'eduit.
\end{proof}

\begin{lemm}\label{lem:Ealigne}
Soient $I_{a},I_{b} \in \Seg_{k}$ des segments align\'es tels que $\sharp (I_{a}\cap I_{b})\le 1$. Alors, on a
\[E(I_{a},I_{b})=\frac16 \, \ell_{a} + \frac16\,\ell_{b} + \frac12\, d_{ab},\]
o\`u $d_{ab}$ d\'esigne la distance de~$I_{a}$ \`a~$I_{b}$.
\end{lemm}
\begin{proof}
On peut supposer que $k$ est alg\'ebriquement clos. Quitte \`a changer de coordonn\'ee, on peut supposer qu'il existe $\alpha\in k$ et $r,s,t,u \in \R_{>0}$ avec $r\le s\le t\le u$ tels que $I_{a}=\intff{\eta_{\alpha,r},\eta_{\alpha,s}}$ et $I_{b}=\intff{\eta_{\alpha,t},\eta_{\alpha,u}}$. D\'efinissons $\sigma_{\alpha,r,s}$ et $\sigma_{\alpha,t,u}$ comme dans le lemme~\ref{lem:sigmaalphars}. On a alors 
\begin{align*}
2E(I_{a},I_{b}) &= \int_{\E{1}{k}} \Big(\frac{1}{\ell_{a}}\sigma_{\alpha,r,s}-\frac{1}{\ell_{b}}\sigma_{\alpha,t,u}\Big) \diff (\mu_{I_{b}}-\mu_{I_{a}})\\
&= \frac1{\ell_{a}}  \int_{\E{1}{k}} \sigma_{\alpha,r,s} \diff \mu_{I_{b}} + \frac1{\ell_{b}}  \int_{\E{1}{k}} \sigma_{\alpha,t,u} \diff \mu_{I_{a}}\\
&\quad - \frac1{\ell_{a}}  \int_{\E{1}{k}} \sigma_{\alpha,r,s} \diff \mu_{I_{a}} - \frac1{\ell_{b}}  \int_{\E{1}{k}} \sigma_{\alpha,t,u} \diff \mu_{I_{b}}.
\end{align*}
Un calcul direct utilisant l'expression explicite de~$\sigma_{\alpha,r,s}$ montre que l'on a
\begin{align*} 
\int_{\E{1}{k}} \sigma_{\alpha,r,s} \diff (\ell_{a}\mu_{I_{a}})&= \int_{\log(r)}^{\log(s)} \Big( \frac12 x^2-\log(r)x + \frac12 \log(s)^2\Big) \diff x\\
 &=\frac23 \log(s)^3+\frac13\log(r)^3 -\log(r)\log(s)^2\\
&= \frac13
(2\log(s)+\log(r))(\log(s)-\log(r))^2,
\end{align*}
d'o\`u
\begin{align*}
\frac1{\ell_{a}}  \int_{\E{1}{k}} \sigma_{\alpha,r,s} \diff \mu_{I_{a}} 
&= \frac23 \log(s)+\frac13\log(r).
\end{align*}
De m\^eme, on a 
\[ \frac1{\ell_{b}}  \int_{\E{1}{k}} \sigma_{\alpha,t,u} \diff \mu_{I_{b}} =  \frac23 \log(u) +\frac13 \log(t).\]
On a \'egalement 
\begin{align*} 
\frac1{\ell_{a}} \int_{\E{1}{k}} \sigma_{\alpha,r,s} \diff \mu_{I_{b}} &= \frac{1}{\ell_{a}\ell_{b}} \int_{\log(t)}^{\log(u)} \ell_{a} x\diff x\\
&= \frac12 (\log(u)+\log(t)),
\end{align*}
et finalement
\begin{align*} 
\frac1{\ell_{b}}  \int_{\E{1}{k}} \sigma_{\alpha,t,u} \diff \mu_{I_{a}} &=\frac1{\ell_{a}\ell_{b}}  \int_{\log(r)}^{\log(s)} \frac12(\log(u)^2 - \log(t)^2) \diff x\\
&= \frac12 (\log(u)+\log(t)).
\end{align*}
En sommant les diff\'erents termes, on obtient
\begin{align*}
2E(I_{a},I_{b}) &= -\frac23 \log(s) -\frac13\log(r) +\frac13\log(u)+\frac23\log(t)\\
&=\frac13\, \ell_{a} + \frac13\,\ell_{b} + d_{ab}.
\end{align*}
\end{proof}

Commen\c cons par calculer $E(I_{a},I_{b})$ dans le cas o\`u l'intersection $I_{a}\cap I_{b}$ est petite, c'est-\`a-dire vide ou un singleton. On peut alors \'ecrire de fa\c con unique $I_{a} = I'_{a} \cup I''_{a}$ et $I_{b} = I'_{b} \cup I''_{b}$ avec $I'_{a}\cap I''_{a} = \{z_{a}\}$, $I'_{b}\cap I''_{b} = \{z_{b}\}$ et $d_{ab} = \ell(\intff{z_{a},z_{b}})$, o\`u~$d_{ab}$ d\'esigne la distance de~$I_{a}$ \`a~$I_{b}$ (\cf~figure~\ref{fig:intersectionpetite}). 

\begin{figure}[!h]
\centering
\begin{tikzpicture}
\draw  (-4.5,2.2)-- (-3.9,1.6);
\draw  (-3.9,1.6)-- (-4.5,0.1);
\draw (-1.8,1.6)-- (-0.8,2.1);
\draw  (-1.8,1.6)-- (-1.3,0.6);
\draw [<->,dashed] (-3.8,1.6) -- (-1.9,1.6);
\begin{scriptsize}
\draw[fill=black] (-3.9,1.6) circle (1pt);
\draw (-3.8,1.9) node {$z_a$};
\draw (-4.5,1.8) node {$I''_a$};
\draw (-4.6,0.7) node {$I'_a$};
\draw[fill=black] (-1.8,1.6) circle (1pt);
\draw (-1.8,1.9) node {$z_b$};
\draw (-0.8,1.8) node {$I''_b$};
\draw (-1.2,1) node {$I'_b$};
\draw (-2.8,1.4) node {$d_{ab}$};
\end{scriptsize}
\end{tikzpicture}
\caption{Segments~$I_{a}$ et $I_{b}$ disjoints.}\label{fig:intersectionpetite}
\end{figure}
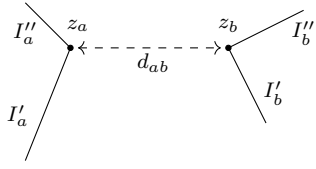

\begin{theo}\label{thm:Eintersectionvide}
Soient $I_{a},I_{b}\in \Seg_{k}$ tels que $\sharp (I_{a}\cap I_{b}) \le 1$. 
Alors, avec les notations de la figure~\ref{fig:intersectionpetite}, on a
\[E(I_{a},I_{b}) = \frac16\, \ell_{a} + \frac16\, \ell_{b} +\frac12 \, d_{ab}  - \frac12 \frac{\ell'_{a} \ell''_{a}}{\ell_{a}} -\frac12 \frac{\ell'_{b} \ell''_{b}}{\ell_{b}}.\]

En particulier, on a
\[E(\mu_{a},\mu_{b}) \ge \frac1{24}\, \ell_{a} + \frac1{24}\, \ell_{b} + \frac12 \,d_{ab}.\]
\end{theo}
\begin{proof}
D'apr\`es les lemmes~\ref{lem:Eunion} et~\ref{lem:Ealigne}, on a
\begin{align*}
E(I'_{a},I_{b}) &= E(I'_{a}, I'_{b}\cup I''_{b})\\
&=  \frac{\ell'_{b}}{\ell_{b}}\, E(I'_{a},I'_{b}) + \frac{\ell''_{b}}{\ell_{b}}\, E(I'_{a},I''_{b}) - \frac{\ell'_{b} \ell''_{b}}{\ell_{b}^2} \, E(I'_{b},I''_{b})\\
&=  \frac{\ell'_{b}}{\ell_{b}} \big( \frac16\, \ell'_{a} + \frac16\, \ell'_{b} + \frac12\,d_{ab}\big) +  \frac{\ell''_{b}}{\ell_{b}} \big( \frac16\, \ell'_{a} + \frac16\, \ell''_{b} + \frac12 \,d_{ab}\big)  - \frac{\ell'_{b} \ell''_{b}}{\ell_{b}^2} \big( \frac16\, \ell_{b})\\
& =  \frac16\, \ell'_{a}  + \frac12\,d_{ab} + \frac1{6\ell_{b}} \, ({\ell'_{b}}^2+{\ell''_{b}}^2 -\ell'_{b}\ell''_{b})\\
& = \frac16\, \ell'_{a}  +\frac12\, d_{ab} +  \frac13\, \ell_{b}  -\frac12 \frac{\ell'_{b} \ell''_{b}}{\ell_{b}} 
\end{align*}
De m\^eme, on a 
\[E(I''_{a},I_{b}) =  \frac16\, \ell''_{a}  +\frac12\, d_{ab} +  \frac16\, \ell_{b}  - \frac12 \frac{\ell'_{b} \ell''_{b}}{\ell_{b}}. 
\]
Finalement, on calcule
\begin{align*}
E(I_{a},I_{b}) &= E(I'_{a}\cup I''_{a}, I_{b})\\
&=  \frac{\ell'_{a}}{\ell_{a}}\, E(I'_{a},I_{b}) + \frac{\ell''_{a}}{\ell_{a}}\, E(I''_{a},I_{b}) - \frac{\ell'_{a} \ell''_{a}}{\ell_{a}^2} \, E(I'_{a},I''_{a})\\
&=  \frac{\ell'_{a}}{\ell_{a}} \big(  \frac16\, \ell'_{a}  + \frac12\,d_{ab} +  \frac12\, \ell_{b}  -\frac12 \frac{\ell'_{b} \ell''_{b}}{\ell_{b}} \big) +  \frac{\ell''_{a}}{\ell_{a}} \big(  \frac16\, \ell''_{a}  + \frac12\,d_{ab} +  \frac16\, \ell_{b}  -\frac12 \frac{\ell'_{b} \ell''_{b}}{\ell_{b}}  \big)  - \frac12 \frac{\ell'_{a} \ell''_{a}}{\ell_{a}^2} \big( \frac16\, \ell_{a})\\
& = \frac12\,d_{ab} +  \frac16\, \ell_{b} -\frac12 \frac{\ell'_{b} \ell''_{b}}{\ell_{b}} + \frac1{6\ell_{a}}\, ({\ell'_{a}}^2+{\ell''_{a}}^2 -\ell'_{a}\ell''_{a})\\
& = \frac12\,d_{ab} +  \frac16\, \ell_{b} -\frac12 \frac{\ell'_{b} \ell''_{b}}{\ell_{b}} + \frac16\, \ell_{a} -\frac12 \frac{\ell'_{a} \ell''_{a}}{\ell_{a}} .
\end{align*}

La minoration finale se d\'emontre \`a l'aide des in\'egalit\'es $\ell'_{a} \ell''_{a} \le \frac14\,\ell_{a}^2$ et $\ell'_{b} \ell''_{b} \le \frac14\,\ell_{b}^2$.  
\end{proof}

Pla\c cons-nous maintenant dans le cas o\`u les segments se rencontrent. Posons $I_{ab} := I_{a}\cap I_{b}$. On peut alors \'ecrire de fa\c con unique $I_{a} = I'_{a} \cup I_{ab} \cup I''_{a}$ et $I_{b} = I'_{b} \cup I_{ab} \cup I''_{b}$, o\`u les segments $I'_{a}$ et~$I_{ab}$, $I_{ab}$ et~$I''_{a}$, $I'_{b}$ et~$I_{ab}$, $I_{ab}$ et~$I''_{b}$ sont embo\^it\'es et~$I'_{a}$ et~$I'_{b}$ sont situ\'es du m\^eme c\^ot\'e de~$I_{ab}$ (\cf~figure~\ref{fig:intersectionnonvide}). 

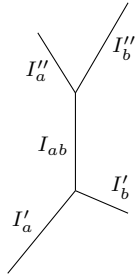
\begin{figure}[!h]
\centering
\begin{tikzpicture}
\draw  (-3.6,2.5)-- (-3.6,1.2);
\draw  (-3.6,1.2)-- (-4.5,0.1);
\draw(-3.6,2.5)-- (-4.1,3.3);
\draw  (-3.6,2.5)-- (-2.9,3.7);
\draw  (-3.6,1.2)-- (-2.9,0.9);
\begin{scriptsize}
\draw(-3.9,1.8) node {$I_{ab}$};
\draw (-4.3,0.8) node {$I'_a$};
\draw (-4.1,2.8) node {$I''_a$};
\draw(-2.95,3.1) node {$I''_b$};
\draw (-3,1.25) node {$I'_b$};
\end{scriptsize}
\end{tikzpicture}
\caption{Segments~$I_{a}$ et $I_{b}$ qui se rencontrent.}\label{fig:intersectionnonvide}
\end{figure}

\begin{lemm}\label{lem:EI1dansI2}
Soient $I_{a},I_{b} \in \Seg_{k}$ tels que $I_{a}\subset I_{b}$. Alors, avec les notations de la figure~\ref{fig:intersectionnonvide}, on a
\[E(I_{a},I_{b}) = \frac16 \ell_{b} - \frac13 \ell_{a} + \frac16\frac{\ell_{a}^2}{\ell_{b}} - \frac12\frac{\ell'_{b}\ell''_{b}}{\ell_{b}}.\] 
\end{lemm}
\begin{proof}
D'apr\`es le th\'eor\`eme~\ref{thm:Eintersectionvide}, on a
\[E(I_{a},I'_{b}) = \frac16 (\ell_{a} +  \ell'_{b}),\ E(I_{a},I''_{b}) = \frac16 (\ell_{a} +  \ell''_{b}) \textrm{ et }
E(I'_{b}\cup I_{a},I''_{b}) = \frac16\ell_{b}.\]
D'apr\`es le lemme~\ref{lem:Eunion}, on a 
\begin{align*}
E(I_{a},I'_{b}\cup I_{a}) &= \frac{\ell'_{b}}{\ell'_{b}+\ell_{a}} \, E(I_{a},I'_{b}) + \frac{\ell_{a}}{\ell'_{b}+\ell_{a}} \, E(I_{a},I_{a}) - \frac{\ell'_{b} \ell_{a}}{(\ell'_{b}+\ell_{a})^2} \, E(I'_{b},I_{a})\\
&= \frac{\ell'_{b}(\ell'_{b}+\ell_{a}) - \ell'_{b} \ell_{a}}{(\ell'_{b}+\ell_{a})^2} \frac16 (\ell_{a} +  \ell'_{b})\\
&=\frac16 \frac{{\ell'_{b}}^2}{\ell_{a}+\ell'_{b}},
\end{align*}
puis
\begin{align*}
E(I_{a},I_{b}) &= E(I_{a},(I'_{b}\cup I_{a})\cup I''_{b})\\
&= \frac{\ell'_{b}+\ell_{a}}{\ell_{b}}\, E(I_{a},I'_{b}\cup I_{a}) + \frac{\ell''_{b}}{\ell_{b}}\, E(I_{a},I''_{b}) -  \frac{(\ell'_{b}+\ell_{a})\ell''_{b}}{\ell_{b}^2} \, E(I'_{b}\cup I_{a},I''_{b})\\
&= \frac16 \frac{{\ell'_{b}}^2}{\ell_{b}} + \frac16 \frac{\ell''_{b}}{\ell_{b}}\, (\ell_{a} + \ell''_{b}) - \frac16 \frac{(\ell'_{b}+\ell_{a})\ell''_{b}}{\ell_{b}}\\
&= \frac16 \frac{{\ell'_{b}}^2+{\ell''_{b}}^2-\ell'_{b}\ell''_{b}}{\ell_{b}}\\
&=\frac16 \frac{(\ell'_{b}+\ell_{a}+\ell''_{b})^2-2\ell_{a}(\ell'_{b}+\ell_{a}+\ell''_{b}) + \ell_{a}^2-3\ell'_{b}\ell''_{b}}{\ell_{b}}\\
&= \frac16 \ell_{b} - \frac13 \ell_{a} + \frac16\frac{\ell_{a}^2}{\ell_{b}} - \frac12\frac{\ell'_{b}\ell''_{b}}{\ell_{b}}. 
\end{align*}
\end{proof}

\begin{lemm}\label{lem:Ealignessansinclusion}
Soient $I_{a},I_{b} \in \Seg_{k}$ tels que $I_{a}\cap I_{b} \ne \emptyset$, $I_{a} \not\subset I_{b}$ et $I_{b} \not\subset I_{a}$. Alors, avec les notations de la figure~\ref{fig:intersectionnonvide}, on a
\[E(I_{a},I_{b}) = \frac16 \ell_{a} + \frac16\ell_{b} - \frac12\ell_{ab} +\frac16 \frac{\ell_{ab}^3}{\ell_{a}\ell_{b}}.\] 
\end{lemm}
\begin{proof}
On peut supposer que $I''_{a} = I'_{b} = \emptyset$. D'apr\`es les lemmes~\ref{lem:EI1dansI2}et~\ref{lem:Ealigne}, on a
\[E(I_{ab},I_{a}) = \frac16 \ell_{a} -\frac13  \ell_{ab} + \frac16 \frac{\ell_{ab}^2}{\ell_{a}},\ 
E(I_{a},I''_{b}) = \frac16 (\ell_{a} +  \ell''_{b}) \textrm{ et }
E(I_{ab},I''_{b}) = \frac16\ell_{b}. \]
D'apr\`es le lemme~\ref{lem:Eunion}, on a 
\begin{align*}
E(I_{a},I_{ab}\cup I''_{b}) &= \frac{\ell_{ab}}{\ell_{b}} \, E(I_{a},I_{ab}) + \frac{\ell''_{b}}{\ell_{b}} \, E(I_{a},I''_{b}) - \frac{\ell_{ab} \ell''_{b}}{\ell_{b}^2} \, E(I_{ab},I''_{b})\\
&= \frac1{6\ell_{b}} \Big(\ell_{ab}\big(\ell_{a}-2\ell_{ab}+\frac{\ell_{ab}^2}{\ell_{a}}\big) + \ell''_{b}(\ell_{a}+\ell''_{b}) - \ell_{ab}\ell''_{b}\Big)\\
&= \frac1{6\ell_{b}} \Big( \ell_{a}\ell_{b} + \ell_{b}^2 -3 \ell_{ab}^2 - 3\ell_{ab}\ell''_{b} + \frac{\ell_{ab}^3}{\ell_{a}}\Big)\\
&=\frac16 \ell_{a} + \frac16\ell_{b}- \frac12\ell_{ab} +\frac16 \frac{\ell_{ab}^3}{\ell_{a}\ell_{b}}.
\end{align*}
\end{proof}

\begin{lemm}\label{lem:Edepasseduncote}
Soient $I_{a},I_{b} \in \Seg_{k}$ tels que $I_{a}\cap I_{b} \ne \emptyset$. Avec les notations de la figure~\ref{fig:intersectionnonvide}, supposons que $I''_{a} = \emptyset$. Alors, on a
\[E(I_{a},I_{b}) = 
\frac16 \ell_{a} + \frac16\ell_{b}-\frac12\ell_{ab} +\frac16 \frac{\ell_{ab}^3}{\ell_{a}\ell_{b}} - \frac12\frac{\ell'_{b}\ell''_{b}}{\ell_{b}} - \frac12\frac{\ell'_{a}\ell'_{b}\ell_{ab}}{\ell_{a}\ell_{b}}.
\] 
\end{lemm}
\begin{proof}
D'apr\`es le th\'eor\`eme~\ref{thm:Eintersectionvide} et le lemme~\ref{lem:Ealignessansinclusion}, on a
\[E(I_{a},I'_{b}) = \frac16 \ell_{a} +\frac16  \ell'_{b} - \frac12 \frac{\ell'_{a}\ell_{ab}}{\ell_{a}}, \ E(I'_{b},I_{ab}\cup I''_{b}) = \frac16\ell_{b}\]
et
\[
E(I_{a},I_{ab}\cup I''_{b}) = \frac16 \ell_{a} + \frac16 (\ell_{ab}+\ell''_{b}) - \frac12\ell_{ab} +\frac16 \frac{\ell_{ab}^3}{\ell_{a}(\ell_{ab}+\ell''_{b})}.\]
D'apr\`es le lemme~\ref{lem:Eunion}, on a 
\begin{align*}
E(I_{a},I_{b}) &= \frac{\ell'_{b}}{\ell_{b}} \, E(I_{a},I'_{b}) + \frac{\ell_{ab}+\ell''_{b}}{\ell_{b}} \, E(I_{a},I_{ab}\cup I''_{b}) - \frac{\ell'_{b}(\ell_{ab}+\ell''_{b})}{\ell_{b}^2} \, E(I'_{b},I_{ab}\cup I''_{b})\\
&= \frac1{6\ell_{b}} \Big(\ell'_{b}\ell_{a}+{\ell'_{b}}^2 -  3\frac{\ell'_{a}\ell'_{b}\ell_{ab}}{\ell_{a}} + (\ell_{ab}+\ell''_{b})\ell_{a} + (\ell_{ab}+\ell''_{b})^2 - 3(\ell_{ab}+\ell''_{b})\ell_{ab} +\frac{\ell_{ab}^3}{\ell_{a}}\\
&\quad - \ell'_{b}(\ell_{ab}+\ell''_{b})\Big)\\
&= \frac1{6\ell_{b}} \Big( \ell_{b}\ell_{a} + \ell_{b}^2 -3\ell_{ab}\ell_{b}- 3\ell'_{b}\ell''_{b} + \frac{\ell_{ab}^3}{\ell_{a}} - 3\frac{\ell'_{a}\ell'_{b}\ell_{ab}}{\ell_{a}}\Big)\\
\end{align*}
\end{proof}

\begin{theo}\label{thm:Ecasgeneral}
Soient $I_{a},I_{b} \in \Seg_{k}$ tels que $I_{a}\cap I_{b} \ne \emptyset$. Alors, avec les notations de la figure~\ref{fig:intersectionnonvide}, on a
\[E(I_{a},I_{b}) = 
\frac16 \ell_{a} + \frac16\ell_{b}-\ell_{ab} +\frac16 \frac{\ell_{ab}^3}{\ell_{a}\ell_{b}} -\frac12 \frac{\ell'_{a}\ell''_{a}}{\ell_{a}} - \frac12 \frac{\ell'_{b}\ell''_{b}}{\ell_{b}} - \frac12 \frac{(\ell'_{a}\ell'_{b}+\ell''_{a}\ell''_{b})\ell_{ab}}{\ell_{a}\ell_{b}}.\] 
\end{theo}
\begin{proof}
D'apr\`es le lemme~\ref{lem:Edepasseduncote} et le th\'eor\`eme~\ref{thm:Eintersectionvide}, on a
\[E(I'_{a}\cup I_{ab},I_{b}) = \frac16 (\ell'_{a}+\ell_{ab}) +\frac16  \ell_{b} - \frac12\ell_{ab} +\frac16 \frac{\ell_{ab}^3}{(\ell'_{a}+\ell_{ab})\ell_{b}} - \frac12\frac{\ell'_{b}\ell''_{b}}{\ell_{b}}- \frac12\frac{\ell'_{a}\ell'_{b}\ell_{ab}}{(\ell'_{a}+\ell_{ab})\ell_{b}},\]
\[E(I''_{a},I_{b}) = \frac16 \ell''_{a} + \frac16 \ell_{b} - \frac12 \frac{(\ell'_{b}+\ell_{ab})\ell''_{b}}{\ell_{b}} \textrm{ et }
E(I'_{a}\cup I_{ab},I''_{a}) = \frac16\ell_{a}.\]
D'apr\`es le lemme~\ref{lem:Eunion}, on a 
\begin{align*}
E(I_{a},I_{b}) &= \frac{\ell'_{a}+\ell_{ab}}{\ell_{a}} \, E((I'_{a}\cup I_{ab}),I_{b}) + \frac{\ell''_{a}}{\ell_{a}} \, E(I''_{a},I_{b}) - \frac{(\ell'_{a}+\ell_{ab})\ell''_{a}}{\ell_{a}^2} \, E(I'_{a}\cup I_{ab}, I''_{a})\\
&= \frac1{6\ell_{a}} \Big( (\ell'_{a}+\ell_{ab})^2 + (\ell'_{a}+\ell_{ab})\ell_{b} - 3(\ell'_{a}+\ell_{ab})\ell_{ab}  + \frac{\ell_{ab}^3}{\ell_{b}}\\ 
&\quad - 3 \frac{(\ell'_{a}+\ell_{ab})\ell'_{b}\ell''_{b}}{\ell_{b}} - 3\frac{\ell'_{a}\ell'_{b}\ell_{ab}}{\ell_{b}}\\
&\quad + {\ell''_{a}}^2 + \ell''_{a}\ell_{b} - 3 \frac{(\ell'_{b}+\ell_{ab})\ell''_{a}\ell''_{b}}{\ell_{b}} - (\ell'_{a}+\ell_{ab})\ell''_{a}\Big)\\
&= \frac1{6\ell_{a}} \Big( \ell_{a}^2 + \ell_{a}\ell_{b} - 3\ell_{a}\ell_{ab} - 3 \ell'_{a}\ell''_{a} + \frac{\ell_{ab}^3}{\ell_{b}}  - 3 \frac{(\ell'_{a}+\ell_{ab})\ell'_{b}\ell''_{b}}{\ell_{b}} - 3\frac{\ell'_{a}\ell'_{b}\ell_{ab}}{\ell_{b}}\\
&\quad - 3 \frac{\ell'_{b}\ell''_{a}\ell''_{b}}{\ell_{b}} - 3 \frac{\ell_{ab}\ell''_{a}\ell''_{b}}{\ell_{b}}\Big).\\
\end{align*}
Le r\'esultat s'en d\'eduit.
\end{proof}

\begin{coro}\label{cor:Eminoration}
Soient $I_{a},I_{b} \in \Seg_{k}$ tels que $I_{a}\cap I_{b} \ne \emptyset$. Fixons les notations de la figure~\ref{fig:intersectionnonvide}. Pour $i\in \{a,b\}$, posons $m_{i} := \frac12 (\ell_{i} - \ell_{ab})$. On a alors 
\begin{align*}
E(I_{a},I_{b}) &\ge \frac16 \frac{m_{a}^2}{\ell_{a}} + \frac16 \frac{m_{b}^2}{\ell_{b}} - \frac13 \ell_{ab}\frac{m_{a}}{\ell_{a}}\frac{m_{b}}{\ell_{b}}\\
& \qquad {\scriptstyle \Big( = \frac16 \Big( \frac{m_{a}}{\sqrt{\ell_{a}}} - \frac{m_{b}}{\sqrt{\ell_{b}}} \Big)^2 + \frac13 (\sqrt{\ell_{a}\ell_{b}} - \ell_{ab}) \frac{m_{a}}{\ell_{a}}\frac{m_{b}}{\ell_{b}}\Big)}\\
&\ge \frac16 \frac{(\ell_{a}-\ell_{b})^2}{\max(\ell_{a},\ell_{b})}.
\end{align*} 
\end{coro}
\begin{proof}
Pour $i\in \{a,b\}$, posons 
$u_{i} := \ell'_{i} - m_{i}$. On a alors $\ell'_{i} = m_{i} + u_{i}$, $\ell''_{i} = m_{i} - u_{i}$. Posons 
\[ A := \ell_{b} \ell'_{a}\ell''_{a} + \ell_{a}\ell'_{b}\ell''_{b} + \ell_{ab}(\ell'_{a}\ell'_{b} + \ell''_{a}\ell''_{b}).\]
On a alors 
\begin{align*}
A & = \ell_{b}(m_{a}^2-u_{a}^2) + \ell_{a}(m_{b}^2-u_{b}^2) + \ell_{ab}\big((m_{a}+u_{a})(m_{b}+u_{b}) + (m_{a}-u_{a})(m_{b}-u_{b})\big)\\
&= \ell_{b}m_{a}^2+\ell_{a}m_{b}^2+2 \ell_{ab} m_{a}m_{b} + 2 \ell_{ab} u_{a}u_{b} - \ell_{b}u_{a}^2-\ell_{a}u_{b}^2.
\end{align*}
Or, on a $\ell_{ab} \le \ell_{a}$ et $\ell_{ab} \le \ell_{b}$, d'o\`u
\[2 \ell_{ab} u_{a}u_{b} - \ell_{b}u_{a}^2-\ell_{a}u_{b}^2 \le 2 \sqrt{\ell_{a}}u_{a}\sqrt{\ell_{b}}u_{b} - (\sqrt{\ell_{a}} u_{a})^2 - (\sqrt{\ell_{b}} u_{b})^2 \le 0.\]

Fixons dor\'enavant $\ell_{a}$, $\ell_{b}$ et~$\ell_{ab}$, et donc $m_{a}$ et $m_{b}$. La quantit\'e~$A$ est maximale lorsque $u_{a}=u_{b}=0$ et vaut alors 
\[A_{0}:=\ell_{b}m_{a}^2+\ell_{a}m_{b}^2+2 \ell_{ab} m_{a}m_{b}.\]
D'apr\`es le th\'eor\`eme~\ref{thm:Ecasgeneral}, la quantit\'e $E(I_{a},I_{b})$ est donc minimale lorsque $u_{a}=u_{b}=0$ et vaut alors 
\begin{align*}
E_{0} &:= \frac16 \ell_{a} + \frac16 \ell_{b} - \frac12\ell_{ab} + \frac16 \frac{\ell_{ab}^3}{\ell_{a}\ell_{b}} - \frac12\frac{A_{0}}{\ell_{a}\ell_{b}}\\
&=  \frac13 m_{a} + \frac13 m_{b} - \frac16 \ell_{ab} + \frac16 \frac{\ell_{ab}^3}{\ell_{a}\ell_{b}} - 
\frac12\frac{m_{a}^2}{\ell_{a}} - \frac12\frac{m_{b}^2}{\ell_{b}} -  \ell_{ab} \frac{m_{a}m_{b}}{\ell_{a}\ell_{b}}.
\end{align*}
Or, on a 
\begin{align*}
\frac{\ell_{ab}^3}{\ell_{a}\ell_{b}}-\ell_{ab} &=\ell_{ab} \Big(\frac{\ell_{ab}^2}{\ell_{a}\ell_{b}}-1\Big)\\
&= \ell_{ab} \Big( \frac{(\ell_{a}-2m_{a})(\ell_{b}-2m_{b})}{\ell_{a}\ell_{b}}-1\Big)\\
&= 4\ell_{ab}\frac{m_{a}m_{b}}{\ell_{a}\ell_{b}} - 2\ell_{ab}\frac{m_{a}}{\ell_{a}}- 2\ell_{ab}\frac{m_{b}}{\ell_{b}}\\
&= 4\ell_{ab}\frac{m_{a}m_{b}}{\ell_{a}\ell_{b}} - 2(\ell_{a}-2m_{a})\frac{m_{a}}{\ell_{a}}- 2(\ell_{b}-2m_{b})\frac{m_{b}}{\ell_{b}}\\
&= 4\ell_{ab}\frac{m_{a}m_{b}}{\ell_{a}\ell_{b}} - 2m_{a} +4 \frac{m_{a}^2}{\ell_{a}}- 2m_{b} +4\frac{m_{b}^2}{\ell_{b}},
\end{align*}
d'o\`u
\[
E_{0} = \frac16 \frac{m_{a}^2}{\ell_{a}} + \frac16 \frac{m_{b}^2}{\ell_{b}} - \frac13 \ell_{ab}\frac{m_{a}}{\ell_{a}}\frac{m_{b}}{\ell_{b}}.
\]

\medbreak

Pour d\'emontrer la derni\`ere \'egalit\'e, exprimons~$E_{0}$ \`a l'aide de~$\ell_{a}$, $\ell_{b}$ et~$\ell_{ab}$~:
\begin{align*}
3\ell_{a}\ell_{b}E_{0} &= \ell_{b}(\ell_{a}-\ell_{ab})^2+\ell_{a}(\ell_{b}-\ell_{ab})^2 - 2\ell_{ab}(\ell_{a}-\ell_{ab})(\ell_{b}-\ell_{ab})\\
&= -2 \ell_{ab}^3 + 3(\ell_{a}+\ell_{b})\ell_{ab}^2 - 6\ell_{a}\ell_{b}\ell_{ab} +\ell_{a}\ell_{b}(\ell_{a}+\ell_{b}).
\end{align*}
Fixons $\ell_{a}$ et~$\ell_{b}$. L'expression pr\'ec\'edente est un polyn\^ome~$P$ en~$\ell_{ab}$, de d\'eriv\'ee
\[P' = -6\ell_{ab}^2+6(\ell_{a}+\ell_{b})\ell_{ab} -6\ell_{a}\ell_{b} = -6(\ell_{ab}-\ell_{a})(\ell_{ab}-\ell_{b}).\]
En particulier~$P$ est d\'ecroissant sur le segment $[0,\min(\ell_{a},\ell_{b})]$. Puisque~$\ell_{ab}$ appartient \`a ce segment, on en d\'eduit que
\[6\ell_{a}\ell_{b}E_{0} \ge P(\min(\ell_{a},\ell_{b})) = \min(\ell_{a},\ell_{b})\, (\ell_{a}-\ell_{b})^2.\]
Le r\'esultat s'ensuit.
\end{proof}

\begin{coro}\label{cor:Eqrho}
Soient $I_{a},I_{b} \in \Seg_{k}$ tels que $I_{a}\cap I_{b} \ne \emptyset$. Fixons les notations de la figure~\ref{fig:intersectionnonvide}. Soit $\lambda \in \intff{0,\max(\ell_{a}-\ell_{ab},\ell_{b}-\ell_{ab})}$. Soit $\varrho\in \R_{\ge 0}$ tel que $\max(\ell_{a},\ell_{b}) \le \varrho \lambda$. Alors, on a
\[E(I_{a},I_{b}) \ge \frac{\lambda}{48\varrho^2}.\] 
\end{coro}
\begin{proof}
Pour $i\in \{a,b\}$, posons $m_{i} := \frac12 (\ell_{i} - \ell_{ab})$. Quitte \`a \'echanger $I_{a}$ et $I_{b}$, on peut supposer que $\ell_{a}\ge \ell_{b}$, et donc $m_{a}\ge m_{b}$.

$\bullet$ Supposons que $m_{b} \ge \frac12\, m_{a}$.

On a alors $\min(\ell_{a},\ell_{b})=\ell_{b} = \ell_{ab} +2m_{b} \ge \ell_{ab}+ \frac12\, \lambda$, d'o\`u $\sqrt{\ell_{a}\ell_{b}}-\ell_{ab} \ge \frac12\, \lambda$. D'apr\`es le corollaire~\ref{cor:Eminoration}, on a donc
\[E(I_{a},I_{b}) \ge   \frac16 \, \lambda\,  \frac{m_{a}}{\ell_{a}}\frac{m_{b}}{\ell_{b}} \ge \frac{\lambda}{48\varrho^2}.\]

\medbreak

$\bullet$ Supposons que $m_{b} \le \frac12\, m_{a}$.

On a alors $\abs{\ell_{a}-\ell_{b}} = \ell_{a}-\ell_{b} = 2m_{a}-2m_{b} \ge m_{a}\ge \frac12\, \lambda$. D'apr\`es le corollaire~\ref{cor:Eminoration}, on a donc
\[E(I_{a},I_{b}) \ge \frac1{24} \frac{\lambda^2}{\ell_{a}} \ge \frac{\lambda}{24\varrho}  \ge \frac{\lambda}{48\varrho^2}.\]
\end{proof}

\section{Morphismes de Latt\`es}\label{sec:Lattes}

Dans cette section, nous pr\'esentons les mesures dont nous allons calculer les \'energies mutuelles. Ils s'agit de mesures d'\'equilibre associ\'ees \`a des morphismes de Latt\`es, provenant de la multiplication par~2 sur des courbes elliptiques. \JP{ajouter des refs ?}

\subsection{Projections standards}\label{sec:projectionstandard}

Soit~$K$ un corps et soit~$E$ une courbe elliptique sur~$K$.

\begin{defi}
On appelle \emph{projection standard} sur~$E$ tout morphisme $\pi \colon E \to \P^1_{K}$ qui fait commuter un diagramme de la forme
\[\begin{tikzcd}
E \ar[d] \ar[r, "\pi"] & \P^1_{K},\\
E/\{\pm1\} \ar[ur, "\sim"]
\end{tikzcd}\]
o\`u $E \to E/\{\pm1\}$ est la projection canonique et $E/\{\pm1\} \xrightarrow[]{\sim} \P^1_{K}$ est un isomorphisme.
\end{defi}

Une telle projection standard sur~$E$ est un morphisme de degr\'e~2, ramifi\'e en l'ensemble~$E[2]$ des points de 2-torsion de~$E$.

\medbreak

R\'eciproquement, soit~$R$ un ensemble de quatre points distincts de~$\P^1(K)$. Soit $\pi_{R} \colon E_{R} \to \P^1_{K}$ un rev\^etement de degr\'e~2 ramifi\'e en les points de~$R$. En choisissant un point rationnel au-dessus de l'un des points de~$R$, $E_{R}$ acquiert une structure de courbe elliptique. Le morphisme~$\pi_{R}$ est alors un morphisme standard et l'on a 
\[ \pi_{R}(E_{R}[2]) = R.\]

\medbreak

Dans la suite de ce texte, nous passerons donc librement de couples $(E,\pi)$, o\`u $E$ est une courbe elliptique et $\pi$ une projection standard sur~$E$, \`a des quadruplets de points distincts de la droite projective, et r\'eciproqument.

\subsection{Sur un corps valu\'e complet}\label{sec:Lattescorps}

Soit $k$ un corps valu\'e complet de caract\'eristique diff\'erente de~2. 

Soit $\gamma=(\gamma_{1},\gamma_{2},\gamma_{3},\gamma_{4})$ un quadruplet de points distincts de~$\P^1(k)$. Par la construction de la section~\ref{sec:projectionstandard}, on lui associe une courbe elliptique~$E_{\gamma}$ et une projection standard $\pi_{\gamma} \colon E_{\gamma} \to \P^1_{k}$.
Notons~$L_{\gamma}$ le morphisme de Latt\`es associ\'e, c'est-\`a-dire l'unique endomorphisme de~$\P^1_{k}$ qui fait commuter le diagramme
\[\begin{tikzcd}
E_{\gamma} \ar[r, "{[}2{]}"] \ar[d, "\pi_{\gamma}"]& E_{\gamma} \ar[d, "\pi_{\gamma}"]\\
\P^1_{k} \ar[r, "L_{\gamma}"] & \P^1_{k}
\end{tikzcd}.\]
Lorsque $\gamma_{1}=\infty$, $\gamma_{2}=0$ et~$\gamma_{3}=1$, on peut d\'ecrire la courbe elliptique~$E_{\gamma}$ par l'\'equation de Legendre
\[ E_{\gamma} \colon y^2 = x(x-1)(x-\gamma_{4})\]
et le morphisme de Latt\`es~$L_{\gamma}$ par la formule 
\[ L_{\gamma} \colon t\in \P^1_{k} \mapstoo \frac{(t^2-\gamma_{4})^2}{4t(t-1)(t-\gamma_{4})} \in \P^1_{k}.\]

\begin{nota}\label{nota:Lattes}
On note $\mu_{\gamma}$ la mesure d'\'equilibre sur~$\EP{1}{k}$ associ\'ee au morphisme de Latt\`es~$L_{\gamma}^\an$. Elle ne d\'epend que de l'ensemble $\{\gamma_{1},\gamma_{2},\gamma_{3},\gamma_{4}\}$. On la note parfois \'egalement $\mu_{(E_{\gamma},\pi_{\gamma})}$.

Cette mesure poss\`ede un potentiel continu au sens de la d\'efinition~\ref{def:mesureadelique}, i).
\end{nota}

La notation pr\'esent\'ee ici conduit \`a identifier la mesure~$\mu_{\gamma}$ \`a son image r\'eciproque sur~$\EP{1}{k'}$ pour toute extension valu\'ee compl\`ete~$k'$ de~$k$. Aucun probl\`eme n'en r\'esulte (\cf~remarque~\ref{rem:energieL}).

\medbreak

Dans \cite[proposition~5.1]{FRLErgodique}, Ch.~Favre et J.~Rivera--Letelier d\'ecrivent explicitement la mesure~$\mu_{\gamma}$ sur un corps ultram\'etrique lorsque $\gamma_{1}=\infty$, $\gamma_{2}=0$ et~$\gamma_{3}=1$ (voir aussi~\cite[proposition~3.2]{DKY}). On peut toujours se ramener \`a ce cas par une homographie, et donc d\'ecrire~$\mu_{\gamma}$ en g\'en\'eral.

\begin{nota}\label{nota:Igamma}
Supposons que~$k$ est ultram\'etrique. Il existe une permutation~$\sigma$ de~$\{1,2,3,4\}$ telle que $\intff{\gamma_{\sigma(1)}, \gamma_{\sigma(2)}} \cap \intff{\gamma_{\sigma(3)}, \gamma_{\sigma(4)}} \ne \emptyset$ dans~$\EP{1}{K}$. On pose
\[ I_{\gamma} := \intff{\gamma_{\sigma(1)}, \gamma_{\sigma(2)}} \cap \intff{\gamma_{\sigma(3)}, \gamma_{\sigma(4)}}.\]
C'est un segment de~$\EP{1}{k}$ form\'e de points de type~2 ou~3. Il est ind\'ependant du choix de~$\sigma$.
\end{nota}

\begin{exem}\label{ex:Igamma01infini}
Lorsque $\gamma_{1}=\infty$, $\gamma_{2}=0$, $\gamma_{3}=1$ et $\abs{\gamma_{4}}>1$, on a
\[I_{\gamma} = \intff{\eta_{1},\eta_{\abs{\gamma_{4}}}}.\]
\end{exem}

\begin{prop}\label{prop:mugamma}
Supposons que~$k$ est ultram\'etrique et de caract\'eristique r\'esiduelle diff\'erente de~2. Alors, avec la notation~\ref{nota:muI}, on a
\[\mu_{\gamma} = \mu_{I_{\gamma}}.\]
En d'autres termes, la mesure~$\mu_{\gamma}$ est la mesure de masse totale~1 proportionnelle \`a la mesure de Lebesgue sur le segment~$I_{\gamma}$ s'il est non trivial, ou la mesure de Dirac support\'ee par~$I_{\gamma}$. 
\qed
\end{prop}

\begin{coro}\label{cor:mugammabonnered}
Supposons que 
\[\forall i \ne j,\ \abs{\gamma_{i}} = \abs{\gamma_{j}} = \abs{\gamma_{i}-\gamma_{j}}.\]
Alors, on a $I_{\gamma} = \{\eta_{0,1}\}$ et $\mu_{\gamma} = \chi_{0,1}$.
\qed
\end{coro}

Il est utile de remarquer que la longueur de~$I_{\gamma}$ peut \^etre calcul\'ee \`a l'aide d'un \emph{birapport}. Nous noterons 
\[ [\gamma_{1},\gamma_{2},\gamma_{3},\gamma_{4}] := \frac{(\gamma_{3}-\gamma_{1})(\gamma_{4}-\gamma_{2})}{(\gamma_{3}-\gamma_{2})(\gamma_{4}-\gamma_{1})} \in K\setminus\{0,1\}\]
le birapport des points~$\gamma_{1},\gamma_{2},\gamma_{3},\gamma_{4}$ (avec une g\'en\'eralisation convenable de l'expression lorsque l'un des points est \`a l'infini).

Rappelons que, lorsque l'on modifie l'ordre des points $\gamma_{1},\gamma_{2},\gamma_{3},\gamma_{4}$, le birapport~$\beta$ peut prendre n'importe laquelle des six valeurs
\[\beta,  \frac1\beta, 1-\beta, \frac1{1-\beta}, \frac{\beta-1}{\beta}, \frac{\beta}{\beta-1}.\]

\begin{lemm}\label{lem:longueurIgamma}
La longueur de~$I_{\gamma}$ est \'egale \`a 
\[ \ell(I_{\gamma}) =  \max_{\sigma \in S_{4}} (\abs{[\gamma_{\sigma(1)},\gamma_{\sigma(2)},\gamma_{\sigma(3)},\gamma_{\sigma(4)}]}).\]
\end{lemm} 
\begin{proof}
En utilisant l'invariance de la formule \`a d\'emontrer par permutation des~$\gamma_{i}$ et l'invariance du birapport par homographie, on se ram\`ene au cas o\`u $\gamma_{1}=\infty$, $\gamma_{2}=0$, $\gamma_{3}=1$ et $\abs{\gamma_{4}}>1$. Un calcul direct bas\'e sur l'exemple~\ref{ex:Igamma01infini} montre alors que 
\[\ell(I_{\gamma})=\abs{\gamma_{4}} =  \max_{\sigma \in S_{4}} (\abs{[\gamma_{\sigma(1)},\gamma_{\sigma(2)},\gamma_{\sigma(3)},\gamma_{\sigma(4)}]}).\] 
\end{proof}

\begin{rema}\label{rem:mugammaC}
On peut \'egalement d\'ecrire la mesure~$\mu_{\gamma}$ dans le cas archim\'edien. Lorsque $k=\C$, c'est l'image de la mesure de Haar sur~$E_{\gamma}(\C)$ par le morphisme $E_{\gamma}(\C) \to \P^1(\C)$. Cette mesure pr\'esente une densit\'e plus importante au voisinage des points de ramification, ce qui permet de retrouver~$\gamma$. En particulier, \`a la diff\'erence de la situation ultram\'etrique, deux quadruplets distincts donnent lieu \`a deux mesures distinctes. \JP{ref ?} 
\end{rema}

Introduisons \'egalement une notation pour les mesures ad\'eliques form\'ees \`a partir des mesures pr\'ec\'edentes.

\begin{nota}\label{nota:muQ}
Soit $K$ un corps de nombres. Soit $P\in (\P^1(K))^4$ un quadruplet de points distincts. Pour toute place~$v$ de~$K$, $P$ d\'efinit un \'el\'ement~$P_{v}$ de $(\P^1(K_{v}))^4$ et l'on peut consid\'erer la mesure~$\mu_{P_{v}}$ associ\'ee. Posons
\[ \mu_{P} := (\mu_{P_{v}})_{v\in M_{K}}.\]
D'apr\`es le corollaire~\ref{cor:mugammabonnered}, c'est une mesure ad\'elique. 

Par la construction de la section~\ref{sec:projectionstandard}, on associe \`a~$P$ une courbe elliptique~$E_{P}$ et une projection standard $\pi_P \colon E_P \to \P^1_{K}$. Notons~$L_{P}$ le morphisme de Latt\`es sur~$\P^1_{K}$ correspondant. On a alors
\[ \mu_{P} = \rho_{L_{P}},\]
au sens de l'exemple~\ref{ex:rhoR}. On notera encore \'egalement $\mu_{(E_{P},\pi_{P})}$ cette mesure.
\end{nota}

\subsection{En famille}\label{sec:Lattesfamille}

Nous expliquons maintenant comment d\'efinir en famille les mesures associ\'ees aux morphismes de Latt\`es. Consid\'erons la droite affine $\A^1_{\Z[\frac12]}$ avec coordonn\'ee~$\lambda$, son ouvert de Zariski
\[\cZ^0 := \{\lambda \ne 0, 1\} \subset \A^1_{\Zud}\]
et posons
\[\cX^0 := \P^1_{\Zud} \times_{\Zud} \cZ^0.\]
Dans le facteur $\P^1_{\Zud}$, fixons des coordonn\'ees homog\`enes~$[T_{1} \mathbin: T_{2}]$. Posons $t:=\frac{T_{1}}{T_{2}}$ et $u:=\frac{T_{2}}{T_{1}}$. 

Consid\'erons le rev\^etement de $\A^1_{\Zud}\times_{\Zud} \cZ^0 \subset \cX^0$ d\'efini par l'\'equation $v^2 = t(t-1)(t-\lambda)$ et le rev\^etement de $(\P^1_{\Zud} \setminus\{0\})\times_{\Zud} \cZ^0 \subset \cX^0$ d\'efini par l'\'equation $w^2 = u(u-1)(u-\lambda^{-1})$. Ces rev\^etements se recollent par les identifications $u=t^{-1}$ et $w=vt^{-2}$ pour former un rev\^etement~$\cE^0$ de~$\cX^0$ de degr\'e~2. Notons $\rho^0 \colon \cE^0 \to \cX^0$ le morphisme obtenu.

Par construction, pour tout $z\in \cZ^0$, $\cE^0_{z}$ est un rev\^etement de $ \cX^0_{z} \simeq \P^1_{\kappa(z)}$ de degr\'e~2 ramifi\'e en $\infty,0,1, \lambda(z)$. Munie du point rationnel situ\'e au-dessus de~$\infty$, $\cE^0_{z}$ devient une courbe elliptique dont on peut calculer explicitement l'image des points de 2-torsion~:
\[\rho^0_{z}(\cE^0_{z}[2]) = \{\infty, 0, 1, \lambda(z)\}.\]

La courbe~$\cE^0$ est une courbe elliptique relative sur~$\cX^0$ et, comme dans la section~\ref{sec:Lattescorps}, la formule 
\[ L^0(t) = \frac{(t^2-\lambda)^2}{4t(t-1)(t-\lambda)}\]
d\'efinit un morphisme fini~$L^{0} \colon \cX^0 \to \cX^0$, polaris\'e de degr\'e~4, qui fait commuter le diagramme
\[\begin{tikzcd}
\cE^0 \ar[r, "{[}2{]}"] \ar[d,"\rho^0"]& \cE^0 \ar[d,"\rho^0"]\\
\cX^0 \ar[r, "L^0"] & \cX^0
\end{tikzcd}.\]

Consid\'erons maintenant l'espace $(\P^1_{\Zud})^4$ avec coordonn\'ees $c_{1},c_{2},c_{3},c_{4}$ (sur $\A^4_{\Zud}$) et son ouvert de Zariski
\[\cZ := \{\forall i\ne j \in \cn{1}{4},\ c_{i} \ne c_{j}\} \subset (\P^1_{\Zud})^4.\]
Consid\'erons le morphisme $\cZ \to \cZ^0$ donn\'e par 
\[(c_{1},c_{2},c_{3},c_{4}) \mapstoo  [c_{1},c_{2},c_{3},c_{4}] = \frac{(c_{3}-c_{1})(c_{4}-c_{2})}{(c_{3}-c_{2})(c_{4}-b_{1})}.\]
Le morphisme obtenu par changement de base
\[\rho^0 \times_{\cZ^0} \cZ \colon \cE^0 \times_{\cZ^0} \cZ \too  \cX^0 \times_{\cZ^0} \cZ \simeq \P^1_{\Zud} \times_{\Zud} \cZ\]
est une courbe elliptique relative dont la fibre au-dessus de $z\in \cZ$ est un morphisme ramifi\'e en $\infty,0,1,[c_{1},c_{2},c_{3},c_{4}](z)$. En effectuant le changement de coordonn\'ee global
\[t \mapstoo \frac{c_{1}(c_{3}-c_{2})t + c_{2}(c_{1}-c_{3})}{(c_{3}-c_{2})t+(c_{1}-c_{3})}\] 
 sur le facteur~$\P^1_{\Zud}$, on obtient un morphisme 
 \[\rho \colon \cE := \cE^0 \times_{\cZ^0} \cZ \too \cX := \P^1_{\Zud} \times_{\Zud}\cZ\] 
 au-dessus de~$\cZ$ qui est une courbe elliptique relative et dont la fibre au-dessus de $z\in \cZ$ est un morphisme ramifi\'e en $c_{1}(z),c_{2}(z),c_{3}(z),c_{4}(z)$.

En tirant en arri\`ere le morphisme~$L^0$, on obtient un morphisme~$L$ fini polaris\'e de degr\'e~4 qui fait commuter le diagramme
\[\begin{tikzcd}
\cE \ar[r, "{[}2{]}"] \ar[d,"\rho"]& \cE \ar[d,"\rho"]\\
\cX \ar[r, "L"] & \cX
\end{tikzcd}.\]

\medbreak

Consid\'erons maintenant deux copies~$\cZ_{a}$ et~$\cZ_{b}$ de~$\cZ$ avec coordonn\'ees respectives $a=(a_{1},a_{2},a_{3},a_{4})$ et $b=(b_{1},b_{2},b_{3},b_{4})$ et posons $\cZ_{ab} := \cZ_{a}\times_{\Zud}\cZ_{b}$. Notons respectivement~$Z_{a}$, $Z_{b}$ et $Z_{ab}$ leur analytifi\'e au-dessus de~$\cU_{2}$, au sens de la section~\ref{sec:Z1N}. Rappelons la notation~\ref{nota:Lattes}. 

\begin{theo}\label{th:Econtinue}
La fonction
\[ \cE_{ab} \colon z\in Z_{ab} \mapstoo \la \mu_{a(z)}, \mu_{b(z)} \ra \in\R\]
est $\log$-flottante et continue.
\end{theo}
\begin{proof}
Posons $X_{ab} := \EP{1}{\Zud} \times_{\Zud} Z_{ab}$. Notons~$\varphi_{a}$ (resp.~$\varphi_{b}$) le tir\'e en arri\`ere sur~$X_{ab}$ du morphisme de Latt\`es sur la droite relative sur~$Z_{a}$ (resp.~$Z_{b}$). D'apr\`es \cite[lemmes~6.23 et 6.24]{DynamiqueI}, les familles de mesures $(\mu_{a(z)})_{z\in Z_{ab}}$ et $(\mu_{b(z)})_{z\in Z_{ab}}$ sont flottantes \`a potentiels $\log$-flottants. Il d\'ecoule alors des d\'efinitions que la fonction~$\cE_{ab}$ est $\log$-flottante. Elle est continue par \cite[corollaire~D]{DynamiqueI}.
\end{proof}

Comme rappel\'e au d\'ebut de la section~\ref{sec:mesuresdef}, l'\'energie mutuelle est invariante par changement de coordonn\'ees. Nous pouvons donc nous contenter d'\'etudier la fonction~$\cE_{ab}$ sur un espace plus petit que~$Z_{ab}$. 

On dispose d'une action du groupe sym\'etrique~$S_{4}$ sur~$\cZ_{a}$ et~$\cZ_{b}$, par permutation des points, et d'une action de $\PGL_{2}$ sur~$\P^1$, correspondant \`a un changement de coordonn\'ees. L'invariance de l'\'energie mutuelle rappel\'ee ci-dessus entra\^ine que la fonction~$\cE_{ab}$ se factorise par le quotient $(S_{4} \times S_{4})\backslash Z_{ab} / \PGL_{2}$, o\`u $\PGL_{2}$ agit diagonalement.

Notons $\cD_{ab}$ la diagonale de~$\cZ_{ab}$ et~$D_{ab}$ son analytifi\'ee. Il nous suffit d'\'etudier l'\'energie mutuelle sur $Z_{ab} - D_{ab}$, et donc $(S_{4} \times S_{4})\backslash (Z_{ab} - D_{ab}) / \PGL_{2}$, la diagonale \'etant stable par les diff\'erentes actions.

Afin d'obtenir une expression plus simple de ce quotient, posons 
\[\cZ'_{a} := \{a_{1}\ne 0, a_{2}\ne 0, a_{3}\ne 0, a_{4}=\infty\} \subset \cZ_{a},\]
\[\cZ'_{b} := \{b_{1}\ne \infty, b_{2}\ne \infty, b_{3}\ne \infty, b_{4}=0\} \subset \cZ_{b}\]
et $\cZ'_{ab} := \cZ'_{a} \times_{\Zud} \cZ'_{b} \subset \cZ_{ab}$. Notons $Z'_{a}$, $Z'_{b}$ et $Z'_{ab}$ les analytifi\'es correspondants.

Le groupe sym\'etrique~$S_{3}$ agit sur~$\cZ'_{a}$ et~$\cZ'_{b}$. Le tore~$T$ de~$\SL_{2}$, contenant les matrices diagonales, agit sur $\G_{m}$ par homoth\'eties, donc sur $\cZ'_{ab}$ \textit{via} l'action diagonale. On obtient alors une identification canonique 
\[ (S_{4} \times S_{4})\backslash (Z_{ab} - D_{ab}) / \PGL_{2} = (S_{3} \times S_{3})\backslash Z'_{ab}  / T.\]

Remarquons finalement que le quotient $\cZ'_{ab}/T$ s'identifie \`a l'ouvert de Zariski~$\cY_{ab}'$ de $\P^5_{\Zud}$ avec coordonn\'ees homog\`enes $[a_{1}\mathbin:a_{2}\mathbin:a_{3}\mathbin: b_{1}\mathbin:b_{2}\mathbin:b_{3}]$ d\'efini par les conditions
\[ \cY_{ab}' := \{ \forall i, a_{i} \ne 0, b_{i}\ne 0,\ \forall i\ne j, a_{i}\ne a_{j}, b_{i} \ne b_{j} \} \subset \P^5_{\Zud}.\]
Notons~$Y'_{ab}$ son analytifi\'e.

Dans la suite de ce texte, et particuli\`erement \`a la section~\ref{sec:minoration}, nous travaillerons avec l'espace~$\cY'_{ab}$. Le raisonnement ci-dessus assure que la fonction $\cE_{ab} \colon Z_{ab} \to \R$ du th\'eor\`eme~\ref{th:Econtinue} induit une fonction sur~$Y'_{ab}$, not\'ee identiquement, qui reste $\log$-flottante et continue. En outre, la restriction de l'\'etude \`a $Y'_{ab}$ ne nuit pas \`a la g\'en\'eralit\'e du propos.

\section{Estimations centrale et globale}\label{sec:fibrecentralehauteurs}

Dans cette section, \'etant donn\'es un espace de Berkovich flottant sur~$\Z$ et une fonction $\log$-flottante, nous expliquons comment obtenir,  \`a partir d'une minoration sur la fibre centrale, une minoration uniforme en presque toute place, ainsi que des minorations plus faibles aux places restantes. Ce r\'esultat nous permettra par la suite d'obtenir des minorations globales par des hauteurs. 

Fixons le cadre. Soit $N\in \Z \setminus \{0\}$. Soit~$\cY$ un sch\'ema localement de type fini sur~$\Z[\frac1N]$, g\'en\'eriquement non vide. Notons $Y := \cY^\an$ son analytifi\'e, au sens de la section~\ref{sec:Z1N}. C'est un espace de Berkovich sur~$\Z$ flottant. L'image du morphisme structural $\pr \colon Y \to \cM(\Z)$ est de la forme $\cU_{N_{Y}}$, o\`u $N_{Y}$ est un multiple de~$N$. 

Rappelons qu'une fonction $f \colon Y \to \R$ est dite $\log$-flottante si
\[ \forall y\in Y, \forall \eps \in \intof{0,1},\ f(y^\eps)=\eps f(y).\]

Nous avons choisi d'\'enoncer les r\'esultats de cette section dans le cadre des espaces de Berkovich sur~$\Z$, afin de ne pas alourdir les notations. Pr\'ecisons cepandant qu'ils restent valables \textit{mutatis mutandis} sur tout anneau d'entiers de corps de nombres.

\begin{lemm}\label{lem:extensionmajoration}
Soit~$V$ une partie flottante de~$Y$ et supposons qu'il existe $N_{V}\in \Z$ tel que $\pr(V)=\cU_{N_{V}}$. Soient $\cE,f_{1},f_{2} \in\cC(V,\R)$ des fonctions log-flottantes continues. Pour $s \in \R$, posons 
\[ V_{s} := \{y \in V : f_{1}(y) = s\} \textrm{ et } V_{> s} :=  \{y \in V : f_{1}(y) > s\}.\]
Supposons qu'il existe $s_{0} \in  \R_{>0}$ 
tel que l'application $\pr_{\vert V_{s_{0}}} \colon V_{s_{0}} \to \cU_{N_{V}}$ soit ferm\'ee et $e_{0} \in  \R$ tel que
\[ \forall y\in V_{s_{0}} \cap \pr^{-1}(a_{0}),\ \cE(y) \ge e_{0} + f_{2}(y).\]

Alors, pour tout $\alpha\in \R_{>0}$,  il existe $t_{\alpha}\in  \R_{>0}$ et un ensemble fini $M_{\alpha} \subset M_{\Q}$ tels que
\[\forall v \in M_{\Q} \setminus M_{\alpha},\ \forall y\in V_{>0} \cap \pr^{-1}(a_{v}),\ \cE(y) \ge\, \frac{e_{0}-\alpha}{s_{0}}\, f_{1}(y)+f_{2}(y)\]
et 
\[\forall v \in M_{\alpha},\ \forall y\in V_{>t_{\alpha}} \cap \pr^{-1}(a_{v}),\ \cE(y) \ge\, \frac{e_{0}-\alpha}{s_{0}}\, f_{1}(y) +f_{2}(y).\]
\end{lemm}
\begin{proof}
Soit $\alpha\in \intoo{0,1}$. 
Posons 
\[W := \{y \in V_{s_{0}} : \cE(y) \ge e_{0}- \alpha + f_{2}(y)\} \]
et
\[W' := \{y \in V_{s_{0}} : \cE(y) \le e_{0}- \alpha + f_{2}(y)\}. \]
L'ensemble~$W'$ est une partie ferm\'ee de~$V_{s_{0}}$ et son image $\pr(W')$ est donc une partie ferm\'ee de~$\cU_{N_{V}}$. Puisque $W'$ ne rencontre pas~$\pr^{-1}(a_{0})$, l'ensemble $B_{\alpha} := \cU_{N_{V}} \setminus \pr(W')$ est un voisinage de~$a_{0}$. Par construction, on a $\pr^{-1}(B_{\alpha})\cap V_{s_{0}} \subset W$.

Il existe une partie finie~$M_{\alpha}$ de~$M_{\Q}$ contenant la place infinie et $\eps_{\alpha} \in \intof{0,1}$ telles que
\[B_{\alpha} \supset \bigcup_{v \in M_{\Q} \setminus M_{\alpha}} \intfo{a_{0},a_{v,\infty}} \cup \bigcup_{v \in M_{\alpha}} \intfo{a_{0},a_{v,\eps_{\alpha}}}. \]

Soient $v \in M_{\Q} \setminus M_{\alpha}$ et $y \in V_{>0} \cap \pr^{-1}(a_{v})$. Posons $\eps := s_{0}/f_{1}(y) \in \R_{>0}$. Puisque~$f_{1}$ est log-flottante, on a $f_{1}(y^\eps) = \eps f_{1}(y) = s_{0}$, donc $y^\eps \in \pr^{-1}(B_{\alpha}) \cap V_{s_{0}} \subset W$. On en d\'eduit que $\cE(y^\eps) \ge e_{0}-\alpha +f_{2}(y^\eps)$ et donc, puisque $\cE$ et~$f_{2}$ sont log-flottantes, que
\[\cE(y) \ge \frac{e_{0}-\alpha}{\eps}+f_{2}(y) = \frac{e_{0}-\alpha}{s_{0}}\, f_{1}(y) +f_{2}(y).\]

Posons $t_{\alpha} := s_{0}/{\eps_{\alpha}} \in \R_{>0}$. Soient $v \in M_{\alpha}$ et $y \in V_{>t_{\alpha}} \cap \pr^{-1}(a_{v})$. Posons $\eps := s_{0}/f_{1}(y) \in \intoo{0,\eps_{\alpha}}$. En raisonnant comme pr\'ec\'edemment, on montre que $y^\eps \in W$, puis que 
\[\cE(y) \ge \frac{e_{0}-\alpha}{s_{0}}\, f_{1}(y) +f_{2}(y).\]
\end{proof}

Introduisons quelques notations qui nous seront utiles pour pr\'esenter une version globale du r\'esultat.

\begin{nota}\label{nota:LQ}
Soit $Q \in \cY(\Qbar)$. Le corps r\'esiduel $\kappa(Q)$ est une extension finie de~$\Q$. Pour tout $v\in M_{\Q}$, l'image~$L_{v}(Q)$ de $\Spec( \kappa(Q) \otimes_{\Q} \Q_{v}) \to \cY_{\Q_{v}}$ est un ensemble fini de points ferm\'es. Il s'identifie \`a un ensemble de points de $\pr^{-1}(a_{v}) \simeq \cY_{\Q_{v}}^\an$.

Pour tout $q \in L_{v}(Q)$, posons 
\[N_{q} := \frac{[\cH(q) \mathbin: \Q_{v}]}{[\kappa(Q) \mathbin: \Q]}.\]  
Posons 
\[L(Q) := \bigcup_{v\in M_{\Q}} L_{v}(Q).\]
Pour toute fonction $\cE \colon Y \to \R_{\ge0}$, posons
\[h_{\cE}(Q) := \sum_{q\in L(Q)\cap V} N_{q} \,\cE(q).\]
\end{nota}

\begin{exem}\label{ex:hauteur}
Consid\'erons le sch\'ema $\G_{\mathrm{m}}$ sur~$\Z$ avec coordonn\'ee~$T$ 
et la fonction
\[\log^+(\abs{T}) \colon \G^\an_{\mathrm{m}} \to \R.\]
Alors, pour tout $x \in \Qbar$, on a
\[h_{\log^+(\abs{T})}(x) = h(x).\]
\end{exem}

\begin{theo}\label{thm:minoration}
Soit~$V$ une partie flottante de~$Y$ et supposons qu'il existe $N_{V}\in \Z$ tel que $\pr(V)=\cU_{N_{V}}$. Soient $\cE, f_{1} \colon V \to \R_{\ge 0}$, $f_{2} \colon V \to \R$ et $g_{1} \colon V \to \R\cup\{-\infty\}$ des fonctions $\log$-flottantes continues. 

Pour $s,t\in \R_{>0}$, posons 
\[V_{s} := \{y\in V : f_{1}(y) \le s\}\]
et
\[W_{t} := \{y\in V : g_{1}(y) \le t f_{2}(y) \}.\]

Supposons que
\begin{enumerate}[i)]
\item pour tous $s,t\in \R_{>0}$, l'application $\pr_{\vert V_{s}\cap W_{t}} \colon  V_{s} \cap W_{t} \to \cU_{N_{V}}$ est propre~;
\item il existe $A\in \R_{>0}$ et une famille presque nulle $(B_{v})_{v\in M_{\Q}}$ d'\'elements de~$\R$ tels que
\[\forall v\in M_{\Q}, \forall y\in V\cap \pr^{-1}(a_{v}),\ g_{1}(y) \le A\, f_{1}(y) + B_{v}~;\]
\item il existe $r_{0}\in \R_{>0}$ et $\tau \colon \R_{>0} \to \R_{>0}$, avec $\lim_{+\infty}\tau = +\infty$, tels que
\[\forall r\ge r_{0}, \forall y\in W_{\tau(r)}\cap \pr^{-1}(a_{0}),\ \cE(y) \ge \frac1r  f_{2}(y).\]
\end{enumerate}
Alors, pour tout $\eps \in \R_{>0}$, il existe $C\in \R_{>0}$ et $D\in \R$ telles que, pour tout $Q \in \cY(\Qbar)$, on ait \footnote{Les deux membres de l'\'egalit\'e peuvent \^etre infinis.}
\begin{align*}
\sum_{q\in L_{f_{1}}(Q)\cap V} N_{q}\,\cE(q) & \ge C\, \Big( \sum_{q\in L_{f_{1}}(Q)\cap V} N_{q} f_{2}(q) - \eps \sum_{q\in L(Q)\cap V} N_{q} f_{1}(q)\Big) +D,
\end{align*}
o\`u $L_{f_{1}}(Q) := \{q\in L(Q) : f_{1}(q)>0\}$.
\end{theo}
\begin{proof}
Pour $r,s \in \R_{>0}$, posons
\[ Z_{r,s} := \{y \in W_{\tau(r)} : f_{1}(y) = s\} \textrm{ et } Z_{r,> s} :=  \{y \in W_{\tau(r)} : f_{1}(y) > s\}.\]

Soit $s\in \R_{>0}$. Soit $r\in \intfo{r_{0},+\infty}$ tel que $\tau(r)\ge \frac{2A}{\eps}$. La partie~$W_{\tau(r)}$ de~$V$ est $\log$-flottante. Il suit de l'hypoth\`ese~i) que l'application $\pr_{\vert Z_{r,s}} \colon Z_{r,s} \to \cU_{N_{V}}$ est propre, donc ferm\'ee. On a donc
\[ \forall y\in Z_{r,s} \cap \pr^{-1}(a_{0}),\ \cE(y) \ge \frac1r\, f_{2}(y).\]
D'apr\`es le lemme~\ref{lem:extensionmajoration} appliqu\'e avec $e_{0}=0$, $f_{1}=f_{1}$, $f_{2}=\frac1r f_{2}$ et $\alpha = \frac{s\eps}{2r}$, il existe $t_{r} \in \R_{>0}$ et un ensemble fini~$M_{r}$ de places de~$\Q$ tels que
\[\forall v \in M_{\Q} \setminus M_{r},\ \forall y\in Z_{r,>0} \cap \pr^{-1}(a_{v}),\ \cE(y) \ge\, -\frac{\eps}{2r}\,f_{1}(y) +\frac{1}{r}\, f_{2}(y)\]
et 
\[\forall v \in M_{r},\ \forall y\in Z_{r,>t_{r}} \cap \pr^{-1}(a_{v}),\ \cE(y) \ge\, -\frac{\eps}{2r}\,f_{1}(y) +\frac{1}{r}\, f_{2}(y).\]
Soit $v\in M_{r}$. Par hypoth\`ese, l'ensemble
\[\{y \in W_{\tau(r)} \cap \pr^{-1}(a_{v}) : f_{1}(y)\le t_{r}\}\]
est compact. On en d\'eduit qu'il existe une constante $D_{v} \in \R$ telle que 
\[\forall y\in Z_{r,>0}\cap \pr^{-1}(a_{v}),\ \cE(y) \ge\,  -\frac{\eps}{2r}\,f_{1}(y) + \frac{1}{r}\, f_{2}(y)+D_{v}.\]

\medbreak

Soit $Q \in \cY(\Qbar)$. Posons $L_{f_{1},r}(Q) := L_{f_{1}}(Q)\cap W_{\tau(r)}$ et $L'_{f_{1},r}(Q) := L_{f_{1}}(Q) \setminus L_{f_{1},r}(Q)$.
Pour tout $q\in L'_{f_{1},r}(Q) \cap V$, on a donc
\[f_{2}(q) \le \frac1{\tau(r)}\, g_{1}(q) \le \frac{A}{\tau(r)}\, f_{1}(q) +\frac{B_{v}}{\tau(r)}.\]

Soit $M'$ une partie finie de~$M_{\Q}$ et posons $A_{M'} := \{ a_{v} : v\in M'\}$. On a alors
\begin{align*}
\sum_{q \in L_{f_{1},r}(Q)\cap V \cap A_{M'}} N_{q}\,\cE(q) & \ge \sum_{q \in L_{f_{1},r}(Q)\cap V\cap A_{M'}} N_{q}\, \Big(\frac1{r}\, f_{2}(q) - \frac{\eps}{2r}\, f_{1}(q)\Big) + \sum_{v\in M_{r}\cap M'} D_{v}\\
& \ge \sum_{q \in L_{f_{1}}(Q)\cap V\cap A_{M'}} N_q\, \Big(\frac1{r}\, f_{2}(q) - \frac{\eps}{2r}\, f_{1}(q)\Big)  + \sum_{v\in M_{r}\cap M'} D_{v}\\ 
&\quad - \sum_{q \in L'_{f_{1},r}(Q)\cap V\cap A_{M'}} N_q \,\frac1{r}\, f_{2}(q)\\
&\ge \frac1r \sum_{q \in L_{f_{1}}(Q)\cap V\cap A_{M'}} N_q f_{2}(q) -\frac\eps{2r} \sum_{q \in L(Q)\cap V\cap A_{M'}} N_q f_{1}(q) + \sum_{v\in M_{r}\cap M'} D_{v}\\
&\quad  - \frac{A}{r\tau(r)} \sum_{q \in L'_{f_{1},r}(Q)\cap V\cap A_{M'}} N_q f_{1}(q) - \frac{1}{r\tau(r)} \sum_{v\in M'} B_{v}\\
&\ge  \frac1{r}\sum_{q \in L_{f_{1}}(Q)\cap V\cap A_{M'}} N_q f_{2}(q) - \frac1r \Big(\frac\eps2 + \frac{A}{\tau(r)}\Big) \sum_{q \in L(Q)\cap V\cap A_{M'}} N_q f_{1}(q) \\  
&\quad + \sum_{v\in M_{r}\cap M'} D_{v} - \frac{1}{r\tau(r)} \sum_{v\in M'} B_{v}.
\end{align*}
Le r\'esultat s'en d\'eduit en passant au supremum sur l'ensemble des parties finies de~$M_{\Q}$.
\end{proof}

Nous pr\'esentons maintenant des versions simplifi\'ees du th\'eor\`eme~\ref{thm:minoration}. Nous utiliserons \`a plusieurs reprises le r\'esultat suivant.

\begin{lemm}\label{lem:hauteurva}
Soit $K$ un corps de nombres. Soit~$n\in \N^*$. Pour tout $u=(u_{1},\dotsc,u_{n}) \in (K^*)^n$, on a 
\[\sum_{v\in M_{K}} N_{v}\, \max_{1\le i\le n}(\abs{\log(\abs{u_{i}})}_{v}) \le (n+1)\, h(u).\]
\end{lemm}
\begin{proof}
Soit $u=(u_{1},\dotsc,u_{n}) \in (K^*)^n$. Pour $v\in M_{K}$ et $i\in \cn{1}{n}$, posons $m_{i,v} := \min(0,\log(\abs{u_{i}}_{v}))$ et $M_{i,v} := \max(0,\log(\abs{u_{i}}_{v}))$. Pour $v\in M_{K}$, posons $m_{v} := \min_{1\le i\le n} (m_{i,v})$ et $M_{v} := \max_{1\le i\le n} (M_{i,v})$.

On a
\begin{align*}
0 &= \sum_{v\in M_{K}} \sum_{i=1}^n N_{v} \log(\abs{u_{i}}_{v})\\
&=  \sum_{v\in M_{K}} \sum_{i=1}^n N_{v}(m_{i,v}+M_{i,v})\\
&\le \sum_{v\in M_{K}} N_{v}(m_{v} + nM_{v}),
\end{align*}
d'o\`u
\[- \sum_{v\in M_{K}} N_{v} m_{v} \le n  \sum_{v\in M_{K}} N_{v} M_{v}.\]
On en d\'eduit que
\begin{align*}
\sum_{v\in M_{K}} N_{v} \max_{1\le i\le n}(\abs{\log(\abs{u_{i}})}_{v}) & = \sum_{v\in M_{K}} N_{v} \max_{1\le i\le n} (\max(M_{v,i},-m_{v,i}))\\
&= \sum_{v\in M_{K}} N_{v}  \max(M_{v},-m_{v})\\
&\le \sum_{v\in M_{K}} N_{v}  (M_{v}-m_{v})\\
&\le (n+1) \sum_{v\in M_{K}} N_{v}  M_{v}.
\end{align*}
\end{proof}

Le r\'esultat qui suit propose des versions simplifi\'ees de la conclusion du th\'eor\`eme~\ref{thm:minoration}, sous certaines hypoth\`eses.

\begin{coro}\label{cor:minorationF1}
Pla\c cons-nous dans le cadre du th\'eor\`eme~\ref{thm:minoration} avec $V=Y$.

\begin{enumerate}[i)]
\item Supposons qu'il existe une famille finie $F_{1} = (F_{1,i})_{i\in I_{1}}$ d'\'el\'ements inversibles de~$\cO(\cY)$ telle que 
\[f_{1} = \max_{i\in I_{1}} (\abs{\log(\abs{F_{1,i}})}).\]
Alors, pour tout $\eps \in \R_{>0}$, il existe $C\in \R_{>0}$ et $D\in \R$ telles que, pour tout $Q \in \cY(\Qbar)$, on ait
\begin{align*}
h_{\cE}(Q) \ge C\, \Big( \sum_{q\in L_{f_{1}}(Q)} N_{q} f_{2}(q) - \eps \, h(F_{1}(Q))
\Big) +D.
\end{align*}

\item Supposons, en outre, qu'il existe une famille finie $F_{2} = (F_{2,i})_{i\in I_{2}}$ d'\'el\'ements inversibles de~$\cO(\cY)$ telle que 
\[f_{2} = \max_{i\in I_{2}} (\abs{\log(\abs{F_{2,i}})})\]
et que, pour tout $y\in Y$, on ait 
\[\big(\forall i\in I_{1}, \abs{F_{1,i}(y)}=1\big) \implies \big(\forall i\in I_{2}, \abs{F_{2,i}(y)}=1\big).\]
Alors, pour tout $\eps \in \R_{>0}$, il existe $C\in \R_{>0}$ et $D\in \R$ telles que, pour tout $Q \in \cY(\Qbar)$, on ait
\begin{align*}
h_{\cE}(Q) \ge C\, \big( h(F_{2}(Q)) - \eps\, h(F_{1}(Q))\big) +D.
\end{align*}

\item Supposons, en outre, que $F_{1}=F_{2}$. Alors, il existe $C\in \R_{>0}$ et $D\in \R$ telles que, pour tout $Q \in \cY(\Qbar)$, on ait
\begin{align*}
h_{\cE}(Q) \ge C\, h(F_{2}(Q))+D.
\end{align*}
\end{enumerate}
\qed
\end{coro}

\begin{rema}
L'\'enonc\'e du corollaire~\ref{cor:minorationF1} reste valable en rempla\c cant partout $h_{\cE}(Q)$ par $\sum_{q \in L_{f_{1}}(Q)} N_{q}\,\cE(q)$.
\end{rema}

Int\'eressons-nous maintenant \`a des simplications des hypoth\`eses du th\'eor\`eme~\ref{thm:minoration}, rendues possibles par le r\'esultat suivant.

\begin{lemm}\label{lem:logPj}
Soit $K$ un corps de nombres. Soient $n\in \N^*$ et $P\in K[T_{1}^{\pm1},\dotsc,T_{n}^{\pm1}]$. Il existe $A \in \R_{>0}$ et une famille presque nulle $(B_{v})_{v\in M_{K}}$ d'\'elements de~$\R_{\ge 0}$ telles que, pour toute $v\in M_{K}$ et tout $x=(x_{1},\dotsc,x_{n}) \in (K_{v}^\ast)^n$, on ait 
\[\log(\abs{P(x)}_{v}) \le A\, \max_{1\le i\le n}(\abs{\log(\abs{x_{i}}_{v})})+ B_{v}.\]
\end{lemm}
\begin{proof}
Pour tout $h=(h_{1},\dotsc,h_{n})\in \Z^n$, posons $\abs{h}:=\sum_{i=1}^n \abs{h_{i}}$ et $T^h := \prod_{i=1}^n T_{i}^{h_{i}}$. \'Ecrivons $P$ sous la forme  
\[P = \sum_{h\in H} a_{h} T^h,\]
o\`u $H$ est une partie finie de $\Z^n$. On peut supposer que $P\ne 0$, $H\ne \emptyset$ et, pour tout $h\in H$, $a_{h}\ne 0$.

Soit $v\in M_{K}$. Soit $x=(x_{1},\dotsc,x_{n}) \in (K_{v}^\ast)^n$. Posons $M_{x} := \max_{1\le i\le n}(\abs{x_{i}}_{v},\abs{x_{i}}_{v}^{-1})$ et $m_{x}:= \log(M_{x}) = \max_{1\le i\le n}(\abs{\log(\abs{x_{i}}_{v})})$. 

$\bullet$ Supposons que $v\in M_{f,K}$. On a alors
\begin{align*}
\abs{P(x)}_{v} & \le \max_{h=(h_{1},\dotsc,h_{n})\in E}\Big(\abs{a_{h}}_{v} \prod_{i=1}^n \abs{x_{i}}_{v}^{h_{i}}\Big) \\
&\le \max_{h\in H}(\abs{a_{h}}_{v} M_{x}^{\abs{h}}),
\end{align*}
d'o\`u 
\begin{align*}
\log(\abs{P(x)}_{v}) &\le \max_{h\in H} (\log(\abs{a_{h}}_{v}) + \abs{h} m_{x})\\
&\le \max_{h\in H} (\abs{h})\, m_{x} + \max_{h\in H} (\log(\abs{a_{h}}_{v})).
\end{align*}

$\bullet$ Supposons que $v\in M_{\infty,K}$. On a alors
\[
\abs{P(x)}_{v} \le \sharp H \max_{h=(h_{1},\dotsc,h_{n})\in H}\Big(\abs{a_{h}}_{v} \prod_{i=1}^n \abs{x_{i}}_{v}^{h_{i}}\Big),
\]
d'o\`u il d\'ecoule que 
\[\log(\abs{P(x)}_{v})\le \max_{h\in H} (\abs{h})\, m_{x} + \max_{h\in H} (\log(\abs{a_{h}}_{v})) + \log(\sharp H).\]
\end{proof}

Le r\'esultat de simplification promis en d\'ecoule.

\begin{lemm}\label{lem:majorationg1}
Supposons qu'il existe une famille finie $F_{1} = (F_{1,i})_{i\in I_{1}}$ d'\'el\'ements inversibles  de~$\cO(\cY)$ telle que 
\[f_{1} = \max_{i\in I_{1}} (\abs{\log(\abs{F_{1,i}})}),\]
ainsi qu'une famille finie $(P_{j})_{j\in J_{1}}$ d'\'el\'ements de $\Q[T_{i},T_{i}^{-1}, i\in I_{1}]$ telle que
\[g_{1} \le \max_{j\in J_{1}}(\log(\abs{P_{j}(F_{1})})).\]
Alors l'hypoth\`ese~ii) figurant dans l'\'enonc\'e du th\'eor\`eme~\ref{thm:minoration} est satisfaite.
\qed
\end{lemm}

Pour conclure cette section, nous proposons un exemple d'utilisation du lemme~\ref{lem:extensionmajoration} pour obtenir une majoration.

\begin{prop}\label{prop:majoration}
Soit $V$ une partie flottante de~$Y$ et supposons qu'il existe $N_{V}\in \Z$ tel que $\pr(V)=\cU_{N_{V}}$. Soit $\cE \colon V \to \R_{\ge 0}$ une fonction $\log$-flottante continue. Soit $(F_{i})_{i\in I}$ une famille finie d'\'el\'ements inversibles de~$\cO(V)$. Pour $r \in \R_{>0}$, posons 
\[ W_{r} := \{y \in V : \max_{i\in I} (\abs{F_{i}(y)},\abs{F_{i}^{-1}(y)}) = r\}\]
et
\[W'_{r} := \{ y\in V : \min_{i\in I}(\abs{F_{i}(y)}) < r^{-1} \textrm{ ou } \max_{i\in I}(\abs{F_{i}(y)}) > r\}.\]
Supposons qu'il existe $r_{0} \in \R_{>1}$ tel que l'application $\pr_{\vert W_{r_{0}}} \colon W_{r_{0}} \to \cU_{N_{V}}$ soit ferm\'ee et $e_{0} \in \R_{\ge 0}$ tel que 
\[ \forall y\in W_{r_{0}} \cap \pr^{-1}(a_{0}),\ \cE(y) \le e_{0}.\]
Alors, pour tout $\alpha \in \R_{>0}$, il existe $s_{\alpha} \in \R_{>1}$ et un ensemble fini~$M_{\alpha}$ de places de~$\Q$ tels que
\[\forall v \in M_{\Q} \setminus M_{\alpha},\ \forall y\in W'_{1} \cap \pr^{-1}(a_{v}),\ \cE(y) \le \frac{e_{0} + \alpha}{\log(r_{0})}\, \max_{i\in I}\big(\abs{\log(\abs{F_{i}(y)})}\big)\]
et 
\[\forall v \in M_{\alpha},\ \forall y\in W'_{s_{\alpha}} \cap \pr^{-1}(a_{v}),\ \cE(y) \le\, \frac{e_{0} + \alpha}{\log(r_{0})}\, \max_{i\in I}\big(\abs{\log(\abs{F_{i}(y)})}\big).\]
\end{prop}
\begin{proof}
Soit~$\alpha\in \R_{> 0}$. Le r\'esultat d\'ecoule du lemme~\ref{lem:extensionmajoration} appliqu\'e avec $\cE = -\cE$, $f_{1} = \max_{i\in I} (\abs{\labs{F_{i}}})$, $f_{2}=0$, $V = V$, $s_{0} = \log(r_{0})$ et $e_{0} = - e_{0}$, $\alpha = \alpha$. 
\end{proof}

Une fois ce r\'esultat appliqu\'e, pour obtenir une majoration globale de la fonction~$\cE$, il reste \`a traiter le cas de \og bonne r\'eduction \fg{} ($\abs{F_{i}}=1$ pour tout $i$), qui est g\'en\'eralement plus simple, et celui des compacts $\{ s_{\alpha} \le \abs{F_{i}} \le s_{\alpha}^{-1}\} \cap \pr^{-1}(a_{v})$, pour $v$ appartenant \`a l'ensemble fini de places $M_{\alpha}$, o\`u l'on sait qu'une majoration existe \emph{a priori}.

\section{Minoration de l'\'energie mutuelle globale}\label{sec:minoration}

Dans cette section, nous combinons les r\'esultats de la section~\ref{sec:fibrecentralehauteurs} aux estimations sur la fibre centrale de la section~\ref{sec:segmentsegment} pour obtenir une minoration de l'\'energie mutuelle globale. Nous reprenons les notations de la section~\ref{sec:Lattesfamille}. 

Rappelons que, d'apr\`es la proposition~\ref{prop:mugamma}, si $z \in \pr^{-1}(a_{0})$, les mesures $\mu_{a(z)}$ et $\mu_{b(z)}$ s'identifient \`a des mesures de Lebesgue sur des segments. En particulier, on peut calculer leur \'energie mutuelle en utilisant les r\'esultats de la section~\ref{sec:segmentsegment}.

\medbreak

Le but de cette section est de d\'emontrer une minoration de l'\'energie mutuelle globale~$h_{\cE_{ab}}$ sur~$\cY_{ab}'(\Qbar)$. Dans un premier temps, nous minorerons~$h_{\cE_{ab}}$ par une hauteur.

\begin{nota}
Pour $Q\in \cY_{ab}'(\Qbar)$, posons
\begin{align*} 
h_{ab}([Q]) &:= h([a_{1}(Q) \mathbin: a_{2}(Q) \mathbin: a_{3}(Q) \mathbin: b_{1}(Q) \mathbin: b_{2}(Q) \mathbin: b_{3}(Q)])\\
& = \sum_{q\in L(Q)} \max_{1\le i,j\le 3}(\labs{a_{i}(q)},\labs{b_{j}(q)})\\
&=  \sum_{q\in L(Q)} \max_{1\le i,j\le 3}\Big(0,\biglabs{\frac{a_{i}(q)}{b_{1}(q)}},\biglabs{\frac{b_{j}(q)}{b_{1}(q)}}\Big)
\end{align*}
\end{nota}

Pour des raisons pratiques, posons 
\[(u_{1},u_{2},u_{3},u_{4},u_{5},u_{6}) := (a_{1},a_{2},a_{3},b_{1},b_{2},b_{3}).\]

\begin{lemm}\label{lem:hf1}
Consid\'erons la fonction de~$Y_{ab}'$ dans~$\R$ d\'efinie par
\[ f_{1} := \max_{i\ne j} \Big( \bigabs{\biglabs{\frac{u_{i}}{u_{j}}}}, \bigabs{\biglabs{\frac{u_{i}}{u_{j}}-1}}\Big).\]
Il existe $C_{1},D_{1} \in \R_{>0}$ telles que, pour tout $Q\in \cY_{ab}'(\Qbar)$, on ait
\[h_{f_{1}}(Q) \le C_{1}\, h_{ab}([Q]) + D_{1}.\]
\end{lemm}
\begin{proof}
Soit $Q\in \cY_{ab}'(\Qbar)$. D'apr\`es le lemme~\ref{lem:hauteurva}, on a 
\[ h_{f_{1}}(Q) \le 61 \sum_{q\in L(Q)} N_q \max_{i\ne j}\big(0, \biglabs{\frac{u_{i}(q)}{u_{j}(q)}},\biglabs{\frac{u_{i}(q)}{u_{j}(q)}-1}\big).\]
Soient $i,j \in \cn{1}{6}$ avec $i\ne j$. Pour tout $q\in Y'_{ab}$, on a 
\begin{align*}
\biglabs{\frac{u_{i}(q)}{u_{j}(q)}} &\le  \biglabs{\frac{u_{i}(q)}{b_{1}(q)}} -  \biglabs{\frac{u_{j}(q)}{b_{1}(q)}}\\ 
& \le 2 \max\big( \bigabs{\biglabs{\frac{u_{i}(q)}{b_{1}(q)}}},\bigabs{\biglabs{\frac{u_{j}(q)}{b_{1}(q)}}}\big).
\end{align*}
Pour tout $q\in Y_{ab}'$, on a \'egalement
\begin{align*}
\biglabs{\frac{u_{i}(q)}{u_{j}(q)}-1} &\le \log\big(\bigabs{\frac{u_{i}(q)}{u_{j}(q)}}+1\big)\\ 
& \le \log\Big( 2 \max\big(\bigabs{\frac{u_{i}(q)}{u_{j}(q)}},0 \big) \Big)\\
& \le \log(2) + 2 \max\Big( \Bigabs{\biglabs{\frac{u_{i}(q)}{b_{1}(q)}}},\Bigabs{\biglabs{\frac{u_{j}(q)}{b_{1}(q)}}}\Big).
\end{align*}
On en d\'eduit que 
\[ h_{f_{1}}(Q) \le 61\log(2) +  122 \sum_{q\in L(Q)} N_q \max_{1\le i\le 6}\big(\bigabs{\biglabs{\frac{u_{i}(q)}{b_{1}(q)}}}\big)\]
et on conclut par le lemme~\ref{lem:hauteurva}.
\end{proof}

\begin{lemm}\label{lem:E0smax}
Soit $\rho \in \intfo{1,+\infty}$. Consid\'erons les fonctions de~$Y_{ab}'$ dans~$\R$ d\'efinies par
\[f_{2} := \max\big( 0 ,  \smax_{1\le i,j\le 3}\big(\biglabs{\frac{a_{i}}{b_{j}}}\big) \big)\]
et 
\[g_{1} := \max_{i,j,k \textrm{ distincts}} \Big(\biglabs{\frac{a_{i}-a_{j}}{a_{i}-a_{k}}}, \biglabs{\frac{b_{j}}{b_{k}}\, \frac{b_{k}-b_{i}}{b_{j}-b_{i}}}\Big).\]
Soit $q\in Y_{ab}' \cap \pi^{-1}(a_{0})$ tel que $g_{1}(q) \le \rho f_{2}(q)$. Alors, on a 
\[\cE_{ab}(q) \ge \frac{1}{48\rho^2} \, f_{2}(q).\]
\end{lemm}
\begin{proof}
On peut supposer que $\abs{a_{3}(q)}\ge \abs{a_{2}(q)}\ge \abs{a_{1}(q)}$ et $\abs{b_{3}(q)}\ge\abs{b_{2}(q)}\ge \abs{b_{1}(q)}$. On a alors $f_{2}(q) = \max\big(0,\log(\abs{\frac{a_{3}(q)}{b_{2}(q)}}),\log(\abs{\frac{a_{2}(q)}{b_{1}(q)}})\big)$. 

Rappelons que l'\'energie commune est invariante par changement de coordonn\'ees sur~$\P^1$. Le changement de~$t$ en~$t^{-1}$ a pour effet de remplacer $(a_{1}(q),a_{2}(q),a_{3}(q),\infty)$ par $(b_{3}^{-1}(q), b_{2}^{-1}(q), b_{1}^{-1}(q),0)$ et $(b_{1}(q),b_{2}(q),b_{3}(q),0)$ par $(a_{3}^{-1}(q), a_{2}^{-1}(q), a_{1}^{-1}(q),0)$. En particulier, $\frac{a_{3}(q)}{b_{2}(q)}$ et $\frac{a_{2}(q)}{b_{1}(q)}$ sont \'echang\'es par cette op\'eration. On peut donc supposer que $f_{2}(q) = \max(0,\log(\abs{\frac{a_{3}(q)}{b_{2}(q)}}))$. 

On peut \'egalement supposer que $f_{2}(q) >0$, autrement dit, que $\abs{a_{3}(q)} > \abs{b_{2}(q)}$.

\medbreak

Avec la notation~\ref{nota:Igamma}, consid\'erons les segments de~$\EP{1}{k}$ suivants~: $I_{1} := I_{(a_{1}(q),a_{2}(q),a_{3}(q),\infty)}$ et $I_{2} := I_{(b_{1}(q),b_{2}(q),b_{3}(q),0)}$. L'interpr\'etation de leur longueur en termes de birapport (\cf~lemme~\ref{lem:longueurIgamma}) assure que l'on a
\[ \max(\ell(I_{1}),\ell(I_{2})) \le g_{1}(q) \le \rho f_{2}(q).\]

\medbreak

$\bullet$ Supposons que $I_{1}\cap I_{2} = \emptyset$.

Notons $d_{12}$ la distance de~$I_{1}$ \`a~$I_{2}$. On a alors $\ell_{1} + \ell_{2} + d_{12} \ge \log(\abs{\frac{a_{3}(q)}{b_{2}(q)}})$, \cf~figure~\ref{fig:energiedisjoints}. En utilisant le th\'eor\`eme~\ref{thm:Eintersectionvide}, on en d\'eduit que
\[\cE_{ab}(q) \ge \frac1{24}\,  \biglabs{\frac{a_{3}(q)}{b_{2}(q)}} \ge \frac{1}{48\rho^2} \, f_{2}(q).\]

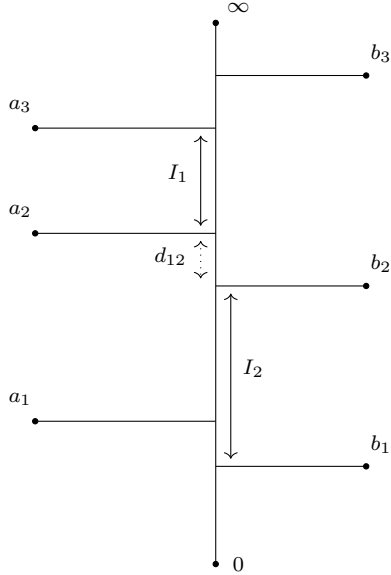
\begin{figure}[!h]
\centering
\begin{tikzpicture}
\draw  (-0.4,4)-- (-0.4,-3.2);
\draw  (-0.4,1.2)-- (-2.8,1.2);
\draw  (-0.4,2.6)-- (-2.8,2.6);
\draw  (-0.4,-1.3)-- (-2.8,-1.3);
\draw  (-0.4,0.5)-- (1.6,0.5);
\draw  (-0.4,3.3)-- (1.6,3.3);
\draw  (-0.4,-1.9)-- (1.6,-1.9);
\draw[<->]  (-0.6,2.5)-- (-0.6,1.3);
\draw[<->]  (-0.2,0.4)-- (-0.2,-1.8);
\draw[<->,dotted]  (-0.6,1.1)-- (-0.6,0.6);
\begin{scriptsize}
\draw [fill=black] (-0.4,4) circle (1pt);
\draw[color=black] (-0.1,4.2) node {$\infty$};
\draw [fill=black] (-0.4,-3.2) circle (1pt);
\draw[color=black] (-0.1,-3.2) node {0};
\draw [fill=black] (-2.8,1.2) circle (1pt);
\draw[color=black] (-3,1.5) node {$a_2$};
\draw [fill=black] (-2.8,2.6) circle (1pt);
\draw[color=black] (-3,2.9) node {$a_3$};
\draw [fill=black] (-2.8,-1.3) circle (1pt);
\draw[color=black] (-3,-1) node {$a_1$};
\draw [fill=black] (1.6,0.5) circle (1pt);
\draw[color=black] (1.8,0.8) node {$b_2$};
\draw [fill=black] (1.6,3.3) circle (1pt);
\draw[color=black] (1.8,3.6) node {$b_3$};
\draw [fill=black] (1.6,-1.9) circle (1pt);
\draw[color=black] (1.8,-1.6) node {$b_1$};
\draw[color=black] (-0.9,2) node {$I_1$};
\draw[color=black] (0.1,-0.6) node {$I_2$};
\draw[color=black] (-1,.9) node {$d_{12}$};
\end{scriptsize}\end{tikzpicture}
\caption{Segments~$I_{1}$ et $I_{2}$ disjoints.}\label{fig:energiedisjoints}
\end{figure}

\medbreak

$\bullet$ Supposons que $I_{1}\cap I_{2} \ne \emptyset$.

Alors, $I_{1} \setminus I_{2}$ contient un intervalle~$J$ de longueur sup\'erieure ou \'egale \`a $\log(\abs{\frac{a_{3}(q)}{b_{2}(q)}})$, \cf~figure~\ref{fig:energieintersection}. D'apr\`es le corollaire~\ref{cor:Eqrho} appliqu\'e avec $\lambda =  f_{2}(q)$, on a donc 
\[\cE_{ab}(q)\ge  \frac{1}{48\rho^2} \, f_{2}(q) .\]

\begin{figure}[!h]
\centering
\begin{tikzpicture}
\draw  (-0.4,4)-- (-0.4,-3.2);
\draw  (-0.4,-.2)-- (-2.8,-.2);
\draw  (-0.4,2.6)-- (-2.8,2.6);
\draw  (-0.4,-1.3)-- (-2.8,-1.3);
\draw  (-0.4,0.5)-- (1.6,0.5);
\draw  (-0.4,3.3)-- (1.6,3.3);
\draw  (-0.4,-1.9)-- (1.6,-1.9);
\draw[<->]  (-0.6,2.5)-- (-0.6,-.1);
\draw[<->]  (-0.2,0.4)-- (-0.2,-1.8);
\draw[<->]  (-0.2,0.6)-- (-0.2,2.5);
\begin{scriptsize}
\draw [fill=black] (-0.4,4) circle (1pt);
\draw[color=black] (-0.1,4.1) node {$\infty$};
\draw [fill=black] (-0.4,-3.2) circle (1pt);
\draw[color=black] (-0.1,-3.2) node {0};
\draw [fill=black] (-2.8,-.2) circle (1pt);
\draw[color=black] (-3,0.1) node {$a_2$};
\draw [fill=black] (-2.8,2.6) circle (1pt);
\draw[color=black] (-3,2.9) node {$a_3$};
\draw [fill=black] (-2.8,-1.3) circle (1pt);
\draw[color=black] (-3,-1) node {$a_1$};
\draw [fill=black] (1.6,0.5) circle (1pt);
\draw[color=black] (1.8,.8) node {$b_2$};
\draw [fill=black] (1.6,3.3) circle (1pt);
\draw[color=black] (1.8,3.6) node {$b_3$};
\draw [fill=black] (1.6,-1.9) circle (1pt);
\draw[color=black] (1.8,-1.6) node {$b_1$};
\draw[color=black] (-0.9,1.2) node {$I_1$};
\draw[color=black] (0.1,-0.7) node {$I_2$};
\draw[color=black] (0.1,1.7) node {$J$};
\end{scriptsize}\end{tikzpicture}
\caption{Segments~$I_{1}$ et $I_{2}$ qui se rencontrent.}\label{fig:energieintersection}
\end{figure}
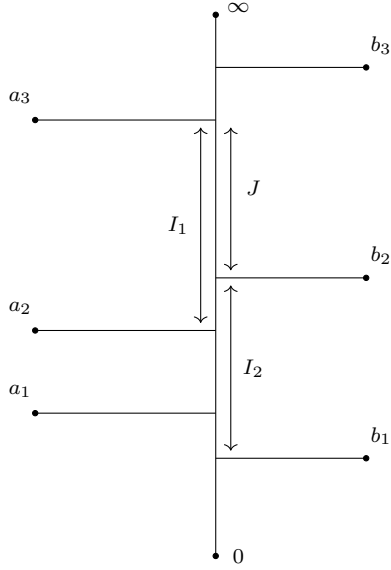
\end{proof}

\begin{lemm}\label{lem:hf2}
Avec les notations du lemme~\ref{lem:E0smax}, il existe $C_{2} \in \R_{>0}$ telle que, pour tout $Q\in \cY_{ab}'(\Qbar)$, on ait
\[h_{f_{2}}(Q) \ge C_{2}\, h_{ab}([Q]).\]
\end{lemm}
\begin{proof}
Posons $P := \disp\prod_{i=1}^3 \dfrac{a_{i}}{b_{i}} \colon Y_{ab}' \to \R$. Pour tout $q\in Y_{ab}'$, on a
\begin{align*}
\smax_{1\le i,j\le 3}\big(\log\big(\bigabs{\frac{a_{i}(q)}{b_{j}(q)}}\big)\big) &= \smax_{1\le i\le 3}(\labs{a_{i}(q)}+\labs{b_{j}(q)^{-1}})\\
& \ge  \frac12 \Big( \max_{1\le i\le 3}(\labs{a_{i}(q)}) + \smax_{1\le i\le 3}(\labs{b_{j}(q)^{-1}}) \\ 
& \quad +  \smax_{1\le i\le 3}(\labs{a_{i}(q)}) + \max_{1\le i\le 3}(\labs{b_{j}(q)^{-1}})\Big)\\
& \ge \frac12 \Big( \sum_{i=1}^3 \labs{a_{i}(q)}  -\min_{1\le i\le 3}(\labs{a_{i}(q)}) \\ 
& \quad + \sum_{j=1}^3 \labs{b_{j}(q)^{-1}} - \min_{1\le j\le 3}(\labs{b_{j}(q)^{-1}})\Big)\\
& \ge  \frac12 \Big( \sum_{i=1}^3 \biglabs{\frac{a_{i}(q)}{b_{i}(q)}} \\ 
& \quad + \max_{1\le i\le 3}(\labs{a_{i}(q)^{-1}}) +  \max_{1\le j\le 3}(\labs{b_{j}(q)})\Big)\\
& \ge \frac12 \max_{1\le i,j\le 3}\big(\biglabs{P(q)\, \frac{b_{j}(q)}{a_{i}(q)}}\big).
\end{align*}

Soit $Q\in \cY_{ab}'(\Qbar)$. \`A l'aide de l'in\'egalit\'e pr\'ec\'edente et du lemme~\ref{lem:hauteurva}, on obtient
\[h_{f_{2}}(Q) \ge \frac1{20}\, \sum_{q\in L(Q)} N_q \max_{1\le i,j\le 3}\big(\bigabs{\biglabs{P(q)\, \frac{b_{j}(q)}{a_{i}(q)}}}\big).\]

Soient $i,j \in \cn{1}{6}$. Pour tout $q\in Y_{ab}'$, on a
\[ \Bigabs{\biglabs{\frac{b_{i}(q)}{b_{j}(q)}}} \le \Bigabs{\biglabs{P(q) \frac{b_{i}(q)}{a_{1}(q)}}} + \Bigabs{\biglabs{P(q) \frac{b_{j}(q)}{a_{1}(q)}}}\]
et 
\[ \Bigabs{\biglabs{\frac{a_{i}(q)}{a_{j}(q)}}} \le \Bigabs{\biglabs{P(q) \frac{b_{1}(q)}{a_{i}(q)}}} + \Bigabs{\biglabs{P(q) \frac{b_{1}(q)}{a_{j}(q)}}}.\]

Soit $k\in \cn{1}{3}$. Pour tout $q\in Y_{ab}'$, il existe $u,v \in \cn{1}{3}$ tels que
\[\biglabs{\frac{a_{k}(q)}{b_{1}(q)}} = \smax_{1\le i,j\le 3}\big(\biglabs{\frac{a_{i}(q)}{b_{j}(q)}}\big) + \biglabs{\frac{a_{k}(q)}{a_{u}(q)}} + \biglabs{\frac{b_{v}(q)}{b_{1}(q)}}.\]
Par cons\'equent, pour tout $q\in Y_{ab}'$, on a 
\[\biglabs{\frac{a_{k}(q)}{b_{1}(q)}} \le f_{2}(q) + \max_{1\le i,j\le 3} \big(\bigabs{\biglabs{\frac{a_{i}(q)}{a_{j}(q)}}}\big) + \max_{1\le i,j\le 3} \big(\bigabs{\biglabs{\frac{b_{i}(q)}{b_{j}(q)}}}\big)\]

En combinant les in\'egalit\'es pr\'ec\'edentes, on obtient finalement
\begin{align*}
h_{ab}([Q]) & = \sum_{q\in L(Q)} N_q \max_{1\le i,j\le 3}\Big(0, \biglabs{\frac{a_{i}(q)}{b_{1}(q)}}, \biglabs{\frac{b_{j}(q)}{b_{1}(q)}} \Big)\\
&\le h_{f_{2}}(Q) + 4 \sum_{q\in L(Q)} N_q \max_{1\le i,j\le 3}\Big(\Bigabs{\biglabs{P(q)\, \frac{b_{j}(q)}{a_{i}(q)}}}\Big)\\
&\le 81 \, h_{f_{2}}(Q).
\end{align*}
\end{proof}

\begin{theo}\label{th:hEhab}
Il existe $C_{0} \in \R_{>0}$ et $D_{0} \in \R$ telles que, pour tout $Q\in \cY_{ab}'(\Qbar)$, on ait
\[h_{\cE_{ab}}(Q) \ge C_{0}\, h_{ab}([Q]) +D_{0}.\]
\end{theo}
\begin{proof}
D\'efinissons les fonctions~$f_{1}$, $f_{2}$ et~$g_{1}$ comme dans les lemmes~\ref{lem:hf1} et~\ref{lem:E0smax}. Montrons que les trois conditions du th\'eor\`eme~\ref{thm:minoration} sont satisfaites pour $V=Y_{ab}'$.

i) Soient $s,t\in \R_{>0}$. On a 
\[V_{s} := \Big\{ e^{-s} \le \bigabs{\frac{u_{i}}{u_{j}}},\bigabs{\frac{u_{i}-u_{j}}{u_{j}}} \le e^s,\ 1\le i,j\le 6 \Big\} \subset Y_{ab}',\]
donc l'application $V_{s} \to \cU_{2}$ est propre. La condition s'en d\'eduit. 

ii) Pour tous $i,j,k$ distincts, on a
\[\biglabs{\frac{a_{i}-a_{j}}{a_{i}-a_{k}}} =  \biglabs{1 - \frac{a_{j}}{a_{i}}} - \biglabs{1-\frac{a_{k}}{a_{i}}}\]
et 
\[\biglabs{\frac{b_{j}}{b_{k}}\, \frac{b_{k}-b_{i}}{b_{j}-b_{i}}} =  \biglabs{\frac{b_{j}}{b_{k}}} + \biglabs{\frac{b_{k}}{b_{i}}-1} - \biglabs{\frac{b_{j}}{b_{i}}-1}.\]
La condition d\'ecoule alors du lemme~\ref{lem:majorationg1}.

iii) Cette condition d\'ecoule du lemme~\ref{lem:E0smax}.

On peut donc appliquer le th\'eor\`eme~\ref{thm:minoration}. Il assure que, pour tout $\eps\in \R_{>0}$, il existe $C_{0,\eps} \in \R_{>0}$ et $D_{0,\eps} \in \R$ telles que, pour tout $Q\in \cY_{ab}'(\Qbar)$, on ait
\[h_{\cE_{ab}}(Q) \ge C_{0,\eps}\, \Big(\sum_{q \in L_{f_{1}}(Q)} N_q f_{2}(q) - \eps \sum_{q \in L(Q)} N_q f_{1}(q)\Big) +D_{0,\eps},\]
d'o\`u 
\[ h_{\cE_{ab}}(Q) \ge C_{0,\eps} \, (h_{f_{2}}(Q) - \eps h_{f_{1}}(Q)) + D_{0,\eps},\]
en utilisant le fait que $f_{2}$ est nulle d\`es que~$f_{1}$ l'est. On conclut \`a l'aide des lemmes~\ref{lem:hf1} et~\ref{lem:hf2} en choisissant~$\eps$ assez petit.
\end{proof}

\begin{rema}
Signalons que tous les arguments pr\'esent\'es jusqu'ici s'appliquent encore si l'on travaille non plus sur un corps de nombres, mais sur un corps de fonctions de caract\'eristique diff\'erente de~2. En particulier, le th\'eor\`eme~\ref{th:hEhab} reste valable dans ce cadre, avec la m\^eme preuve. Plus pr\'ecis\'ement, on peut remplacer~$\Q$ par $k(T)$, o\`u $k$ est un corps de caract\'eristique diff\'erente de~2, et appliquer la th\'eorie en prenant comme espace de base non plus l'ouvert~$\cU_{2}$ de~$\cM(\Z)$, mais la droite projective de Berkovich~$\EP{1}{k}$.
\end{rema}

\begin{rema}\label{rem:PcPd}
Pour illustrer la port\'ee des m\'ethodes employ\'ees, montrons comment elles s'adaptent pour obtenir une minoration de l'\'energie mutuelle de deux syst\`emes dynamiques de la forme $P_{c} \colon z \mapsto z^2+c$ avec $c\in \Qbar$. 

Ce r\'esultat est bien plus simple \`a \'etablir que son analogue dans le cas des morphismes de Latt\`es. Expliquons, tout d'abord, comment adapter les constructions de la section~\ref{sec:Lattes}. Au lieu de l'espace de modules~$\cZ$, consid\'erons la droite affine $\wti\cZ:=\A^1_{\Z}$ avec coordonn\'ee~$u$. Notons~$\wti\cX$ la droite projective relative au-dessus de~$\wti\cZ$ et~$\Pi$  l'endomorphisme de~$\wti\cX$ polaris\'e de degr\'e~2 d\'efini par $z\mapsto z^2+u$, o\`u $z$ est une coordonn\'ee affine relative sur~$\wti\cX$.

Pour chaque point $z\in \wti\cZ^\an$, la fibre de~$\wti\cX^\an$ au-dessus de~$z$ s'identifie \`a $\EP{1}{\cH(z)}$, et l'on dispose d'une mesure d'\'equilibre $\tilde\mu_{u(z)}$ associ\'ee \`a l'endomorphisme induit par~$\Pi$.

Consid\'erons maintenant deux copies~$\wti\cZ_{c}$ et~$\wti\cZ_{d}$ de~$\wti\cZ$ avec coordonn\'ees respectives~$c$ et~$d$ et posons $\wti\cZ_{cd} := \wti\cZ_{c} \times_{\Z} \wti\cZ_{d}$. D'apr\`es \cite[Corollaire~D]{DynamiqueI}, la fonction
\[ \wti\cE_{cd} \colon z\in \wti\cZ^\an_{cd} \mapstoo \la \wti\mu_{c(z)}, \wti\mu_{d(z)} \ra \in\R\]
est $\log$-flottante et continue. Notons que, par d\'efinition, pour tout $Q \in \wti\cZ_{cd}(\Qbar)$, on a
\[ \la P_{c(Q)},P_{d(Q)} \ra = h_{\wti\cE_{cd}}(Q).\]

Finalement, notons $\wti\cY_{cd}$ l'ouvert de Zariski de $\wti\cZ_{cd}$ d\'efini par la condition $c\ne d$.

\medbreak

Suivant le principe expos\'e \`a la section~5, pour minorer la fonction $h_{\wti\cE_{cd}}$ sur $\wti\cY_{cd}(\Qbar)$, nous allons chercher \`a minorer $\wti\cE_{cd}$ sur la fibre centrale. Pour ce faire, reprenons les estimations \'etablies par L.~DeMarco, H.~Krieger et H.~Ye dans \cite[Theorem~5.1]{DKY2}~: pour tout point $z$ de $\wti\cY_{cd}$ au-dessus de~$\va_{0}$, on a
\[\la \wti\mu_{c(z)}, \wti\mu_{d(z)} \ra \ge \frac14\, \log^+(\abs{c(z)-d(z)}).\]
De plus, si $\abs{c(z)}=\abs{d(z)}>1$ et $\abs{c(z)-d(z)} > \abs{c(z)}^{-1/2}$, alors on a
\[ \la \wti\mu_{c(z)}, \wti\mu_{d(z)} \ra \ge \frac1{16}\, \log{\abs{c(z)}}.\]
(Ces minorations sont \'enonc\'ees pour le compl\'et\'e d'un corps de nombres en une place finie ne divisant pas~2, mais les preuves restent valables pour tout corps valu\'e complet ultram\'etrique de caract\'eristique r\'esiduelle diff\'erente de~2, et en particulier sur la fibre centrale.) 

On en d\'eduit ais\'ement que, pour tout point $z$ de $\wti\cY_{cd}$ au-dessus de~$\va_{0}$, on a
\[\la \wti\mu_{c(z)}, \wti\mu_{d(z)} \ra \ge \frac1{12} \Big(  \log(\abs{c(z)-d(z)}) +\frac12\,  \log^+(\max(\abs{c(z)},\abs{d(z)})) \Big).\]

Nous pouvons maintenant appliquer le th\'eor\`eme~\ref{thm:minoration} avec les fonctions 
\[\begin{cases}
f_{1} = \max( \labs{c},\labs{d},\abs{\labs{c-d}}),\\ 
f_{2} = \frac1{12} \, \big(  \log(\abs{c-d}) +\frac12\,  \log^+(\max(\abs{c},\abs{d})) \big),\\
g_{1} = -\infty.
\end{cases}\] 
En remarquant que, pour tout $z \in \wti\cY_{cd}$, la condition $f_{1}(z)=0$ entra\^ine $\wti\cE_{cd}(z)=0$, on en d\'eduit qu'il existe $C'\in \R_{>0}$ et $D'\in \R$ telles que, pour tout $Q \in \wti\cY_{cd}(\Qbar)$, on a 
\[h_{\wti\cE_{cd}}(Q) \ge C' \, h(c(Q),d(Q)) + D',\]
o\`u $h(\wc,\wc)$ d\'esigne la hauteur logarithmique usuelle sur $\Qbar^2$. Autrement dit, pour tous $c \ne d \in \Qbar$, on a 
\[ \la P_{c},P_{d} \ra \ge C' \, h(c,d) + D'.\]
Nous obtenons ainsi une d\'emonstration rapide de \cite[Theorem~1.7]{DKY2} (sans valeur explicite pour les constantes). Nous renvoyons \`a~\textit{ibid.} pour l'application \`a la majoration du nombre de points pr\'ep\'eriodiques communs \`a~$P_{c}$ et~$P_{d}$, uniform\'ement en~$c$ et~$d$.
\end{rema}

Revenons maintenant au cas des morphismes de Latt\`es. 

\begin{theo}\label{th:hEm0}
Il existe $m_{0} \in \R_{>0}$ telle que, pour tout $Q \in \cY_{ab}'(\Qbar)$, on ait $h_{\cE_{ab}}(Q) \ge m_{0}$.
\end{theo}
\begin{proof}
On suit fid\`element \cite[section~6.3]{DKY}. Consid\'erons la fonction de~$Y_{ab}'$ dans~$\R$ d\'efinie par
\[ m := \min_{i\ne j} \Big( \bigabs{\frac{u_{i}}{u_{j}}}, \bigabs{\frac{u_{i} - u_{j}}{u_{j}}}, \bigabs{\frac{u_{i}}{u_{i} - u_{j}}} \Big).\]

Rappelons que la fonction $\cE_{ab} \colon Y_{ab}' \to \R_{\ge 0}$ est continue. En outre, d'apr\`es la remarque~\ref{rem:mugammaC} et les propri\'et\'es de l'\'energie mutuelle des mesures, elle est strictement positive en tout point de~$\pi^{-1}(a_{\infty})$.

Pour tout $u \in \R_{>0}$, 
\[K_{u} := \{ y \in \pi^{-1}(a_{\infty}) : m(y) \ge u\}\]
est une partie compacte non vide de~$\pi^{-1}(a_{\infty})$. Notons $e_{0} >0$ le minimum de~$\cE_{ab}$ sur~$K_{1}$.

Pour tout $e \in \intoo{0,e_{0}}$, posons 
\[ s(e) := \sup(\{m(y) : y \in \pi^{-1}(a_{\infty}), \cE_{ab}(y) \le e\}) \in \intff{0,1}.\]
La fonction $s \colon \intoo{0,e_{0}} \to  \intff{0,1}$ est d\'ecroissante.

Posons $u_{0} := \inf(\{s(e) : e\in \intoo{0,e_{0}} \})$. Supposons, par l'absurde, que $u_{0}>0$. Soit $u_{1}\in \intoo{0,u_{0}}$ et notons $m>0$ le minimum de~$\cE_{ab}$ sur la partie compacte~$K_{u_{1}}$ de~$\pi^{-1}(a_{\infty})$. Alors, pour tout $e \in \intoo{0,m}$, on a $s(e) \le u_{1}$ et on aboutit \`a une contradiction. Nous avons ainsi montr\'e que $u_{0}=0$. On en d\'eduit que $\lim_{e\to 0} s(e) = 0$.

\medbreak

Supposons, par l'absurde, qu'il existe une suite $(Q_{n})_{n\in \N^*}$ de $\cY_{ab}'(\Qbar)$ telle que 
\[ \lim_{n\to +\infty} h_{\cE_{ab}}(Q_{n}) = 0.\]
Quitte \`a extraire une sous-suite, on peut supposer que, pour tout $n\in \N^*$, on a $h_{\cE_{ab}}(Q_{n}) \le \frac1{2n}$. Pour tout $n\in \N^*$, soit~$K_{n}$ un corps de nombres tel que $Q_{n} \in \cY'_{ab}(K_{n})$.

Soit $n\in \N^*$. Puisque $\cE_{ab}$ est \`a valeurs positives, on a
\[\sum_{q \in L_{\infty}(Q_{n})} N_q\, \cE_{ab}(q) \le \frac1{2n}.\]
Posons $M_{n} := \{ q \in L_{\infty}(Q_{n}) : \cE_{ab}(q) \le \frac1 n\}$. Pour tout $q\in M_{n}$, on a alors 
\[m(q) =  \min_{i\ne j} \Big( \bigabs{\frac{u_{i}(q)}{u_{j}(q)}}, \bigabs{\frac{u_{i}(q) - u_{j}(q)}{u_{j}(q)}}, \bigabs{\frac{u_{i}(q)}{u_{i}(q) - u_{j}(q)}} \Big) \le s\big(\frac1n\big).\]
Remarquons que l'on a
\[ \sum_{q \in L_{\infty}(Q_{n}) \setminus M_{n}} N_q\, \cE_{ab}(q) > \frac1n  \sum_{q \in L_{\infty}(Q_{n}) \setminus M_{n}} N_q,\]
d'o\`u  
\[  \sum_{q \in M_{n}} N_q =  1 - \sum_{q \in L_{\infty}(Q_{n}) \setminus M_{n}} N_q > \frac12.\]
Il existe un sous-ensemble~$M'_{n}$ de~$M_{n}$ et une fonction $v \colon Y_{ab}' \to \R_{>0}$ de la forme $\frac{u_{i}}{u_{j}}$ ou $\big(\frac{u_{i}}{u_{j}}-1)^{\pm1}$ avec  $i,j\in \cn{1}{6}$, $i \ne j$, tels que
\[ \forall q\in M'_{n},\ \abs{v(q)} \le s\big(\frac1n\big)\]
et 
\[  \sum_{q \in M'_{n}} N_q  > \frac1{18}.\]
Quitte \`a extraire une sous-suite, on peut supposer que~$v$ est ind\'ependante de~$n$.

On a alors 
\[  \sum_{q\in L(Q_{n})} N_q \, \abs{\labs{v(q)}} > - \frac{1}{18} \,\log\big(s\big(\frac 1 n\big)\big) \xrightarrow[n\to +\infty]{} +\infty,\]
d'o\`u, d'apr\`es le lemme~\ref{lem:hf1}, $\lim_{n\to+\infty}h_{ab}([Q_{n}]) =+\infty$, ce qui contredit le th\'eor\`eme~\ref{th:hEhab}.
\end{proof}

\section{Retour aux courbes elliptiques}\label{sec:courbeselliptiques}

Dans cette section finale, nous appliquons les in\'egalit\'es obtenues jusqu'ici pour \'etudier les images de points de torsion ou de petite hauteur sur des courbes elliptiques.

Soient $K$ un corps de nombres, $E$ une courbe elliptique sur~$K$ telle que $E[2]\subset E(K)$ et $\pi \colon E \to \P^1_{K}$ une projection standard. Rappelons que l'on peut associer \`a ces donn\'ees une mesure ad\'elique~$\mu_{(E,\pi)}$ (\cf~notation~\ref{nota:muQ}) et une hauteur~$h_{\mu_{(E,\pi)}}$ (\cf~notation~\ref{nota:[F]}), que nous noterons \'egalement~$h_{(E,\pi)}$, par souci de simplicit\'e. Cette derni\`ere n'est autre que la hauteur dynamique associ\'ee au morphisme de Latt\`es correspondant \`a~$(E,\pi)$. Rappelons \'egalement que, d'apr\`es la remarque~\ref{rem:energieL}, le corps de nombres sur laquelle est d\'efinie la courbe elliptique est sans importance pour les calculs d'\'energie. 

Commen\c cons par r\'e\'ecrire le th\'eor\`eme~\ref{th:hEm0} dans le cadre des courbes elliptiques.

\begin{theo}\label{th:minorationenergie}
Il existe $m_{0} \in \R_{>0}$ telle que, pour toutes courbes elliptiques munies de projections standards $(E_{1},\pi_{1})$ et $(E_{2},\pi_{2})$ sur~$\Qbar$ avec $\pi_{1}(E_{1}[2]) \ne \pi_{2}(E_{2}[2])$, on ait 
\[ \la \mu_{(E_{1},\pi_{1})},\mu_{(E_{2},\pi_{2})} \ra \ge m_{0}.\]
\end{theo}
\begin{proof}
Le raisonnement de la fin de la section~\ref{sec:Lattesfamille} permet de d\'eduire cet \'enonc\'e de celui du th\'eor\`eme~\ref{th:hEm0}.

De fa\c con plus pr\'ecise, consid\'erons $(E_{1},\pi_{1})$ et $(E_{2},\pi_{2})$ comme dans l'\'enonc\'e. D'apr\`es les rappels effectu\'es \`a la section~\ref{sec:mesuresdef}, la quantit\'es $\la\mu_{P_{1}},\mu_{P_{2}}\ra$ est invariante par changement de coordonn\'ees sur~$\P^1$. On peut donc supposer que $\pi_{1}(E_{1}) = \{a_{1},a_{2},a_{3},\infty\}$ et $\pi_{2}(E_{2}) = \{b_{1},b_{2},b_{3},0\}$, avec $a_{1},a_{2},a_{3}, b_{1},b_{2},b_{3} \in \Qbar^\ast$. Posons $P_{1} := (a_{1},a_{2},a_{3},\infty) \in \P^1(\Qbar)^4$, $P_{2} := (b_{1},b_{2},b_{3},0) \in \Qbar^4$ et $Q_{12} := (a_{1} , a_{2} , a_{3} , b_{1} , b_{2} , b_{3}) \in (\Qbar^*)^5$. On a alors $\mu_{(E_{1},\pi_{1})} = \mu_{P_{1}}$ et $\mu_{(E_{2},\pi_{2})} = \mu_{P_{2}}$, $Q_{12}$ d\'efinit un point de $\cY'_{ab}(\Qbar)$ et on a $h_{\cE_{ab}}(Q_{12}) = \la \mu_{P_{1}},\mu_{P_{2}}\ra$. Le r\'esultat d\'ecoule maintenant du th\'eor\`eme~\ref{th:hEm0}.
\end{proof}

Pour conclure, nous utilisons un r\'esultat de majoration de l'\'energie mutuelle en termes de hauteurs d\^u \`a Thomas Gauthier (\cf~\cite[Theorem~B]{GauthierHoelderEstimates}). Il vaut pour des familles quelconques de syst\`emes dynamiques sur~$\P^1$. Nous nous contentons de l'\'enoncer ici dans le cas particulier qui nous int\'eresse.\footnote{Le r\'esultat original de Th. Gauthier est r\'edig\'e avec~$\Q$ pour corps de base, mais vaut en r\'ealit\'e sur tout corps de nombres. D'autre part, il fait intervenir une certaine hauteur sur $\Rat_{4} \times \Rat_{4}$. Notre \'enonc\'e s'en d\'eduit en tirant en arri\`ere par un morphisme $\cY'_{ab} \to \Rat_{4} \times \Rat_{4}$.}

\begin{theo}\label{th:CsurS}
Il existe une constante $C_{1}\in \R_{>0}$ satisfaisant la propri\'et\'e suivante. Soit $K$ un corps de nombres. Soient $(a_{1},a_{2},a_{3}), (b_{1},b_{2},b_{3}) \in (K^*)^3$ tels que, pour tous $j\ne j'$, $a_{j} \ne a_{j'}$ et $b_{j} \ne b_{j'}$. Posons $P_{1} := (a_{1},a_{2},a_{3},\infty) \in (\P^1(K))^4$, $P_{2} := (b_{1},b_{2},b_{3},0) \in K^4$ et $Q_{12} := (a_{1} , a_{2} , a_{3} , b_{1} , b_{2} , b_{3}) \in (K^*)^5$. Pour toute partie non vide $F$ de~$K$ et tout $\delta \in \intoo{0,1}$, on a 
\[\la \mu_{P_{1}},\mu_{P_{2}}\ra \le h_{\mu_{P_{1}}}(F) + h_{\mu_{P_{2}}}(F)+ C_{1} \, \Big(\delta - \frac{\log(\delta)}{\sharp F}\Big)\,  (h([Q_{12}])+1).\]
\qed
\end{theo}

\begin{theo}\label{th:BFTfinal}
Il existe $m,M \in \R_{>0}$ telles que, pour tout corps de nombres~$K$ et toutes courbes elliptiques~$E_{1}$ et~$E_{2}$ sur $K$ avec $E_{1}[2]\subset E_{1}(K)$ et $E_{1}[2]\subset E_{1}(K)$ et toutes projections standards $\pi_{1} \colon E_{1}\to \P^1_{K}$ et $\pi_{2} \colon E_{2}\to \P^1_{K}$ avec $\pi_{1}(E_{1}[2]) \ne \pi_{2}(E_{2}[2])$, on ait
\[\sharp \big\{ x \in \P^1(\bar K) : h_{(E_{1},\pi_{1})}(x) + h_{(E_{2},\pi_{2})}(x) \le m\big\} \le M.\]
En particulier, on a
\[ \sharp \big( \pi_{1}(E_{1}[\infty]) \cap \pi_{2}(E_{2}[\infty])\big) \le M.  \]
\end{theo}
\begin{proof}
Reprenons les notations des th\'eor\`emes~\ref{th:hEhab}, \ref{th:minorationenergie} et~\ref{th:CsurS}. 
Posons $m := \frac{m_{0}}2$. 

Soit $K$ un corps de nombres et soient $(E_{1},\pi_{1})$ et $(E_{2},\pi_{2})$ comme dans l'\'enonc\'e. Quitte \`a effectuer un changement de coordonn\'ees sur~$\P^1$, on peut supposer que $\pi_{1}(E_{1}) = \{a_{1},a_{2},a_{3},\infty\}$ et $\pi_{2}(E_{2}) = \{b_{1},b_{2},b_{3},0\}$, avec $a_{1},a_{2},a_{3}, b_{1},b_{2},b_{3} \in K^\ast$. Posons $P_{1} := (a_{1},a_{2},a_{3},\infty) \in (\P^1(K))^4$, $P_{2} := (b_{1},b_{2},b_{3},0) \in K^4$ et $Q_{12} := (a_{1} , a_{2} , a_{3} , b_{1} , b_{2} , b_{3}) \in (K^*)^5$. 
On a alors $\mu_{(E_{1},\pi_{1})} = \mu_{P_{1}}$ et $\mu_{(E_{2},\pi_{2})} = \mu_{P_{2}}$. En outre, $Q_{12}$ d\'efinit un point de $\cY'_{ab}(K)$ et on a
\[ h_{\cE_{ab}}(Q_{12}) = \la \mu_{P_{1}},\mu_{P_{2}}\ra \textrm{ et } h_{ab}([Q_{12}])=h([Q_{12}]).\]

Soit $F$ une partie finie non vide de
\[\{ x \in \bar K : h_{(E_{1},\pi_{1})}(x) + h_{(E_{2},\pi_{2})}(x) \le m\}.\] 
Il suffit de trouver $M\in \R_{>0}$, ind\'ependante des donn\'ees telle que $\sharp F\le M$. Soit $L$ une extension de~$K$ contenant~$F$. D'apr\`es la remarque~\ref{rem:energieL}, on peut calculer toutes les \'energies apr\`es extension des scalaires \`a~$L$.

Il existe $h_{0} \in \R_{>0}$ tel que, pour tout $h\ge h_{0}$, on ait $\frac{C_{0}\, h + D_{0}-m}{h+1} \ge \frac12\, C_{0}$.

$\bullet$ Supposons que $h([Q_{12}]) \ge h_{0}$.

D'apr\`es le th\'eor\`eme~\ref{th:hEhab} et le th\'eor\`eme~\ref{th:CsurS} appliqu\'e avec $\delta_{0} := \min\big(\frac12,\frac{C_{0}}{4C_{1}}\big)$, on a 
\[ C_{0} \, h([Q_{12}]) +D_{0} \le m + C_{1}\, \Big(\delta_{0} - \frac{\log(\delta_{0})}{\sharp F}\Big)\,  (h([Q_{12}])+1),\]
d'o\`u
\[ \sharp F \le \frac{ - C_{1} \log(\delta_{0})}{\frac{C_{0}\, h([Q_{12}]) + D_{0}-m}{h([Q_{12}])+1} - C_{1}\delta_{0}} \le \frac{ - 4C_{1} \log(\delta_{0})}{ C_{0}}.\]

$\bullet$ Supposons que $h([Q_{12}]) \le h_{0}$.

D'apr\`es le th\'eor\`eme~\ref{th:minorationenergie} et le th\'eor\`eme~\ref{th:CsurS} appliqu\'e avec $\delta_{1} := \min\big(\frac12,\frac{m_{0}}{4(h_{0}+1)C_{1}}\big)$, on a  
\[ m_{0} \le m + C_{1}\, \Big(\delta_{1} - \frac{\log(\delta_{1})}{\sharp F}\Big)\,  (h([Q_{12}])+1),\]
d'o\`u
\[\frac{- C_{1} \log(\delta_{1})}{\sharp F} \ge \frac{m_{0} - m}{h([Q_{12}])+1} - C_{1} \delta_{1} \ge  \frac{m_{0}}{4(h_{0}+1)}\]
et
\[\sharp F \le  \frac{- 4(h_{0}+1) C_{1} \log(\delta_{1})}{m_{0}}.\]

\medskip

Le r\'esultat s'en d\'eduit.
\end{proof}

Par un argument standard, la derni\`ere partie du r\'esultat s'\'etend \`a tout corps de caract\'eristique nulle.

\begin{coro}\label{cor:BFTfinal}
Il existe $M \in \R_{>0}$ telle que, pour tout corps~$L$ de caract\'eristique~0 et toutes courbes elliptiques munies de projections standards $(E_{1},\pi_{1})$ et $(E_{2},\pi_{2})$ sur~$L$ avec $\pi_{1}(E_{1}[2]) \ne \pi_{2}(E_{2}[2])$, on ait
\[ \sharp \big( \pi_{1}(E_{1}[\infty]) \cap \pi_{2}(E_{2}[\infty])\big) \le M.\]
\end{coro}
\begin{proof}
Consid\'erons la constante~$M$ donn\'ee par le th\'eor\`eme~\ref{th:BFTfinal}. On construit un sch\'ema~$S_{M+1}$ de type fini sur~$\Q$ dont les points sur toute extension~$L$ de~
$\Q$ correspondent \`a la donn\'ee de~:
\begin{enumerate}[i)] 
\item une paire de courbes elliptiques avec projections standards $(E_{1},\pi_{1})$ et $(E_{2},\pi_{2})$ sur~$L$ avec $E_{1}[2]\subset E_{1}(L)$, $E_{2}[2]\subset E_{2}(L)$ et $\pi_{1}(E_{1}[2]) \ne \pi_{2}(E_{2}[2])$ ;
\item $M+1$ points distincts de~$\P^1(L)$ contenus dans $\pi_{1}(E_{1}[\infty]) \cap \pi_{2}(E_{2}[\infty])$.
\end{enumerate}
On peut r\'ealiser~$S_{M+1}$ comme un sous-sch\'ema de $\big(\cZ_{ab}\otimes_{\Zud}\Q \big) \times_{\Q} (\P^1_{\Q})^{M+1}$.

Le th\'eor\`eme~\ref{th:BFTfinal} assure que $S_{M+1}$ ne contient aucun point d\'efini sur~$\Qbar$. D'apr\`es le Nullstellensatz, il est donc vide. Le r\'esultat s'ensuit.
\end{proof}

\nocite{}
\bibliographystyle{alpha}
\bibliography{../../biblio}
\end{document}